\documentclass[[review,onefignum,onetabnum]{siamart220329}

\DeclareGraphicsExtensions{.eps,.ps,.jpg,.bmp,.png,.tif}
	
	\usepackage{lineno,hyperref}
	\usepackage{lipsum}
	\usepackage{threeparttable}
	\usepackage{stmaryrd}
	\usepackage{amsfonts}
	\usepackage{graphicx}
	\usepackage{epstopdf}
	\usepackage{algorithmic}
	\usepackage{multirow}
	\usepackage{cases}
	\usepackage{amssymb}
	\usepackage{url}
	\usepackage{subfigure}
	\usepackage{mathrsfs}
	\usepackage{cases}
	\usepackage{booktabs}
	\usepackage{amsopn}
	
	\ifpdf
	\DeclareGraphicsExtensions{.eps,.pdf,.png,.jpg}
	\else
	\DeclareGraphicsExtensions{.eps}
	\fi
	
	\newsiamremark{remark}{Remark}
	\newsiamremark{hypothesis}{Hypothesis}
	\crefname{hypothesis}{Hypothesis}{Hypotheses}
	\newsiamthm{claim}{Claim}
	\newtheorem{aassumption}{Assumption}

	\headers{An iPMM method for nonconvex composite image reconstruction}{B. J. Li, S. H. Pan and T. Y. Zeng}
	
	\title{An inexact proximal MM method for a class of nonconvex composite image reconstruction models\thanks{text
			\funding{This work is funded by the National Natural Science Foundation of China under project No.12371299.}}}
	
	\author{Bujin Li\thanks{School of Mathematics, South China University of Technology, Guangzhou, People's Republic of Chin (\email{mabjli@mail.scut.edu.cn}).}
		\and Shaohua Pan\thanks{School of Mathematics, South China University of Technology, Guangzhou, People's Republic of China (\email{shhpan@scut.edu.cn}).} \and Tieyong Zeng\thanks{Department of Mathematics, The Chinese University of Hong Kong, Hong Kong Special Administrative Region of China, People's Republic of China (\email{zeng@math.cuhk.edu.hk}).}}
	
\headers{An inexact proximal MM method for image reconstruction}{Bujin Li, Shaohua Pan and Tieyong Zeng}

\title{An inexact proximal MM method for a class of nonconvex composite image reconstruction models\thanks{\funding{This work is funded by the National Natural Science Foundation of China under project No.12371299.}}}

\ifpdf
\hypersetup{
  pdftitle={An inexact proximal MM method for a class of nonconvex composite image reconstruction models},
  pdfauthor={ Bujin Li, Shaohua Pan and Tieyong Zeng}
}
\fi

\begin{document}
	
\modulolinenumbers

\maketitle

\begin{abstract}
 This paper concerns a class of composite image reconstruction models for impluse noise removal, which is rather general and covers existing convex and nonconvex models proposed for reconstructing images with impluse noise. For this nonconvex and nonsmooth optimization problem, we propose a proximal majorization-minimization (MM) algorithm with an implementable inexactness criterion by seeking in each step an inexact minimizer of a strongly convex majorization of the objective function, and establish the convergence of the iterate sequence under the KL assumption on the constructed potential function. This inexact proximal MM method is applied to handle gray image deblurring and color image inpainting problems, for which the associated potential function satisfy the required KL assumption. Numerical comparisons with two state-of-art solvers for image deblurring and inpainting tasks validate the efficiency of the proposed algorithm and models.  
\end{abstract}
\section{Introduction}\label{sec1}
 Image reconstruction problems typically seek a clear image (or matrix) from a corrupted image $b\in\mathbb{R}^{m\times n}$ by minimizing an objective function $f$ of the form
 \begin{equation} \label{p1}
  f(x)=g(x;b)+h(x),
 \end{equation}
 where $g(\cdot;b)\!:\mathbb{R}^{m\times n}\to\mathbb{R}$ is a data-fidelity term, and $h\!:\mathbb{R}^{m\times n}\to\overline{\mathbb{R}}:=(-\infty,\infty]$ is a regularization term. Here, $\mathbb{R}^{m\times n}$ represents the space consisting of all $m \times n$ real matrices, equipped with the trace inner product $\langle\cdot,\cdot\rangle$ and its induced Frobenius norm $\|\cdot\|_F$. In many applications, the relation between the observation matrix $b\in\mathbb{R}^{m\times n}$ and the true (or natural) image $\widehat{x}$ is modeled by
 \begin{equation}
  b=\mathcal{A}(\widehat{x})+\omega,
 \end{equation}
 where $\mathcal{A}:\mathbb{R}^{m\times n}\to\mathbb{R}^{m\times n}$ is a linear operator (see \cite{Buades05,Huang13,Cai09}), and $\omega\in\mathbb{R}^{m\times n}$ is a noise matrix. The relevant data-fidelity term generally takes the following form 
 \begin{equation}
  g(x;b)=\sum_{i=1}^{m}\sum_{j=1}^{n}\theta([|\mathcal{A}(x)-b|]_{i,j}),
 \end{equation}
 where $\theta\!:\mathbb{R}\to\mathbb{R}_+$ is a lower semicontinuous (lsc) function, and for a matrix $y\in\mathbb{R}^{m\times n}$, $|y|$ represents an $m\times n$ matrix with entries being $|y_{ij}|$. From a statistical point of view, $g$ is usually a measure for both the distortion and the noise intervening between the true data $\widehat{x}$ and the observation data $b$. Different data-fidelity term $g$ corresponds to different type of noises; for example,   $g(x;b)=\|\mathcal{A}(x)-b\|^2_F$ induced by $\theta(t)=t^2$ is an appropriate measure for the Gaussian noise \cite{CaiOsher09,Afonso10,Wang08}; while $g(x;b)=\|\mathcal{A}(x)-b\|_0$ induced by $\theta(t)={\rm sign}(t)$ is a suitable measure to treat the impluse noise \cite{Chartrand09,Yuan17}, where for a matrix $z\in\mathbb{R}^{m\times n}$, $\|z\|_0$ denotes the number of nonzero entries of $z$. Due to the combinatorial property of the zero norm $\|\cdot\|_0$, a popular surrogate for such a data-fidelity term is the $\ell_1$-norm data-fidelity term $\|\mathcal{A}(x)-b\|_1$ induced by $\theta(t)=t$ (see \cite{Chan05,Guo09,Bai16}). This convex data-fidelity term exhibits some advantage over impulse noise removal, but it also causes a biased solution of the original problem. To overcome this shortcoming, in recent years a few nonconvex data-fidelity terms were proposed to approximate the zero norm; for example, the SCAD \cite{Gu17}, logarithmic function \cite{Zhang20}, Geman function \cite{Zhang17}, and the family of functions proposed in \cite{Lanza16,Zeng19}. 

 Regularization term $h$ is often used to represent prior information on the natural image $\widehat{x}$. There are two kinds of common regularization terms in image reconstruction. One is the Total Variation (TV) term, first introduced by Rudin-Osher and Fatemi \cite{Rudinleonid92} for removing Gaussian noise. This regularization term is usually used to characterize gradient sparsity of images, and has been widely utilized in the recent two decades for image denoising and image deblurring (see, e.g., \cite{Yang09,Zhang17,Chambolle17}). The other is the low rank (LR) term, which is extensively employed to characterize low rank or approximate low rank property of images. Some common functions to promote low rank property include the nuclear norm \cite{Toh10,Zhou15}, the weighted nuclear norm \cite{GuXie17,GuZhang14}, and the truncated nuclear norm \cite{Hu13,Saeedi22}. For early image reconstruction models, their objective functions generally involve a single regularization term. To achieve better reconstruction effect, it is natural to consider a combination of different regularization terms; for example, the regularization term of TV plus indicator of a box constraint set was used in image deblurring \cite{Beck09,Zuo11,Chan13}. In recent years, hybrid regularization terms involving TV and LR terms have been widely used for image restoration, such as MR image super-resolution \cite{Shi15}, dynamic MRI \cite{Liu20}, remote sensing image destriping \cite{Yang20}, hyperspectral image restoration \cite{He16,Hu21,Cai22}, image denoising \cite{Chen18,Shi22}, image deblurring \cite{SunLi22,Wang23} and image completion \cite{Qiu23,Qiu21}, etc.

 Motivated by the better performance of nonconvex data-fidelity terms to promote sparsity and the better effect of hybrid regularization terms to express prior information, we focus on the following nonconvex and nonsmooth image reconstruction model
 \begin{equation}\label{model}
 	\min_{x\in\mathbb{X}}\Theta(x):=\vartheta(|F(x)|) +\nu\psi(T(x)) +\lambda h(x),
 \end{equation}
 where $\mathbb{X}$ represents a finite-dimensional real vector space endowed with the inner product $\langle\cdot,\cdot\rangle$ and its induced norm $\interleave\cdot\interleave$, $\nu\ge 0$ and $\lambda\ge0$ are the regularization parameters, and the functions $h,F,T$ and $\vartheta,\psi$ satisfy the following basic assumption: 
 \begin{aassumption}\label{ass0}
  \begin{enumerate}
  \item[(i)] $h\!:\mathbb{X}\to\overline{\mathbb{R}}$ is a proper lsc convex function that is  strictly continuous relative to its domain ${\rm dom}\,h$ and is lower bounded;
 		
  \item[(ii)] the mappings $F\!:\mathcal{O}\to\mathbb{R}^{m\times n}$ and $T\!:\mathcal{O}\to\mathbb{R}^{p\times q}$ are continuously differentiable on an open set $\mathcal{O}$ containing ${\rm dom}\,h$; 
 		
  \item[(iii)] $\psi\!:\mathbb{R}^{p\times q}\to\mathbb{R}$ is a lower bounded convex function, and $\vartheta\!:\mathbb{R}^{m\times n}\to\mathbb{R}_{+}$ is a function induced by some $\theta\in\!\mathscr{L}$ via $\vartheta(z):=\sum_{i=1}^{m}\sum_{j=1}^{n}\theta(z_{ij})$ for $z\in\mathbb{R}^{m\times n}$, where $\mathscr{L}$ denotes a family of concave functions $\theta$ satisfying conditions 
 		\begin{itemize}
 		 \item[\bf(C.1)] $\theta$ is differentiable on an open interval $I\!\supset\mathbb{R}_{+}$ with a strictly continuous derivative $\theta'$;
 			
 		\item[\bf(C.2)] $\theta'(t)\ge 0$ for all $t\ge 0$ and $\theta(0)=0$.  
 		\end{itemize}
 \end{enumerate}
 \end{aassumption}

  Model \eqref{model} is very general and covers the case that $h$ is weakly convex. Indeed, if $h$ is weakly convex with modulus $\rho>0$, then $\widehat{h}=h+\frac{\rho}{2}\|\cdot\|^2$ is a convex function satisfying Assumption \ref{ass0} (i), and $\Theta$ has the form $\Theta(\cdot)=\vartheta(|F(\cdot)|) +\nu\widehat{\psi}(\widehat{T}(\cdot)) +\lambda\widehat{h}(\cdot)$ for some function $\widehat{\psi}$ and $\widehat{T}$ satisfying Assumption \ref{ass0} (ii)-(iii). Assumption \ref{ass0} (i) also allows $h$ to be an indicator function of some closed convex set $\Delta\subset\mathbb{X}$. The function family $\mathscr{L}$ contains many common concave functions to promote sparsity; see the examples in Table \ref{tab0} below. It is easy to check that the TVL1 model \cite{Guo09}, the NonconvexTV model \cite{Zhang17} and the TVLog model \cite{Zhang20} are all a special case of model \eqref{model}.
  \renewcommand{\arraystretch}{2}
  \begin{table}[h] 	
  \caption{Common examples in the function family $\mathscr{L}$}\label{tab0}
  \centering
  \begin{threeparttable} 
   \begin{tabular}{|c|c|}
   \hline
  $\theta_{1}(t)=t$ for $t\in\mathbb{R}$ &	 			$\theta_{2}(t)=\ln(t\!+\!\varepsilon)-\ln\varepsilon$ for $t>-\varepsilon$\\
  \hline	
  $\theta_{3}(t)=\frac{t}{t+\varepsilon}$ for $t>-\varepsilon$ &   			$\theta_{4}(t)=\frac{1-\exp(-\varepsilon t)}{1-\exp(-\varepsilon)}$ for $t>-\varepsilon$ \\
  \hline	 
  $\theta_{5}(t)=(t\!+\!\varepsilon)^{q}$ for $t>-\varepsilon$ &  			$\theta_{6}(t)=\frac{2}{\sqrt{3}}{\rm atan}(\frac{1+2\varepsilon t}{\sqrt{3}})-\frac{\pi}{3\sqrt{3}}$ for $t>-\frac{1}{2\varepsilon}$\\
  \hline
  \end{tabular}
  {\small where $\varepsilon$ appearing in this table is a given small positive constant.}
  \end{threeparttable} 
  \end{table} 
  \renewcommand{\arraystretch}{1}
 
 \subsection{Main contribution}\label{sec1.1}
 
 This work aims to develop an efficient algorithm with convergence certificate for solving the composite image reconstruction model \eqref{model}, and its main contributions are summarized as follows. 
 \begin{itemize}
 \item[(i)] We proposed a general image reconstruction model \eqref{model} under impluse noise by incorporating a nonconvex data-fidelity term and two different regularization terms. This model uses a family of concave fuctions to construct data-fidelity term so as to avoid the bias caused by the $\ell_1$-norm data-fidelity term, and adopts two different regularizers to enhance image recovery quality. As discussed above, it contains existing convex and nonconvex reconstruction models for impluse noise removal. 
 	
 \item[(ii)] For the nonconvex and nonsmooth model \eqref{model}, we propose an inexact proximal MM algorithm by seeking in each step an inexact minimizer of a strongly convex subproblem, and verify the convergence of the whole iterate sequence under the KL assumption on a potential function. Although the involved inexactness criterion is similar to the one adopted by Bonettini et al. \cite{Bonettini20} for designing forward-backward (FB) algorithms, our convergence analysis depends on the constructed potential function rather than the FB envelope of the objective function as in \cite{Bonettini20}. This inexactness criterion is easily implementable and weaker than the one in \cite{Zhang17} catering for the direct application of the convergence analysis framework proposed in \cite{Attouch13}. 
 	
 \item [(iii)] Our inexact proximal MM method is applied to resolve the image restoration problems from deblurring and inpainting, which are a special case of model \eqref{model}. Extensive numerical comparisons with the PLMA \cite{Zhang17} on gray image deblurring and ADMMLp \cite{Shang18} on color image inpainting show that our method has better PSNR and SSIM than the PLMA for the high noise case (see Tables \ref{tab1}-\ref{tab2}), and is remarkably superior to ADMMLp in PSNR and SSIM under noiseless and noisy scenarios (see Tables \ref{tab3}-\ref{tab4}).
 \end{itemize} 

 This paper is organized as follows. Section \ref{sec2} introduces the notion of stationary points for problem \eqref{model}, and provides some preliminary results that will be used in the subsequent sections. Section \ref{sec3} describes the iterate steps of the proposed method and verifies its well-definedness, and Section \ref{sec4} focuses on the global convergence analysis of the generated iterate sequence. In Section \ref{sec5}, we present the application of the proposed algorithm in image deblurring and inpainting tasks, and conduct numerical experiments to validate its efficiency in Section \ref{sec6}. Section \ref{sec7} concludes this paper. 
\section{Notation and preliminaries}\label{sec2}

 In this paper, $\mathbb{R}_{+}^{m\times n}$ represents the set comprising all nonnegative matrices in $\mathbb{R}^{m\times n}$, $\mathcal{I}$ is the identity mapping from $\mathbb{X}$ to $\mathbb{X}$, and for an integer $k\ge 1$, write $[k]:=\{1,2,\ldots,k\}$. For any $z\in\mathbb{R}^{m\times n}$, $\mathbb{B}(z,\delta)$ means the closed ball on the Frobenius norm centered at $z$ with radius $\delta$. For an $x\in\mathbb{X}$, $\mathbb{B}_{\mathbb{X}}(x,\delta)$ denotes the closed ball on the norm $\interleave\cdot\interleave$ centered at $x$ with radius $\delta$. For a positive definite (PD) linear operator $\mathcal{Q}\!:\mathbb{X}\to\mathbb{X}$, we write $\interleave x\interleave_{\mathcal{Q}}\!:=\!\sqrt{\langle x,\mathcal{Q}x\rangle}$, and for a linear mapping $\mathcal{A}\!:\mathbb{X}\to\mathbb{R}^{m\times n}$, let $\|\mathcal{A}\|:=\sup_{\interleave\!x\!\interleave=1}\|\mathcal{A}x\|_F$. For a matrix $x\in\mathbb{R}^{m\times n}$ and an index $j\in[n]$, $x_{\cdot j}$ means the $j$th column of $X$, and $\|x\|_{2,1}:=\sum_{j=1}^n\|x_{\cdot j}\|$ denotes the column $\ell_{2,1}$-norm of $x$. For a closed set $\Delta\subset\mathbb{X}$, $\mathbb{I}_{\Delta}$ denotes the indicator function of $\Delta$, i.e., $\mathbb{I}_{\Delta}(x)=0$ if $x\in\Delta$, otherwise $\mathbb{I}_{\Delta}(x)=\infty$, and $\Pi_{\Delta}$ denotes the projection mapping onto the set $\Delta$. For a differentiable mapping $g\!:\mathbb{X}\to\mathbb{R}^{m\times n}$, $\nabla g(x)$ denotes the adjoint of $g'(x)$, the differential mapping of $g$ at $x$, and if $g$ is twice differentiable at $x$, $D^2g(x)$ represents the twice differential mapping of $g$ at $x$, and $D^2g(x)(u,\cdot)$ for $u\in\mathbb{X}$ is a linear mapping from $\mathbb{X}$ to $\mathbb{R}^{m\times n}$. For an extended real-valued function $f\!:\mathbb{X}\to\overline{\mathbb{R}}$, denote by ${\rm dom}\,f:=\{x\in\mathbb{X}\,|\,f(x)<\infty\}$ the effective domain of $f$, and if $f$ is strictly continuous at $\overline{x}\in\mathbb{X}$, ${\rm lip}\,f(\overline{x})$ denotes the Lipschitz modulus of $f$ at $\overline{x}$.

\subsection{Stationary points of problem \eqref{model}}\label{sec2.1} 

 To introduce the notion of stationary points for problem \eqref{model}, we first recall from \cite{RW98} the basic subdifferential of an extended real-valued function $f\!:\mathbb{X}\to\overline{\mathbb{R}}$ at a point of its domain.
 \begin{definition}\label{def-subdiff}(see \cite[Definition 8.3]{RW98}) Consider a function $f\!:\mathbb{X}\to\overline{\mathbb{R}}$ and a point $x\in{\rm dom}\,f$. The regular (or Fr\'echet) subdifferential of $f$ at $x$ is defined as
 	\[
 	\widehat{\partial}f(x):=\bigg\{v\in\mathbb{X}\ |\ \liminf_{x\ne x'\to x}\frac{f(x')-f(x)-\langle v,x'-x\rangle}{\interleave x'-x\interleave}\ge 0\bigg\}, 
 	\]
 	and the basic (also known as limiting or Morduhovich) subdifferential of $f$ at $x$ is defined as 
 	\[
 	\partial f(x)\!:=\!\Big\{v\in \mathbb{X}\ |\ \exists x^k \to x\ {\rm with}\ f(x^k)\to f(x), v^k\!\in \widehat{\partial}f(x^k)\ {\rm with}\ v^k \to v \Big\}. 
 	\]
 \end{definition}
 \begin{remark}\label{remark-subdiff}
  {\bf (a)} By Definition \ref{def-subdiff}, it is easy to check that $\widehat{\partial}f(x)\subset\partial f(x)$, $\widehat{\partial}f(x)$ is closed and convex, and $\partial f(x)$ is closed but generally not convex. When $f$ is convex, they both reduce to the subdifferential in the sense of convex analysis \cite{Roc70}. 
 	
 \noindent
 {\bf(b)} Let $\{(x^k,v^k)\}_{k\in\mathbb{N}}$ be a sequence converging to $(\overline{x},\overline{v})$ from the graph of $\partial\!f$. Then, when $f(x^k)\to f(\overline{x})$, $(\overline{x},\overline{v})\in{\rm gph}\,\partial\!f$. In the sequel, if $0\in\partial f(x)$, the vector $x\in\mathbb{X}$ is called a stationary point of the problem $\inf_{z\in\mathbb{X}}f(z)$ or a critical point of $f$.
 \end{remark} 
 
 Now we characterize the subdifferential of the function $\Theta$. This requires the following chain rule for the subdifferential of composite functions.  
\begin{lemma} \label{chain-sub}
 Let $f(x)\!:=\vartheta(|x|)$ for $x\in\mathbb{R}^{m\times n}$. Then $f$ is strictly continuous and regular, and at any $\overline{x}\in\mathbb{R}^{m\times n}$, $\partial f(\overline{x})=\big\{v\in\mathbb{R}^{m\times n}\ |\ v_{ij}\in\theta'(|\overline{x}_{ij}|)\,\partial|\overline{x}_{ij}|\big\}$.
\end{lemma}
\begin{proof}
 Note that $f(x)=\sum_{i=1}^m\sum_{j=1}^n\theta(|x_{ij}|)$ with $\theta\in\!\mathscr{L}$. The strict continuity of $f$ is trivial. Since $f$ is separable with respect to variable $x_{ij}$, by \cite[Proposition 10.5]{RW98}, it suffices to argue that $\phi(t):=\theta(|t|)$ for $t\in\mathbb{R}$ is regular and its subdifferential at any $\overline{t}\in\mathbb{R}$ is $\partial\phi(\overline{t})=\big\{v\in\mathbb{R}\,|\, v\in\theta'(|\overline{t}|)\partial|\overline{t}|\big\}$. Fix any $\overline{t}\in\mathbb{R}$. When $\overline{t}\neq0$, obviously, $\phi$ is differentiable at $\overline{t}$ with $\phi'(\overline{t})=\theta'(|\overline{t}|){\rm sign}(\overline{t})$, so the desired equality holds. When $\overline{t}=0$, since $\theta'(0)\ge 0$, the function $t'\mapsto \theta'(0)|t'|$ is regular. By  \cite[Theorem 10.49]{RW98},  $\partial\phi(0)=\theta'(0)\partial |0|$, so the desired equality also holds.
\end{proof}
\begin{proposition}\label{subdiff-Phi}
 The function $\Theta$ is strictly continuous relative to ${\rm dom}\,h$ and regular, and at any $x\in{\rm dom}\,h$,
 \[
   \partial\Theta(x)=\nabla F(x)[\nabla\vartheta(|F(x)|)\circ\partial|F(x)|]+\nu\nabla T(x)\partial\psi(T(x)) + \lambda\,\partial h(x)
 \]
 where, for a given $y\in\mathbb{R}^{m\times n}$, $y\circ\partial|F(x)|=\{v\in\mathbb{R}^{m\times n}\,|\, v_{ij}\in y_{ij}\,\partial|[F(x)]_{ij}|\}$. 
\end{proposition}  
\begin{proof}  
 Let $\widetilde{f}(x)=f(F(x))$ and $g(x)=\psi(T(x))$ for $x\in\mathbb{X}$, where $f$ is the function defined in Lemma \ref{chain-sub}. Clearly, $\widetilde{f}+\nu g$ is strictly continuous on $\mathbb{X}$. Recall that $h$ is strictly continuous relative to its domain. The function $\Theta=\widetilde{f}+\nu g+\lambda h$ is strictly continuous relative to ${\rm dom}\,h$. Note that $\widetilde{f}=f\circ F$ and $g=\psi\circ T$. By using Lemma \ref{chain-sub} and \cite[Theorem 10.6]{RW98}, $\widetilde{f}$ and $g$ are regular, and at any $x\in\mathbb{X}$, $\partial\widetilde{f}(x)=\nabla F(x)\partial f(F(x))=\nabla F(x)[\nabla\vartheta(|F(x)|)\circ\partial|F(x)|]$ and $\partial g(x)=\nabla T(x)\partial\psi(T(x))$. By combining \cite[Corollary 10.9 \& Theorem 9.13]{RW98} with the strict continuity and regularity of $\widetilde{f}$ and $g$ and the convexity of $h$, we have that the function $\Theta$ is regular, and at any $x\in\mathbb{X}$, $\partial\Theta(x)=\partial\widetilde{f}(x)+\nu\partial g(x)+\lambda\partial h(x)$. The conclusion then follows.
\end{proof} 

 By Remark \ref{remark-subdiff} (b) and Proposition \ref{subdiff-Phi}, we define the following stationary point. 
\begin{definition}\label{def-Spoint}
 A vector $x\in{\rm dom}\,h$ is called a stationary point of \eqref{model} if 
 \[
  0\in\nabla F(x)[\nabla\vartheta(|F(x)|)\circ\partial|F(x)|]+\nu\nabla T(x)\partial\psi(T(x)) +\lambda\partial h(x).
 \]
\end{definition} 
\subsection{KL property of nonsmooth functions}\label{sec2.2}

 The KL property of an extended real-valued function plays a crucial role in the convergence (rate) analysis of algorithms for nonconvex and nonsmooth optimization problems (see \cite{Attouch10,Attouch13}). To recall the KL property of an extended real-valued function, for any $\eta\in(0,\infty]$, we denote by $\Upsilon_{\!\eta}$ the set consisting of all continuous concave function $\varphi\!:[0,\eta)\to\mathbb{R}_+$ that is continuously differentiable on $(0,\eta)$ with $\varphi(0)=0$ and $\varphi'(s)>0$ for all $s\in(0,\eta)$. 
\begin{definition}\label{Def-KL}
 A proper lsc  $f\!:\mathbb{X}\to\mathbb{\overline{R}}$ is said to have the KL property at $\overline{x}\in{\rm dom}\,\partial f$ if there exist $\eta\in(0,\infty],\, \varphi\in\Upsilon_{\!\eta}$ and a neighborhood $\mathcal{V}$ of $\overline{x}$ such that 
 \[ 
	\varphi'(f(x)-f(\overline{x})){\rm dist}(0,\partial f(x))\ge 1 \quad {\rm for}\ {\rm all}\ x\in \mathcal{V}\cap [f(\overline{x})<f<f(\overline{x})+\eta].    
 \]
 If $\varphi$ can be chosen as $\varphi(t)=ct^{1-\gamma}$ with $\gamma\in[0,1)$ for some $c>0$, then $f$ is said to have the KL property of exponent $\gamma$ at $\overline{x}$. If $f$ has the KL property (of exponent $\gamma$) at each point of ${\rm dom}\,\partial f$, it is called a KL function (of exponent $\gamma$). 
\end{definition}

 As illustrated in \cite[Section 4]{Attouch10}, KL functions are rather extensive and cover semialgebraic functions, global subanalytic functions, and functions definable in an o-minimal structure over the real field $(\mathbb{R},+,\cdot)$. From \cite{Dries96}, the functions $\theta_1$-$\theta_5$ in Table \ref{tab0} and the associated $\vartheta$ are  definable in an o-minimal structure over the real field $(\mathbb{R},+,\cdot)$. By \cite[Lemma 2.1]{Attouch10}, a proper lsc $f\!:\mathbb{X}\to\overline{\mathbb{R}}$ has the KL property of exponent $0$ at every noncritical point, so to argue that a proper lsc function is a KL function of exponent $\gamma\in[0,1)$, it suffices to check its KL property of exponent $\gamma\in[0,1)$ at all critical points. A family of zero-norm regularized functions were proved to have KL property of exponent $1/2$ in \cite{WuPanBi21}. On the calculation rule of KL exponent, see \cite{LiPong18,YuLiPong21}. 
\subsection{Local linearization of composite functions}\label{sec2.3}

 We present a lemma to state the relation between the composite $\vartheta\circ F$ and its local linearization at a point. 
\begin{lemma} \label{local-lip1}
 Fix any $\overline{x}\in\mathbb{X}$. Let $\ell_{F}(x;\overline{x})\!:=F(\overline{x})+F'(\overline{x})(x\!-\!\overline{x})$ for $x\in\mathbb{X}$. Then, there exist $\overline{\delta}_1>0,\overline{\alpha}_1>{\rm lip}\,\vartheta(|F(\overline{x})|)\,{\rm lip}\,F'(\overline{x})$ and $\overline{\alpha}_2\!>{\rm lip}\,\nabla\vartheta(|F(\overline{x})|)$ such that for all $x\in\mathbb{B}(\overline{x},\overline{\delta}_1)$, 
 \begin{align}\label{lip-vtheta1}
 \big|\vartheta(|F(x)|)-\vartheta(|\ell_{F}(x;\overline{x})|)\big|\le (\overline{\alpha}_1/2)\interleave\!x-\overline{x}\interleave^2,\\
 \big|\vartheta(|\ell_{F}(x;\overline{x})|)-\vartheta(|F(\overline{x})|)-\langle\nabla\vartheta(|F(\overline{x})|),|\ell_{F}(x;\overline{x})|- |F(\overline{x})|\rangle\big|\nonumber\\
 \le({\overline{\alpha}_2}/{2})\|F'(\overline{x})\|^2\!\interleave\!x-\overline{x}\!\interleave^2.\qquad\qquad\qquad\label{lip-vtheta2}
 \end{align}
\end{lemma}
\begin{proof}
 Recall that $\theta\in\!\mathscr{L}$ is differentiable on the open interval  $I\supset\mathbb{R}_{+}$ with a strictly continuous derivative $\theta'$, so $\vartheta$ and $\nabla\vartheta$ are strictly continuous on an open set $\mathcal{V}\supset\mathbb{R}_{+}^{m\times n}$. Fix any $\epsilon>0$. There exists $\delta_1>0$ such that for all $y,y'\in\mathbb{B}(|F(\overline{x})|,\delta_1)$, 
 \begin{equation} \label{ineq1-major}
 |\vartheta(y)-\vartheta(y')|\le [{\rm lip}\,\vartheta(|F(\overline{x})|)\!+\!\epsilon]\|y-y'\|_F.
 \end{equation}
 By the continuous differentiability of $F$ at $\overline{x}$, there exists $\overline{\delta}_1>0$ such that for any $x\in\mathbb{B}_{\mathbb{X}}(\overline{x},\overline{\delta}_1)$, 	$\|F(x)-F(\overline{x})\|_F\le\delta_1$ and $\|F'(\overline{x})(x-\overline{x})\|_F\le\delta_1$. Note that for any $y,y'\in\mathbb{R}^{m\times n}$ and any $(i,j)\in[m]\times[n]$, $|\,|y_{ij}|-|y_{ij}'|\,|\le |y_{ij}-y_{ij}'|$. Then, 
 \begin{equation*}
 \|\,|F(x)|-|F(\overline{x})|\,\|_F\le\delta_1\ \ {\rm and}\ \  \|\,|\ell_{F}(x;\overline{x})|-|F(\overline{x})|\,\|_F\le \delta_1\quad\forall x\in\mathbb{B}_{\mathbb{X}}(\overline{x},\overline{\delta}_1).
 \end{equation*}
 As $F'$ is strictly continuous at $\overline{x}$, if necessary by reducing $\overline{\delta}_1$, for any $x,x'\in\mathbb{B}_{\mathbb{X}}(\overline{x},\overline{\delta}_1)$   
 \begin{equation*}
  \|F'(x)-F'(x')\|\le\big[{\rm lip}F'(\overline{x})\!+\!\epsilon\big]\!\interleave \!x-x'\interleave.
 \end{equation*}
  Pick any $x\in\mathbb{B}_{\mathbb{X}}(\overline{x},\overline{\delta}_1)$. Using \eqref{ineq1-major} with $y=|F(x)|$ and $y'=|\ell_{F}(x;\overline{x})|$ yields that 
  \begin{align*}
  &\big|\vartheta(|F(x)|)-\vartheta(|\ell_{F}(x;\overline{x})|)\big|\\
  &\le[{\rm lip}\,\vartheta(|F(\overline{x})|)+\epsilon]\,\big\|\,|F(x)|-|\ell_{F}(x;\overline{x})|\,\big\|_F \\
  &\le[{\rm lip}\,\vartheta(|F(\overline{x})|)+\epsilon]\,\big\|F(x)-F(\overline{x})+F'(\overline{x})(x\!-\!\overline{x})\big\|_F \\
  &=[{\rm lip}\,\vartheta(|F(\overline{x})|)\!+\!\epsilon]\,\Big\|\int_{0}^{1}\![F'(\overline{x}\!+t(x\!-\!\overline{x}))\!-\!F'(\overline{x})\big](x\!-\!\overline{x})dt\Big\|_F\\	
  &\le\frac{1}{2}[{\rm lip}\,\vartheta(|F(\overline{x})|)\!+\epsilon]\,[{\rm lip}F'(\overline{x})\!+\epsilon]\interleave\!x-\overline{x}\interleave^2
  \end{align*}
  where the equality is due to the mean-valued theorem. By the arbitrariness of $x\in\mathbb{B}_{\mathbb{X}}(\overline{x},\overline{\delta}_1)$, we obtain inequality \eqref{lip-vtheta1}. Since $\nabla\vartheta$ is strictly continuous on an open set $\mathcal{V}\supset\mathbb{R}_{+}^{m\times n}$, it is Lipschitz continuous on the compact set $\mathbb{B}(|F(\overline{x})|,\delta_1)$ (if necessary by shrinking $\delta_1$). By invoking the descent lemma (see \cite[Proposition A.24]{Bertsekas99}), there exists $\overline{\alpha}_2>{\rm lip}\,\nabla\vartheta(|F(\overline{x})|)$ such that for all $y,y'\in \mathbb{B}(|F(\overline{x})|,\delta_1)$, 
  \[
	|\vartheta(y')-\vartheta(y)-\langle\nabla\vartheta(y),y'-y\rangle|\le(\overline{\alpha}_2/2)\|y'-y\|_F^2. 
  \] 
  From the above arguements, for any $x\in\mathbb{B}_{\mathbb{X}}(\overline{x},\overline{\delta}_1)$ (if necessary by shrinking $\overline{\delta}_1$), $|\ell_{F}(x;\overline{x})|\in\mathbb{B}(|F(\overline{x})|,\delta_1)$. Now substituting $y'\!=|\ell_{F}(x;\overline{x})|$ and $y\!=|F(\overline{x})|$ into the above inequality yields the desired inequality \eqref{lip-vtheta2}. The proof is then completed. 
 \end{proof}

 Recall that $\psi$ is finite and convex on $\mathbb{R}^{p\times q}$, so it is strictly continuous on $\mathbb{R}^{p\times q}$. Using the same arguments as for the first part of Lemma \ref{local-lip1} gives the following result. 
 \begin{lemma} \label{local-lip2}
 Fix any $\overline{x}\!\in\mathbb{X}$. Let $\ell_{T}(x;\overline{x})\!:=T(\overline{x})+T'(\overline{x})(x\!-\!\overline{x})$ for $x\in\mathbb{X}$. Then, there exist $\overline{\delta}_2>0$ and $\overline{\alpha}_3>{\rm lip}\,\psi(T(\overline{x}))\,{\rm lip}\,T'(\overline{x})$ such that for any $x\in\mathbb{B}_{\mathbb{X}}(\overline{x},\overline{\delta}_2)$, 
 \[
	\big|\psi(T(x))-\psi(\ell_{T}(x;\overline{x}))\big|
	\le(\overline{\alpha}_3/2)\!\interleave\!x-\overline{x}\interleave^2.
 \]
\end{lemma}

 To close this section, we establish the outer semicontinuity of a multifunction, which will be used in the convergence analysis of Section \ref{sec4}.
\begin{lemma}\label{lemma-osc}
 Let $\mathcal{F},\mathcal{G}:\mathbb{X}\times\mathbb{X}\rightrightarrows\mathbb{X}$ be defined by $\mathcal{F}(x,z)\!:=\!\nabla T(x)\partial\psi(\ell_{T}(x;z))$ and $\mathcal{G}(x,z):=\nabla F(x)[\nabla\vartheta(|F(z)|)\circ\partial|\ell_F(x,z)|]$, respectively. Then, for any $(\overline{x},\overline{z})\in\mathbb{X}\times\mathbb{X}$,
 \[
  \limsup_{(x,z)\to(\overline{x},\overline{z})}\mathcal{F}(x,z)\subset\mathcal{F}(\overline{x},\overline{z})\ \ {\rm and}\ \ 
  \limsup_{(x,z)\to(\overline{x},\overline{z})}\mathcal{G}(x,z)\subset\mathcal{G}(\overline{x},\overline{z}),
 \]
  i.e., the multifunctions $\mathcal{F}$ and $\mathcal{G}$ are both outer semicontinuous.   	
 \end{lemma}
 \begin{proof}
  Fix any $(\overline{x},\overline{z})\in\mathbb{X}\times\mathbb{X}$. Pick any $\overline{u}\in\limsup_{(x,z)\to(\overline{x},\overline{z})}\mathcal{F}(x,z)$. By \cite[Definition 5.4]{RW98}, there exist $(x^k,z^k)\to(\overline{x},\overline{z})$ and $u^k\to\overline{u}$ with $u^k\in\mathcal{F}(x^k,z^k)$ for each $k\in\mathbb{N}$. By the expression of $\mathcal{F}$, for each $k\in\mathbb{N}$, there exists $y^k\in\partial\psi(\ell_{T}(x^k,z^k))$ such that $u^k=\nabla T(x^k)y^k$. Recall that $\psi$ is strictly continuous by its finite convexity. From \cite[Theorem 9.13 \& Proposition 5.15]{RW98}, the sequence $\{y^k\}_{k\in\mathbb{N}}$ is bounded. If necessary by taking a subsequence, we assume that $y^k\to\overline{y}$ as $k\to\infty$. Together with the outer semicontinuity of $\partial\psi$ and the continuity of $\ell_{T}(\cdot,\cdot)$, we have $\overline{y}\in\partial\psi(\ell_{T}(\overline{x},\overline{z}))$. Recall that $u^k\to \overline{u}$ and $u^k=\nabla T(x^k)y^k$ for each $k\in\mathbb{N}$. From the continuity of the mapping $T$, we obtain $\overline{u}=\nabla T(\overline{x})\overline{y}\in\nabla T(\overline{x})\partial\psi(\ell_{T}(\overline{x};\overline{z}))=\mathcal{F}(\overline{x},\overline{z})$. By the arbitrariness of $\overline{u}$, the first inclusion holds. Using the similar arguments leads to the second one. 
\end{proof}
\section{Inexact proximal MM algorithm}\label{sec3}

 To state the basic idea of our inexact proximal MM algorithm, define $\ell_{F}\!:\mathcal{O}\times\mathcal{O}\to\mathbb{R}^{n\times m}$ and $\ell_{T}\!:\mathcal{O}\times\mathcal{O}\to\mathbb{R}^{p\times q}$ by
 \begin{equation}\label{ellFG}
  \ell_F(x;z):=F(z)+F'(z)(x-z)\ \ {\rm and}\ \  \ell_T(x;z):=T(z)+T'(z)(x-z).
 \end{equation}
 Clearly, for each $k\in\mathbb{N}$, $\ell_F(\cdot;x^k)$ and $\ell_T(\cdot;x^k)$ are the linear approximation of $F$ and $T$ at $x^k$, respectively. By Lemma \ref{local-lip1} for $\overline{x}=x^k$, there are $\alpha_{1,k}>{\rm lip}\,\vartheta(|F(x^k)|)\,{\rm lip}\,F'(x^k)$ and $\alpha_{2,k}>{\rm lip}\,\nabla\vartheta(|F(x^k)|)$ such that for any $x$ sufficiently close to $x^k$, 
\begin{align}\label{ineq1-majorF}
 \vartheta(|F(x)|)
 &\le\vartheta(|\ell_F(x;x^k)|)+(\alpha_{1,k}/2)\interleave\!x-x^k\interleave^2 \nonumber \\
 &\le\vartheta(|F(x^k)|)+\langle\nabla\vartheta(|F(x^k)|), |\ell_F(x;x^k)|-|F(x^k)|\rangle\nonumber \\
 &\quad +\frac{1}{2}(\alpha_{1,k}+\alpha_{2,k}\|F'(x^k)\|^2)\interleave\!x-x^k\!\interleave^2.
\end{align}
 By invoking Lemma \ref{local-lip2} with $\overline{x}=x^k$, there exists $\alpha_{3,k}>{\rm lip}\,\psi(T(x^k))\,{\rm lip}\,T'(x^k)$ such that for any $x$ sufficiently close to $x^k$, 
\begin{equation}\label{ineq1-majorG}
 \psi(T(x))\le \psi(\ell_{T}(x;x^k))+(\alpha_{3,k}/2)\interleave\!x-x^k\!\interleave^2.
\end{equation}
 For each $k\in\mathbb{N}$, write $L_k\!:=\alpha_{1,k}+\alpha_{2,k}\|F'(x^k)\|^2+\nu\alpha_{3,k}$. From inequalities \eqref{ineq1-majorF}-\eqref{ineq1-majorG} and the expression of $\Theta$, for any $x$ sufficiently close to $x^k$, it holds that 
\begin{equation*}
 \Theta(x)\le \langle\nabla\vartheta(|F(x^k)|),|\ell_{F}(x;x^k)|\rangle +\nu\psi(\ell_{T}(x;x^k))+\lambda h(x)+\!\frac{L_k}{2}\interleave\!x-x^k\!\interleave^2+C^k
\end{equation*}
with $C^k=\vartheta(|F(x^k)|)-\langle\nabla\vartheta(|F(x^k)|),|F(x^k)|\rangle$. 
Thus, by choosing a PD linear operator $\mathcal{Q}_k\!:\mathbb{X}\to \mathbb{X}$ with $\mathcal{Q}_k\succeq L_k\mathcal{I}$, for any $x$ sufficiently close to $x^k$,
\begin{equation*}
 \Theta(x)\le\!\Theta_k(x)\!:=\!\langle\nabla\vartheta(|F(x^k)|), |\ell_{F}(x;x^k)|\rangle+\nu\psi(\ell_T(x;x^k))+\lambda h(x)+\frac{1}{2}\interleave\!x-x^k\interleave_{\mathcal{Q}_k}^2 +C^k.
\end{equation*}
 Together with $\Theta_k(x^k)=\Theta(x^k)$, $\Theta_k$ is a strongly convex majorization of $\Theta$ at $x^k$. 
 
 The fact that $\Theta_k$ is a strongly convex majorization of $\Theta$ at $x^k$  inspires us to propose a proximal MM algorithm, for which we search an appropriate constant $L_k$ and minimize the associated $\Theta_k$ synchronously in each step. Among others, an inexact minimizer of $\Theta_k$ satisfying a mild condition is enough to guarantee the practibility of the proposed algorithm, whose iterate steps are described as follows. 
\begin{algorithm}[h] 
\caption{\label{iPMM}(\bf iPMM for solving problem \eqref{model})}
 \textbf{Initialization:} Input $\theta\in\!\mathscr{L},\nu>0$ and $\lambda>0$. Choose  parameters $\varrho>1,0<\underline{\gamma}<\overline{\gamma}$, $\overline{\mu}>0$ and $\{\mu_k\}_{k\in \mathbb{N}}\subset(0,\overline{\mu}]$, and an initial point $x^0\in\mathbb{X}$. Set $k=0$.  
	
 \textbf{while} the stopping conditions are not satisfied \textbf{do}	
 \begin{enumerate}
  \item Choose $\gamma_{k,0}\in[\underline{\gamma},\overline{\gamma}]$ and compute $C^k=\vartheta(|F(x^k)|)-\langle\nabla\vartheta(|F(x^k)|),|F(x^k)|\rangle$.
		
  \item \textbf{for} $j=0,\,1,\,2,\cdots$
		\begin{itemize}
		 \item[{\bf(a)}] Choose a PD linear operator $\mathcal{Q}_{k,j}$ with $\gamma_{k,j}\mathcal{I}\preceq \mathcal{Q}_{k,j}\preceq \varrho\gamma_{k,j}\mathcal{I}$. Compute		 
		 \begin{align}\label{subprobk}
		  \min_{x\in\mathbb{X}}\ \Theta_{k,j}(x)
		  &:=\langle\nabla\vartheta(|F(x^k)|), |\ell_{F}(x;x^k)|\rangle +\nu\psi(\ell_T(x;x^k))\nonumber\\ 
		  &\qquad + \lambda h(x)+(1/2)\interleave\!x-x^k\!\interleave_{\mathcal{Q}_{k,j}}^2+C^k
		 \end{align}
		 to seek an inexact minimizer $x^{k,j}\in\mathbb{X}$ satisfying the condition that
		\begin{equation} \label{inexct-cond}
		 \Theta_{k,j}(x^{k,j})<\Theta(x^k)\ {\rm and}\  \Theta_{k,j}(x^{k,j})\!-\Theta_{k,j}(\overline{x}^{k,j})\le\frac{\mu_k}{2}(\Theta(x^k)-\Theta_{k,j}(x^{k,j})),
		\end{equation}
		where $\overline{x}^{k,j}$ denotes the unique optimal solution of subproblem (\ref{subprobk}).
			
		\item[{\bf(b)}] If $\Theta(x^{k,j})\le\Theta_{k,j}(x^{k,j})$, set $j_k=j$ and go to step 4; else let $\gamma_{k,j+1}=\varrho\gamma_{k,j}$.
		\end{itemize}
  \item \textbf{end (for)}
		
  \item Set $x^{k+1}=x^{k,j_k},\overline{x}^{k+1}\!=\overline{x}^{k,j_k}, \mathcal{Q}_{k}=\mathcal{Q}_{k,j_k}$. Let $k\gets k+1$ and go to step 1.
 \end{enumerate}
 \textbf{end (while)}
 \end{algorithm}
 \begin{remark}\label{remark-alg}
 {\bf(a)} Note that the Lipschitz modulus of $F'$ and $T'$ at $x^k$, that of $\vartheta$ and $\nabla\vartheta$ at $|F(x^k)|$ and that of $\psi$ at $T(x^k)$ are generally unknown, and it is not easy to provide a good upper estimation for them. The inner for-end loop of Algorithm \ref{iPMM} is searching a desirable upper estimation $L_k$ for ${\rm lip}\,\vartheta(|F(x^k)|)\,{\rm lip}\,F'(x^k)+{\rm lip}\,\nabla\vartheta(|F(x^k)|)\|F(x^k)\|^2+\nu{\rm lip}\,\psi(T(x^k))\,{\rm lip}\,T'(x^k)$ and computing an inexact minimizer of the associated majorization $\Theta_k$ synchronously. 
	
 \noindent
 {\bf(b)} The inexactness criterion for $x^{k,j}$ consists of the two conditions in \eqref{inexct-cond}. Note that $\Theta(x^k)=\Theta_{k,j}(x^k)$ for each $k$ and $j$. The first one requires the iterate $x^{k,j}$ to reduce the objective value of \eqref{subprobk}, while the second one restricts the reduction to be at least  $({\mu_k}/{2})[\Theta_{k,j}(x^k)-\Theta_{k,j}(x^{k,j})]$. The inexactness criterion \eqref{inexct-cond} is practical because a lower estimation of the optimal value $\Theta_{k,j}(\overline{x}^{k,j})$ is easily achieved when applying a dual method to solve \eqref{subprobk}. A similar inexactness condition was adopted by Bonettini et al. \cite{Bonettini20} to design inexact forward-backward algorithms for minimizing the sum of a continuously differentiable function and a closed proper convex function. The inexactness criterion \eqref{inexct-cond} is significantly different from the one in \cite{Zhang17} that takes the form ${\rm dist}(0,\partial\Theta_{k,j}(x^{k,j}))\le\gamma_0\interleave x^{k,j}-x^k\interleave$ for model \eqref{model}. This subdifferential-type inexactness condition aims at the direct use of the convergence analysis technique under the KL framework in \cite{Attouch13}, and as was illustrated by \cite[Example 1]{Bonettini20}, it might not be necessarily true when the inexactness criterion \eqref{inexct-cond} is adopted. 
 
 \noindent
 {\bf(c)} Recall that $\theta'(t)\ge 0$ for all $t\ge 0$. For every $k\in\mathbb{N}$ and $j\in\{0\}\cup[j_k]$, the function $\Theta_{k,j}$ is strongly convex. By \cite[Theorem 23.8]{Roc70}, its subdifferential has the expression
 \begin{align*}
 \partial\Theta_{k,j}(x)
 &=\nabla F(x^k)\big[\nabla\vartheta(|F(x^k)|)\circ\partial|\ell_{F}(x;x^k)|\big]+\nu\nabla T(x^k)\partial\psi(\ell_{T}(x;x^k))\\
 &\quad\ \ +\lambda\partial h(x)+\mathcal{Q}_{k,j}(x-x^k)\quad\ \forall x\in{\rm dom}\,h. 
 \end{align*}  
 By comparing with Proposition \ref{subdiff-Phi}, it is immediate to obtain $\partial \Theta_{k,j_k}(x^k)=\partial\Theta(x^k)$.   
	 
 \noindent
 {\bf(d)} If $x^{k+1}=x^{k}$ for some $k$, from the strong convexity of $\Theta_{k,j}$, the second inequality of \eqref{inexct-cond} with $j=j_k$, and $\Theta(x^k)=\Theta_{k,j_k}(x^k)$, it is easy to have that 
 \[
 \underline{\gamma}\interleave x^{k+1}\!-\overline{x}^{k+1}\interleave^2\le\overline{\mu}[\Theta_{k,j}(x^k)-\Theta_{k,j}(x^{k+1})]=0, 
 \]
 and then $x^{k+1}=\overline{x}^{k+1}=x^k$. Along with $0\in\partial\Theta_{k,j_k}(\overline{x}^{k+1})=\partial \Theta_{k,j_k}(x^{k})=\partial\Theta(x^k)$, it follows that $x^{k}$ is a stationary point of \eqref{model}. According to this fact, one can use $\interleave x^{k+1}\!-x^k\interleave\le\epsilon^*$ for a small $\epsilon^*>0$ as the stopping condition of Algorithm \ref{iPMM}.  
 \end{remark}
	
 From step (2a) of Algorithm \ref{iPMM}, for each $k\in\mathbb{N}$, $x^k\in{\rm dom}\,h$. We close this section by showing that the sequence $\{x^{k}\}_{k\in\mathbb{N}}$ generated by Algorithm \ref{iPMM} is well defined. 
 \begin{lemma}\label{well-define}
 For each $k\in \mathbb{N}$, the inner for-end loop stops within a finite number of steps, and consequently Algorithm \ref{iPMM} is well defined.  
 \end{lemma}
 \begin{proof}
 Suppose on the contrary that there exists $\overline{k}\in\mathbb{N}$ such that the inner for-end loop of Algorithm \ref{iPMM} does not stop within a finite number of steps, i.e., 
 \begin{equation}\label{aim-ineq30}
 \Theta(x^{\overline{k},j})>\Theta_{\overline{k},j}(x^{\overline{k},j})\quad{\rm for\ each}\ j\in\mathbb{N}.
 \end{equation}
 We first claim that $\lim_{j\to\infty}x^{\overline{k},j}= x^{\overline{k}}$. Indeed, from the first inequality of \eqref{inexct-cond} and the expression of $\Theta_{k,j}$, for each $j\in \mathbb{N}$, it holds that 
 \begin{align}\label{tineq1-Phi}
 \Theta(x^{\overline{k}})>\Theta_{\overline{k},j}(x^{\overline{k},j})
 &=\langle\nabla\vartheta(|F(x^{\overline{k}})|), |\ell_{F}(x^{\overline{k},j};x^{\overline{k}})| \rangle+\nu\psi(\ell_T(x^{\overline{k},j};x^{\overline{k}})) \nonumber\\ 
 &\quad\ +\lambda h(x^{\overline{k},j}) +\frac{1}{2}\interleave\!x^{\overline{k},j}\!-x^{\overline{k}}\!\interleave^2_{\mathcal{Q}_{\overline{k},j}}+C^{\overline{k}}.
 \end{align}
 Recall that ${\rm dom}\,h\ne\emptyset$. Along with its convexity, $\emptyset\ne{\rm ri}({\rm dom}\,h)\subset{\rm dom}\,\partial h$. Pick any $\widehat{x}\in{\rm ri}({\rm dom}\,h)$ and any $\widehat{v}\in\partial h(\widehat{x})$. Then, it holds that 
 \[
  h(x^{\overline{k},j})\ge h(\widehat{x})+\langle \widehat{v},x^{\overline{k},j}-\widehat{x}\rangle\quad\ \forall j\in \mathbb{N}.
 \]
 Let $g_{\overline{k}}(x)\!:=\langle\nabla\vartheta(|F(x^{\overline{k}})|), |\ell_{F}(x;x^{\overline{k}})|\rangle$ for $x\in\mathbb{X}$. Obviously, $g_{\overline{k}}$ is a finite convex function, which means that $\partial g_{\overline{k}}(\widehat{x})\ne\emptyset$. Similarly, we also have $\partial \psi(T(\widehat{x}))\ne\emptyset$. Pick any $\widehat{\xi}\in\partial g_{\overline{k}}(\widehat{x})$ and $\widehat{\zeta}\in \partial\psi(T(\widehat{x}))$. From the convexity of $g_{\overline{k}}$ and $\psi$, for each $j\in\mathbb{N}$, 
 \begin{align*}
  g_{\overline{k}}(x^{\overline{k},j})
  &\ge g_{\overline{k}}(\widehat{x})+\langle \widehat{\xi},x^{\overline{k},j}-\widehat{x}\rangle,\\
   \psi(\ell_T(x^{\overline{k},j};x^{\overline{k}}))
  &\ge \psi(T(\widehat{x}))+\langle \widehat{\zeta},\ell_T(x^{\overline{k},j};x^{\overline{k}})-T(\widehat{x})\rangle.
  \end{align*} 
  Together with the above inequality \eqref{tineq1-Phi}, it immediately follows that for each $j\in\mathbb{N}$,
  \begin{align*}
  \Theta(x^{\overline{k}})
  &>g_{\overline{k}}(\widehat{x})+\langle\widehat{\xi},x^{\overline{k},j}-\widehat{x}\rangle +\nu\psi(T(\widehat{x}))+\nu\langle \widehat{\zeta},\ell_T(x^{\overline{k},j};x^{\overline{k}})-T(\widehat{x})\rangle \\ &\quad+\lambda h(\widehat{x})+\lambda\langle \widehat{v},x^{\overline{k},j}-\widehat{x}\rangle+\frac{1}{2}\!\interleave\!x^{\overline{k},j}-x^{\overline{k}}\!\interleave_{\mathcal{Q}_{\overline{k},j}}^2+C^{\overline{k}}.
  \end{align*}
  Note that $\mathcal{Q}_{\overline{k},j}\succeq \gamma_{\overline{k},j}\mathcal{I}$ and $\lim_{j\to\infty}\gamma_{\overline{k},j}=\infty$. Passing the limit $j\to\infty$ to this inequality results in $\lim_{j\to\infty}x^{\overline{k},j}= x^{\overline{k}}$. The claimed limit holds. Now, for each sufficiently large $j$, by using inequalities \eqref{ineq1-majorF}-\eqref{ineq1-majorG} with $k=\overline{k}$ and $x=x^{\overline{k},j}$, we get 
  \begin{align*}
  \vartheta(|F(x^{\overline{k},j})|)
  &\le\vartheta(|F(x^{\overline{k}})|)+
	\langle\nabla\vartheta(|F(x^{\overline{k}})|),|\ell_{F}(x^{\overline{k},j};x^{\overline{k}})-|F(x^{\overline{k}})|\rangle \\
  &\quad+\frac{1}{2} (\alpha_{1,\overline{k}}+\alpha_{2,\overline{k}}\|F'(x^{\overline{k}})\|^2)\interleave\!x^{\overline{k},j}-x^{\overline{k}}\interleave^2 \\
  \psi(T(x^{\overline{k},j}))&\le\psi(\ell_T(x^{\overline{k},j};x^{\overline{k}}))+ ({\alpha_{3,\overline{k}}}/{2})\interleave\!x^{\overline{k},j}-x^{\overline{k}}\!\interleave^2.
  \end{align*}
  By adding these two inequlities together, it follows that for all sufficiently large $j$, 
  \begin{align*}
  \Theta(x^{\overline{k},j})&\le \langle\nabla\vartheta(|F(x^{\overline{k}})|),|\ell_{F}(x^{\overline{k},j};x^{\overline{k}})\rangle+\nu\psi(\ell_T(x^{\overline{k},j};x^{\overline{k}}))\\
  &\quad +\lambda h(x^{\overline{k},j}) +({L_{\overline{k}}}/{2})\interleave\!x^{\overline{k},j}-x^{\overline{k}}\!\interleave^2+C^{\overline{k}} 
  \end{align*}
  where $L_{\overline{k}}=\alpha_{1,\overline{k}}+\alpha_{2,\overline{k}}\|F'(x^{\overline{k}})\|^2+\nu\alpha_{3,\overline{k}}$ and $C^{\overline{k}}\!=\vartheta(|F(x^{\overline{k}})|)-\langle\nabla\vartheta(|F(x^{\overline{k}})|),|F(x^{\overline{k}})|\rangle$. Note that $\mathcal{Q}_{\overline{k},j}\succeq L_{\overline{k}}\mathcal{I}$ for all $j$ large enough. The above inequality, along with the expression of $\Theta_{\overline{k},j}(x^{\overline{k},j})$, implies that for all $j$ large enough, $\Theta(x^{\overline{k},j})\le\Theta_{\overline{k},j}(x^{\overline{k},j})$. Obviously, this contradicts inequality \eqref{aim-ineq30}. The proof is completed. 
  \end{proof}
 \section{Convergence analysis of Algorithm \ref{iPMM}}\label{sec4}
	
 For the convenience of analysis, for each $k\in\mathbb{N}$ and $j\in \{0\}\cup[j_k]$, we define the function $f_{k,j}\!:\mathbb{X}\to\overline{\mathbb{R}}$ by
 \begin{equation}\label{fkj}
  f_{k,j}(x):=\Theta_{k,j}(x)-\Theta(x^k).
 \end{equation}
 Before analyzing the global convergence of Algorithm \ref{iPMM}, we take a closer look at the properties of the sequence $\{x^k\}_{k\in\mathbb{N}}$ and the objective value sequence $\{\Theta(x^k)\}_{k\in\mathbb{N}}$.   
 \begin{proposition}\label{descent-Theta}
  Let $\{x^k\}_{k\in\mathbb{N}}$ be the sequence yielded by Algorithm \ref{iPMM}. Then, 
  \begin{itemize}
  \item [(i)] for each $k\in\mathbb{N}$ and each $j\in\{0\}\cup[j_k]$, it holds that
			\begin{equation*}
				\frac{\underline{\gamma}}{4(1\!+\!\overline{\mu})}\interleave\!x^k\!-\!x^{k,j}\interleave^2\le\frac{1}{4(1\!+\!\mu_k)}\interleave\!x^k\!-\!x^{k,j}\interleave_{\mathcal{Q}_{k,j}}^2\le -f_{k,j}(x^{k,j}).
			\end{equation*} 
			
  \item [(ii)] For each $k\in\mathbb{N}$, $\Theta(x^{k+1})\le\Theta(x^k)-\frac{\underline{\gamma}}{4(1+\overline{\mu})}\interleave\!x^{k+1}-x^k\!\interleave^2$. 
			
  \item[(iii)] The sequence $\{\Theta(x^k)\}_{k\in\mathbb{N}}$ is convergent with the limit denoted by $\varpi^*$.  
			
  \item[(iv)] $0\le -\sum_{k=1}^{\infty}f_{k-1,j_{k-1}}(x^k)=\sum_{k=1}^{\infty}[\Theta(x^{k-1})-\Theta_{k-1,j_{k-1}}(x^k)]<\infty$.
			
  \item[(v)] $\lim_{k\to\infty}\interleave x^{k+1}-x^k\interleave=0$ and $\lim_{k\to\infty}\interleave x^{k}-\overline{x}^k\interleave=0$. 
  \end{itemize}	
 \end{proposition}
 \begin{proof}
  {\bf(i)} Fix any $k\in\mathbb{N}$ and any $j\in\{0\}\cup[j_k]$. The first inequality is immediate by noting that $\mu_k\le\overline{\mu}$ and $\mathcal{Q}_{k,j}\succeq \underline{\gamma}\mathcal{I}$, so it suffices to prove the second inequality. From the strong convexity of $f_{k,j}$, for any $x',x\in {\rm dom}\,h$ and any $w\in\partial f_{k,j}(x)$, 
  \begin{equation}\label{sconvex}
  f_{k,j}(x')\ge f_{k,j}(x)+\langle w,x'-x\rangle +\frac{1}{2}\!\interleave\!x'-x\interleave_{\mathcal{Q}_{k,j}}^2.
  \end{equation}
  Recall that $\overline{x}^{k,j}$ is the unique optimal solution of \eqref{subprobk}, so $0\in\partial f_{k,j}(\overline{x}^{k,j})$. Substituting $x'=x^{k,j}, x=\overline{x}^{k,j}$ and $w=0$ into \eqref{sconvex} leads to $\frac{1}{2}\!\interleave x^{k,j}-\overline{x}^{k,j}\interleave_{\mathcal{Q}_{k,j}}^2 \le f_{k,j}(x^{k,j})-f_{k,j}(\overline{x}^{k,j})$. Along with the inexactness condition  \eqref{inexct-cond} and the expression of $f_{k,j}$, 
  \begin{equation}\label{temp-ineq41}
  \interleave x^{k,j}-\overline{x}^{k,j}\interleave_{\mathcal{Q}_{k,j}}^2 
  \le -\mu_kf_{k,j}(x^{k,j}).
  \end{equation}
  Substituting $x'=x^{k},x=\overline{x}^{k,j}$ and $w=0$ into \eqref{sconvex} and using  $f_{k,j}(x^k)=0$ yields that $\frac{1}{2}\!\interleave\!x^{k}-\overline{x}^{k,j}\interleave_{\mathcal{Q}_{k,j}}^2\le -f_{k,j}(\overline{x}^{k,j})$. Along with the second inequality of \eqref{inexct-cond}, we obtain
  \begin{equation}\label{temp-ineq42}
  \interleave x^{k}-\overline{x}^{k,j}\interleave_{\mathcal{Q}_{k,j}}^2 \le -(2+\mu_k)f_{k,j}(x^{k,j}),
  \end{equation}
  Combining inequalities \eqref{temp-ineq41}-\eqref{temp-ineq42} with the Cauchy-Schwartz inequality yields that 
  \begin{align}\label{temp-ineq43}
  \interleave x^k-x^{k,j}\interleave_{\mathcal{Q}_{k,j}}^2
  &\le 2\interleave x^k-\overline{x}^{k,j}\interleave_{\mathcal{Q}_{k,j}}^2
			+2\interleave\overline{x}^{k,j}\!-\!x^{k,j}\interleave_{\mathcal{Q}_{k,j}}^2\nonumber\\
  &\le -2(2+\mu_k)f_{k,j}(x^{k,j})-2\mu_kf_{k,j}(x^{k,j})\nonumber\\
  &= -4(1+\mu_k)f_{k,j}(x^{k,j}),
 \end{align}
 which implies that the second inequality holds. The proof of part (i) is completed.
		
 \noindent
 {\bf(ii)} Fix any $k\in\mathbb{N}$. From inequality \eqref{temp-ineq43} and  $\underline{\gamma}\mathcal{I}\preceq\gamma_{k,j}\mathcal{I}\preceq\mathcal{Q}_{k,j}$,it follows that 
 \begin{equation}\label{temp-ineq44}
 \frac{\underline{\gamma}}{4(1\!+\!\mu_k)}\interleave x^k-x^{k,j}\interleave^2\le\frac{1}{4(1\!+\!\mu_k)}\interleave x^k-x^{k,j}\interleave_{\mathcal{Q}_{k,j}}^2\le -f_{k,j}(x^{k,j}).  
 \end{equation}
 By taking $j=j_k$ for inequality \eqref{temp-ineq44} and using the definition of $x^{k+1}$,
 \[  
  f_{k,j_k}(x^{k+1}) \le -\frac{\underline{\gamma}}{4(1+\mu_k)}\interleave x^k-x^{k+1}\interleave^2 \le -\frac{\underline{\gamma}}{4(1+\overline{\mu})}\interleave x^{k}-x^{k+1}\interleave^2.  
 \]
 In addition, from step (2b) of Algorithm \ref{iPMM} and the definition of $x^{k+1}$, we have
 \begin{equation} \label{eq-descend}
  \Theta(x^{k+1})\le \Theta_{k,j_k}(x^{k+1})=f_{k,j_k}(x^{k+1})+\Theta(x^k).
 \end{equation}
 The above two inequalities imply that $\Theta(x^{k+1})\le\Theta(x^k)-\frac{\underline{\gamma}}{4(1+\overline{\mu})}\interleave\!x^{k+1}-x^k\!\interleave^2$. 
		
 \noindent
 {\bf(iii)} Recall that $h$ and $\psi$ are lower bounded. Along with the nonnegativity of $\vartheta$, the function $\Theta$ is lower bounded. From part (ii), it follows that the sequence $\{\Theta(x^k)\}_{k\in\mathbb{N}}$ is decreasing. Thus, the sequence $\{\Theta(x^k)\}_{k\in\mathbb{N}}$ is convergent.
		
 \noindent
 {\bf(iv)} From \eqref{temp-ineq43}, $-\sum_{k=1}^{\infty}f_{k-1,j_{k-1}}(x^k)\ge 0$. In addition, for any $l\in\mathbb{N}$, we have
 \[
  -\sum_{k=0}^{l}f_{k,j_k}(x^{k+1})=\!\sum_{k=0}^{l}\big[\Theta(x^k)-\Theta_{k,j_k}(x^{k+1})\big]\le\Theta(x^0)-\Theta(x^{l+1}),
 \] 
  where the inequality is using $\Theta(x^{k+1})\le\Theta_{k,j_k}(x^{k+1})$ for each $k$. Along with the lower boundedness of $\Theta$ by the proof of part (iii), we obtain $-\sum_{k=0}^{l}f_{k,j_k}(x^{k+1})<\infty$. The two sides show that $0\le -\sum_{k=1}^{\infty}f_{k-1,j_{k-1}}(x^k)<\infty$.
		
 \noindent
  {\bf(v)} The first limit follows by parts (ii) and (iii). For the second limit, from the strong convexity of $\Theta_{k,j}$, the second inequality of \eqref{inexct-cond}, and the definitions of $x^{k+1}$ and $\overline{x}^{k+1}$, 
  \[
   (\underline{\gamma}/2)\interleave\!x^{k+1}\!-\overline{x}^{k+1}\interleave^2\le \Theta_{k,j_k}(x^{k,j_k})-\Theta_{k,j_k}(\overline{x}^{k,j_k})
 	\le -(\overline{\mu}/2)f_{k,j_k}(x^{k+1})\quad\forall k\in\mathbb{N}.
  \]
  This, along with $\lim_{k\to\infty}f_{k,j_k}(x^k)=0$ implied by part (iv), gives the second limit.    
 \end{proof}
	
 By means of Proposition \ref{descent-Theta}, we are ready to verify that every accumulation point of the sequence $\{x^k\}_{k\in\mathbb{N}}$ is a stationary point of problem \eqref{model}.
 \begin{proposition}\label{prop-subconverge}
  Let $\{x^k\}_{k\in\mathbb{N}}$ be the sequence generated by Algorithm \ref{iPMM}.  Suppose that the sequence $\{x^k\}_{k\in\mathbb{N}}$ is bounded. Denote by $\mathcal{X}^*$ the set of all accumulation points of $\{x^k\}_{k\in\mathbb{N}}$.  Then, the following statements are true.
  \begin{itemize}
  \item [(i)] The sequence $\{\gamma_k \}_{k\in\mathbb{N}}$ with $\gamma_{k}\!:=\gamma_{k,j_k}$  is bounded, i.e., there exists $\widetilde{\gamma}>0$ such that for all $k\in\mathbb{N}$, $\gamma_k\le\widetilde{\gamma}$. Consequently,  $\|\mathcal{Q}_k\|\le\varrho\widetilde{\gamma}$ for all $k\in\mathbb{N}$.
			
  \item [(ii)] The set $\mathcal{X}^*$ is nonempty and compact, $\Theta$ keeps the constant $\omega^*$ on the set $\mathcal{X}^*$, and $\mathcal{X}^*\subset\mathcal{S}^*$ where $\mathcal{S}^*$ denotes the stationary point set of problem \eqref{model}. 
  \end{itemize}
  \end{proposition}
  \begin{proof}
  {\bf(i)} Suppose on the contrary that the sequence $\{\gamma_k \}_{k\in\mathbb{N}}$ is unbounded. There exists an index set $\mathcal{K}\subset\mathbb{N}$ such that $\lim_{\mathcal{K}\ni k\to\infty}\gamma_k=\infty$. For each $k\in\mathcal{K}$, write $\widehat{\gamma}_k:=\varrho^{-1}\gamma_k=\gamma_{k,j_k-1}$. By step (2b) of Algorithm \ref{iPMM}, $\Theta(x^{k,j_k-1})>\Theta_{k,j_k-1}(x^{k,j_{k-1}})$. Along with the expression of $\Theta_{k,j_k-1}(x^{k,j_{k-1}})$ and $\mathcal{Q}_{k,j_{k-1}}\succeq\gamma_{k,j_{k-1}}\mathcal{I}=\widehat{\gamma}_k\mathcal{I}$, we have
  \begin{align}\label{aim-ineq44}
  \Theta(x^{k,j_k-1})&>\langle\nabla\vartheta(|F(x^k)|),|\ell_{F}(x^{k,j_k-1};x^k)|\rangle +\nu\psi(\ell_{T}(x^{k,j_k-1};x^k)) \nonumber\\
  &\quad\ +\lambda h(x^{k,j_k-1})+(\widehat{\gamma}_k/2)\!\interleave\!x^{k,j_k-1}-x^k\interleave^2+C^k.
 \end{align} 
 Next we claim that $\lim_{\mathcal{K}\ni k\to\infty}\interleave x^{k,j_k-1}\!-x^k\interleave=0$. Indeed, for any $k\in\mathcal{K}$, from the concavity of $\theta$ and $\theta(0)=0$, we deduce that $\vartheta$ is a concave function on $\mathbb{R}^{m\times n}$ and $\vartheta(0)=0$, which by the expression of $C^k$ implies that $C^k\ge 0$. Together with the expression of $\Theta_{k,j_k-1}$ and the nonnegativity of $\langle\nabla\vartheta(|F(x^k)|), |\ell_{F}(x;x^k)|\rangle$, we have
 \[ 
  \Theta_{k,j_k-1}(x^{k,j_k-1})-\nu\psi(\ell_T(x^{k,j_k-1};x^k))-\lambda h(x^{k,j_k-1})\ge\frac{1}{2}\interleave\!x^{k,j_k-1}\!-\!x^k\interleave_{\mathcal{Q}_{k,j_k-1}}^2
 \]
 which along with $\mathcal{Q}_{k,j_k-1}\succeq \gamma_{k,j_k-1}\mathcal{I}=\widehat{\gamma}_k\mathcal{I}$ and the first inequality of \eqref{inexct-cond} leads to 
 \begin{align*}
  ({\widehat{\gamma}_k}/{2})\!\interleave\!x^{k,j_k-1}\!-\!x^k\interleave^2 
  &\le \Theta_{k,j_k-1}(x^{k,j_k-1})-\nu\psi(\ell_T(x^{k,j_k-1};x^k))-\lambda h(x^{k,j_k-1})\\
  &\le \Theta(x^k)-\nu\psi(\ell_T(x^{k,j_k-1};x^k))-\lambda h(x^{k,j_k-1})\\
  & \le \Theta(x^0)-\nu\psi(\ell_T(x^{k,j_k-1};x^k))-\lambda h(x^{k,j_k-1}), 
 \end{align*}
  where the third inequality is due to Proposition \ref{descent-Theta} (ii). Recall that $h$ is lower bounded on its domain and $\psi$ is lower bounded by Assumption \ref{ass0}. Passing the limit $\mathcal{K}\ni k\to\infty$ and using $\lim_{\mathcal{K}\ni k \to \infty}\widehat{\gamma}_k=\infty$ leads to $\lim_{\mathcal{K}\ni k \to \infty}\interleave x^{k,j_k-1}-x^k\interleave=0$, and the claimed limit holds. Now, for each sufficiently large $k\in\mathcal{K}$, using inequalities \eqref{ineq1-majorF}-\eqref{ineq1-majorG} with $x=x^{k,j_k-1}$ yields that 
  \begin{align*}
  \vartheta(|F(x^{k,j_k-1})|)
  &\le\vartheta(|F(x^k)|)+
   \langle\nabla\vartheta(|F(x^k)|),|\ell_{F}(x^{k,j_k-1};x^k)|-|F(x^k)|\rangle \\
  &\quad+(1/2)\big[{\alpha_{1,k}}+\alpha_{2,k}\|F'(x^k)\|^2\big]\interleave x^{k,j_k-1}-x^k\interleave^2, \\
  \psi(T(x^{k,j_k-1}))&\le\psi(\ell_T(x^{k,j_k-1};x^k))+({\alpha_{3,k}}/{2}) \interleave x^{k,j_k-1}-x^k\interleave^2.
 \end{align*}
  Adding these two inequalities together, we have for all sufficiently large $k\in\mathcal{K}$, 
 \begin{align}\label{temp-ineq45}
  \Theta(x^{k,j_k-1})&\le \langle\nabla\vartheta(|F(x^k)|),|\ell_{F}(x^{k,j_k-1};x^k)|\rangle+\nu\psi(\ell_T(x^{k,j_k-1};x^k)) \\
  &\quad\ +\lambda h(x^{k,j_k-1})+({L_{k}}/{2})\interleave\!x^{k,j_k-1}-x^{k}\interleave^2+C^{k} \nonumber. 
  \end{align}
  For each $k\in\mathbb{N}$, write $M_k:={\rm lip}\,\vartheta(|F(x^k)|)\,{\rm lip}\,F'(x^k)+{\rm lip}\,\nabla\vartheta(|F(x^k)|)\|F'(x^k)\|^2+\nu{\rm lip}\,\psi(|T(x^k)|)\,{\rm lip}\,T'(x^k)$. Recall that $L_k$ is sufficiently close to $M_k$ from above. From the boundedness of $\{x^k\}_{k\in\mathbb{N}}$ and \cite[Theorem 9.2]{RW98}, the sequence $\{M_k\}_{k\in\mathbb{N}}$ is bounded, so is the sequence $\{L_{k}\}_{k\in\mathbb{N}}$. Recall that $\lim_{k\to\infty}\widehat{\gamma}_k=\infty$. Inequality \eqref{temp-ineq45} yields a contradiction to \eqref{aim-ineq44}.
		
 \noindent
 {\bf(ii)} As the sequence $\{x^k\}_{k\in\mathbb{N}}$ is bounded, the set $\mathcal{X}^*$ is nonempty. The compactness of $\mathcal{X}^*$ follows by noting that $\mathcal{X}^*=\cap_{l\in\mathbb{N}}\overline{\cup_{k\ge l}\{x^k\}}$, an intersection of compact sets. Pick any $x^*\in\mathcal{X}^*$. Then, there exists an index set $\mathcal{K}\subset\mathbb{N}$ such that $\lim_{\mathcal{K}\ni k\to\infty}x^k=x^*$. Along with Proposition \ref{descent-Theta} (v), we have $\lim_{\mathcal{K}\ni k\to\infty}x^{k-1}=x^*$ and $\lim_{\mathcal{K}\ni k\to\infty}\overline{x}^k=x^*$. For each $k\in\mathcal{K}$, from step (2b) of Algorithm \ref{iPMM} and the definitions of $x^k$ and $\overline{x}^k$,
 \begin{align*}
 \Theta(x^k)&\le \Theta_{k-1,j_{k-1}}(x^k)\le \Theta_{k-1,j_{k-1}}(\overline{x}^k)+(\overline{\mu}/2)(-f_{k-1,j_{k-1}}(x^k))\\
 &\le\langle\nabla\vartheta(|F(x^*)|), |\ell_{F}(x^*;x^{k-1})|\rangle +\nu\psi(\ell_T(x^*;x^{k-1}))+\lambda h(x^*)\\
 &\quad +(1/2)\interleave\!x^*-x^{k-1}\!\interleave_{\mathcal{Q}_{k}}^2+C^{k-1}+(\overline{\mu}/2)(-f_{k-1,j_{k-1}}(x^k))\\
 &\le\langle\nabla\vartheta(|F(x^*)|), |\ell_{F}(x^*;x^{k-1})|\rangle +\nu\psi(\ell_T(x^*;x^{k-1}))+ \lambda h(x^*)\\
  &\quad +(\widetilde{\gamma}/2)\interleave\!x^*-x^{k-1}\!\interleave^2+C^{k-1}+(\overline{\mu}/2)(-f_{k-1,j_{k-1}}(x^k)),
 \end{align*}
 where the second inequality is due to the second condition in \eqref{inexct-cond} and $\mu_k\le\overline{\mu}$, the third one is using the optimality of $\overline{x}^k=\overline{x}^{k-1,j_{k-1}}$ and the feasibility of $x^*$ to subproblem \eqref{subprobk}, and the last one is by $\|\mathcal{Q}_k\|\le\widetilde{\gamma}$. Passing the limit $\mathcal{K}\ni k\to\infty$ to the last inequality and using Proposition \ref{descent-Theta} (iii)-(iv) and $\lim_{\mathcal{K}\ni k\to\infty}x^{k-1}=x^*$ results in $\varpi^*\le \Theta(x^*)$. While from the lower semicontinuity of $\Theta$ and Proposition \ref{descent-Theta} (iii), we have $\varpi^*\ge\Theta(x^*)$. Thus, $\Theta(x^*)=\varpi^*$. To achieve the inclusion $\mathcal{X}^*\subset\mathcal{S}^*$, we only need to argue that $x^*\in\mathcal{S}^*$. Indeed, from the optimality of $\overline{x}^{k}$ to subproblem \eqref{subprobk} and Remark \ref{remark-alg} (c), it holds that
 \begin{align*}
  0 &\in\nabla F(x^k)\big(\nabla\vartheta(|F(x^{k-1})|)\circ\partial| \ell_{F}(\overline{x}^k;x^{k-1})|)+\nu\nabla T(x^{k-1})\partial \psi(\ell_T(\overline{x}^k;x^{k-1}))\nonumber\\
  &\quad +\lambda\partial h(\overline{x}^k)+ \mathcal{Q}_{k,j}(\overline{x}^k-x^{k-1}).
 \end{align*}
 Passing the limit $\mathcal{K}\ni k\to\infty$ to the above inclusion and using $\lim_{\mathcal{K}\ni k\to\infty}x^{k-1}=x^*$ and $\lim_{\mathcal{K}\ni k\to\infty}\overline{x}^k=x^*$, Lemma \ref{lemma-osc} and the outer semicontinuity of $\partial h$ results in  
 \[ 
  0\in \nabla F(x^*)\big(\nabla\vartheta(|F(x^*)|)\circ\partial|F(x^*)|)+\nu \nabla T(x^*) \partial \psi(T(x^*))+\lambda\partial h(x^*).  
 \] 
 By comparing with Definition \ref{def-Spoint}, $x^*$ is a stationary point of \eqref{model}, i.e., $x^*\in\mathcal{S}^*$.   
 \end{proof}
 \begin{remark}\label{remark-bound}
  Note that  $\{x^k\}_{k\in\mathbb{N}}\subset\mathcal{L}_{\Theta(x^0)}\!:=\{x\in\mathbb{X}\ |\ \Theta(x)\le\Theta(x^0)\}$ by Proposition \ref{descent-Theta} (ii), so the boundedness assumption of $\{x^k\}_{k\in\mathbb{N}}$ in Proposition \ref{prop-subconverge} is satisfied when the set $\mathcal{L}_{\Theta(x^0)}$ is bounded. This obviously holds if $h$ has a bounded level set. 
 \end{remark}    
 \subsection{Global convergence}\label{sec4.1}
	
 In this section, we will establish the convergence of the iterate sequence $\{x^k\}_{k\in\mathbb{N}}$ under the following assumption:
 \begin{aassumption}\label{ass1}
 \begin{enumerate}
 \item[(i)] the sequence $\{x^k\}_{k\in\mathbb{N}}$ given by Algorithm \ref{iPMM} is bounded;
 	
 \item[(ii)] $F$ and $T$ are twice continuously differentiable on the open set $\mathcal{O}$.
 \end{enumerate}
 \end{aassumption}
 
 As discussed in Remark \ref{remark-bound}, Assumption \ref{ass1} (i) is rather weak. Assumption \ref{ass1} (ii) requires the twice continuous differentiability of $F$ and $T$, but as illustrated in Section \ref{sec5}, it is satisfied by many image reconstruction models. 
  
 To prove the convergence of $\{x^k\}_{k\in\mathbb{N}}$, we define the following potential function
 \begin{equation}\label{Xifun}
 \Xi(w):=\vartheta(|\ell_F(x;z)|)+\nu\psi(\ell_T(x;z))+\lambda h(x)+({\widehat{\gamma}}/{2})\interleave x-z\interleave^2+t^2
 \end{equation}
 for any $w=(x,z,t)\in\mathbb{X}\times\mathbb{X}\times\mathbb{R}_{+}$, where  $\widehat{\gamma}=2\varrho\widetilde{\gamma}$ and $\widetilde{\gamma}$ is the constant from Proposition \ref{prop-subconverge} (i). The function $\Xi$ is closely related to the objective function $\Theta_{k,j}$ of \eqref{subprobk}. Indeed, from the concavity of $\vartheta$, for each $k\in\mathbb{N}$ and any $x\in\mathbb{X}$,
 \begin{equation*}
  \vartheta(|\ell_F(x;x^k)|)\le\vartheta(|F(x^k)|)+\langle\nabla\vartheta(|F(x^k)|),|\ell_{F}(x;x^k)|-|F(x^k)|\rangle.
 \end{equation*}
 Together with the expression of $\Theta_{k,j}$, for any $(x,t)\in\mathbb{X}\times\mathbb{R}_{+}$, it holds that 
 \begin{equation}\label{ineq-Xi}
 \Xi(x,x^k,t)\le \Theta_{k,j}(x)-\frac{1}{2}\interleave x-x^k\interleave_{\mathcal{Q}_k-\widehat{\gamma}\mathcal{I}}^2+t^2\quad\forall (x,t)\in\mathbb{X}\times\mathbb{R}_{+}.
 \end{equation} 
 The following lemma characterizes the subdifferential of the function $\Xi$ at any $w$.
 \begin{lemma} \label{subdiff-Xi}	
 Under Assumption \ref{ass1} (i), the function $\Xi$ defined in \eqref{Xifun} is regular and at any  $w=(x,z,t)\in\mathbb{X}\times\mathbb{X}\times\mathbb{R}_{+}$,   $\partial\Xi(w)=\mathcal{V}_1(x,z)\times\mathcal{V}_2(x,z)\times\{2t\}$ with 
 \begin{align*}
  \mathcal{V}_1(x,z)&=\nabla F(x)\big[\nabla\vartheta(|\ell_F(x;z)|)\circ\partial|\ell_F(x;z)|\big]+\nu \nabla T(z)\partial\psi(\ell_T(x;z))\\
  &\quad+\lambda\,\partial h(x)+\widehat{\gamma}(x-z),\\
   \mathcal{V}_2(x,z)&=[D^2F(z)(x-z,\cdot)]^{*}\big[\nabla\vartheta(|\ell_F(x;z)|)\circ\partial|\ell_F(x;z)|\big] \nonumber \\
  &\quad+\nu[D^2T(z)(x-z,\cdot)]^{*}\partial\psi(\ell_T(x;z))+\widehat{\gamma}(z-x).
 \end{align*} 
 \end{lemma}
 \begin{proof}
  Let $\widetilde{f}(x',z')\!:=\vartheta(|\ell_F(x';z')|)$ for $(x',z')\in\mathbb{X}\times\mathbb{X}$. Clearly, $\widetilde{f}=f\circ\ell_{F}(\cdot;\cdot)$, where $f$ is the same as in Lemma \ref{chain-sub}. From \cite[Theorem 10.6]{RW98} and the continuous differentiability of $\ell_{F}(\cdot,\cdot)$ by Assumption \ref{ass1} (iii), the function $\widetilde{f}$  is regular with 
  \[
   \partial\widetilde{f}(x,z)=\begin{pmatrix}
	\nabla F(x)\\ [D^2F(z)(x-z,\cdot)]^{*}   
	\end{pmatrix}\big[\nabla\vartheta(|\ell_F(x;z)|)\circ\partial|\ell_F(x;z)|\big]\quad\forall (x,z)\in\mathbb{X}\times\mathbb{X}.
  \]
  Let $g(x',z'):=\nu\psi(\ell_{T}(x';z'))$ for $(x',z')\in\mathbb{X}\times\mathbb{X}$. Recall that $\psi$ is a finite convex function. From \cite[Theorem 10.6]{RW98} and the continuous differentiability of $\ell_{T}(\cdot;\cdot)$ by Assumption \ref{ass1} (iii), the function $g$ is also regular with
  \[
   \partial g(x,z)=\begin{pmatrix}
			\nabla T(z)\\
			[D^2T(z)(x-z,\cdot)]^{*}             
		\end{pmatrix}\partial\psi(\ell_T(x;z))\quad\forall (x,z)\in\mathbb{X}\times\mathbb{X}.
  \]
  Combining the last two equations with the expression of $\Xi$ yields the second part. 
 \end{proof}
	
 In the sequel, write $w^k\!:=(\overline{x}^{k+1},x^{k},t^k)$ with $t^k=\!\sqrt{-(\mu_k/2)f_{k,j_k}(x^{k+1})}$ for each $k\in\!\mathbb{N}$, where $f_{k,j_k}$ is the function in \eqref{fkj}. The following proposition shows that $\Xi(w^k)$ is bounded by $\Theta(x^k)$, and ${\rm dist}(0,\partial\Xi(w^k))$ is upper bounded by $-f_{k,j_k}(x^{k+1})$.   
 \begin{proposition} \label{prop-Xi}
  Under Assumption \ref{ass1}, the following two assertions hold. 
  \begin{itemize}
   \item[(i)] There exists $\overline{k}_1\in\mathbb{N}$ such that for all $k\ge \overline{k}_1$, 
			\[
			\Theta(x^{k+1})\le\Xi(w^k)\le \Theta(x^k)-\underline{\gamma}^{-1}\widehat{\gamma}\min\big\{2,1+\overline{\mu}/2\big\}f_{k,j_k}(x^{k+1}).
			\]
			
  \item[(ii)] There exists $\alpha>0$ such that for each $k\ge \overline{k}_1$ there is $\Gamma^k\in\partial \Xi(w^k)$ satisfying
			\begin{equation} \label{subgap22}
			\|\Gamma^k\|_F\le \alpha\sqrt{-f_{k,j_k}(x^{k+1})}.
			\end{equation}
  \end{itemize}
 \end{proposition}
 \begin{proof}
 {\bf(i)} By the definition of $w^k$ and \eqref{ineq-Xi} with $x=\overline{x}^{k+1}$, for each $k\in\mathbb{N}$, 
 \begin{align*}
  \Xi(w^k)&\le \Theta_{k,j_k}(\overline{x}^{k+1})+(\widehat{\gamma}/2)\!\interleave\overline{x}^{k+1}-x^k\interleave^2-({\mu_k}/{2})f_{k,j_k}(x^{k+1})\nonumber\\
 &=\Theta_{k,j_k}(\overline{x}^{k+1})+(\widehat{\gamma}/2)\!\interleave\overline{x}^{k+1}\!-x^k\interleave^2-({\mu_k}/{2})(\Theta_{k,j_k}(x^{k+1})-\Theta(x^k)).
 \end{align*}
 By the optimality of $\overline{x}^{k+1}=\overline{x}^{k,j_k}$ and the feasibility of $x^{k+1}$ to subproblem \eqref{subprobk}, we have $\Theta_{k,j_k}(\overline{x}^{k+1})\le\Theta_{k,j_k}(x^{k+1})$. By Assumption \ref{ass1} (iii), $\lim_{k\to\infty}\mu_k=0$, so there exists $\overline{k}_1\in\mathbb{N}$ such that for all $k\ge\overline{k}_1$, $\mu_k<2$. Together with the above inequality, 
 \begin{align}\label{temp-ineq46}
  \Xi(w^k)			&\le(1-{\mu_k}/{2})\Theta_{k,j_k}(x^{k+1})+({\mu_k}/{2})\Theta(x^k)+(\widehat{\gamma}/2)\!\interleave\overline{x}^{k+1}\!-x^k\interleave^2\nonumber\\
 &\le\Theta(x^k)+(\widehat{\gamma}/2)\!\interleave\overline{x}^{k+1}\!-x^k\interleave^2\quad\forall k\ge\overline{k}_1,
 \end{align}
 where the second inequality is using $\Theta_{k,j_k}(x^{k+1})<\Theta(x^k)$ and $2-\mu_k\!>0$ for all $k\ge\overline{k}_1$. By using \eqref{temp-ineq42} with $j=j_k, \mathcal{Q}_k\succeq\underline{\gamma}\mathcal{I}$ and $\mu_k\in(0,\min\{2,\overline{\mu}\}]$ for all $k\ge\overline{k}_1$, 
 \begin{equation*} 
 \interleave\overline{x}^{k+1}\!-x^k\interleave^2 	\le-\underline{\gamma}^{-1}(2+\min\{2,\overline{\mu}\})f_{k,j_k}(x^{k+1}),
 \end{equation*}
 where the inequality also uses $-f_{k,j_k}(x^{k,j_k})\ge 0$ by Proposition \ref{descent-Theta} (i).  Together with the above \eqref{temp-ineq46}, the desired second inequality then follows. For the first inequality, from Proposition \ref{descent-Theta} (v), we have $\lim_{k\to\infty}\interleave\overline{x}^{k+1}\!-x^k\interleave=0$. Now by invoking Lemma \ref{local-lip1}, for all $k\ge \overline{k}_1$ (if necessary by increasing $\overline{k}_1$), we have  
 \begin{align*}
  \vartheta(|\ell_{F}(\overline{x}^{k+1};x^k)|) 	
  &\ge \vartheta(|F(x^k)|) + \langle\nabla\vartheta(|F(x^k)|), |\ell_{F}(\overline{x}^{k+1};x^k)|-|F(x^k)|\rangle\\
  &\quad-({\alpha_{2,k}}/{2})\|F'(x^k)\|^2\interleave\overline{x}^{k+1}\!-x^k\interleave^2,
 \end{align*}
 where $\alpha_{2,k}$ is sufficiently close to ${\rm lip}\nabla\vartheta(|F(x^k)|)$ from above. Together with the expression of $\Xi(w^k)$, for all $k\ge \overline{k}_1$, it holds that   
 \begin{align}\label{temp-ineq47}
  \Xi(w^k)&\ge\vartheta(|F(x^k)|) + \langle\nabla\vartheta(|F(x^k)|), |\ell_{F}(\overline{x}^{k+1};x^k)|-|F(x^k)|\rangle-\frac{\mu_k}{2}f_{k,j_k}(x^{k+1})\nonumber\\
  &\quad -\frac{\alpha_{2,k}}{2}\|F'(x^k)\|^2\!\interleave\overline{x}^{k+1}-x^k\interleave^2 +\nu\psi(\ell_T(\overline{x}^{k+1};x^k))+\frac{\widehat{\gamma}}{2}\!\interleave\overline{x}^{k+1}-x^k\interleave^2\nonumber\\
  &=\Theta_{k,j_k}(\overline{x}^{k+1})-\frac{\mu_k}{2}f_{k,j_k}(x^{k+1})+\frac{\widehat{\gamma}}{2}\!\interleave\overline{x}^{k+1}-x^k\interleave^2-\frac{1}{2}\!\interleave\overline{x}^{k+1}-x^k\interleave_{\mathcal{Q}_{k}}^2 \nonumber\\
  &\quad-({\alpha_{2,k}}/{2})\|F'(x^k)\|^2\!\interleave\overline{x}^{k+1}-x^k\interleave^2\nonumber\\
  &\ge \Theta_{k,j_k}(x^{k+1})+\frac{\widehat{\gamma}-\alpha_{2,k}\|F'(x^k)\|^2}{2}\!\interleave\overline{x}^{k+1}\!-x^k\interleave^2-\frac{1}{2}\interleave\overline{x}^{k+1}-x^k\interleave_{\mathcal{Q}_{k}}^2,
 \end{align}
 where the second inequality is using the second condition of \eqref{inexct-cond}. By Assumption \ref{ass1} (ii) and \cite[Theorem 9.2]{RW98}, the sequence ${\rm lip}\nabla\vartheta(|F(x^k)|)$ is bounded. Recall that $\alpha_{2,k}$ is close enough to ${\rm lip}\nabla\vartheta(|F(x^k)|)$ from above, so is bounded. Thus, if necessary by increasing $\widetilde{\gamma}$ and $\overline{k}_1$, we have $\alpha_{2,k}\|F'(x^k)\|^2\le\widehat{\gamma}/2$; while by Proposition \ref{prop-subconverge} (i), $\|\mathcal{Q}_{k}\|\le\widehat{\gamma}/2$. Thus, for any $k\ge\overline{k}_1$,
 \[
  \big[\widehat{\gamma}-\alpha_{2,k}\|F'(x^k)\|^2\big]\!\interleave\overline{x}^{k+1}\!-x^k\interleave^2-\interleave\overline{x}^{k+1}-x^k\interleave_{\mathcal{Q}_{k}}^2\ge 0. 
 \]
  Together with the above \eqref{temp-ineq47} and $\Theta_{k,j_k}(x^{k+1})\ge\Theta(x^{k+1})$, we get $\Xi(w^k)\ge\Theta(x^{k+1})$. 
		
 \noindent
 {\bf(ii)} For each $k\in\mathbb{N}$, as $\overline{x}^{k+1}$ is the unique optimal solution of subproblem \eqref{subprobk} for $j=j_k$, from Remark \ref{remark-alg} (c) it follows that
 \begin{align*}
  0&\in\nabla F(x^k)\big[\nabla\vartheta(|F(x^k)|)\circ\partial|\ell_{F}(\overline{x}^{k+1};x^k)|\big]+\nu \nabla T(x^k)\partial\psi(\ell_T(\overline{x}^{k+1};x^k))\nonumber\\
  &\quad\ +\lambda\partial h(\overline{x}^{k+1})+\mathcal{Q}_{k}(\overline{x}^{k+1}-x^k).
 \end{align*}
  Then, for each $k\in\mathbb{N}$, there exists an element $v^k\in\partial|\ell_{F}(\overline{x}^{k+1};x^k)|$ such that 
  \begin{align*}
  \Gamma_1^k&:=\nabla F(x^k)\big[(\nabla\vartheta(|F(x^k)|)-\nabla\vartheta(|\ell_F(\overline{x}^{k+1};x^k)|))\circ v^k\big]+(\widehat{\gamma}\mathcal{I}-\mathcal{Q}_k)(\overline{x}^{k+1}\!-\!x^k)\\
  &\in\nabla F(x^k)\big[\nabla\vartheta(|\ell_F(\overline{x}^{k+1};x^k)|)\circ\partial|\ell_{F}(\overline{x}^{k+1};x^k)|\big]+\lambda\,\partial h(\overline{x}^{k+1})\nonumber\\
  &\quad+\nu \nabla T(x^k)\partial\psi(\ell_T(\overline{x}^{k+1};x^k))+\widehat{\gamma}(\overline{x}^{k+1}-x^k).
  \end{align*}
  In addition, for each $k\in\mathbb{N}$, pick any $\xi^k\in\partial\psi(\ell_T(\overline{x}^{k+1};x^k))$ and write
  \begin{align*}
  \Gamma_2^k&:=[D^2F(x^k)(\overline{x}^{k+1}\!-\!x^k,\cdot)]^*[\nabla\vartheta(|F(x^k)|)\circ v^k]+\widehat{\gamma}(x^k-\overline{x}^{k+1})\\
  &\quad\ +\nu [D^2T(x^k)(\overline{x}^{k+1}\!-\!x^k,\cdot)]^*\xi^k.
  \end{align*}
  By comparing with Lemma \ref{subdiff-Xi}, we have 	$\Gamma^k\!:=(\Gamma_1^k;\Gamma_2^k;2t^k)\in\partial\Xi(w^k)$ for each $k\in\mathbb{N}$. By Assumption \ref{ass1} (ii) and Proposition \ref{descent-Theta} (v), the sequence  $\{\overline{x}^k\}_{k\in\mathbb{N}}$ is bounded, which implies that $\{\ell_{F}(\overline{x}^{k+1};x^k)\}_{k\in\mathbb{N}}$ and $\{\ell_{T}(\overline{x}^{k+1};x^k)\}_{k\in\mathbb{N}}$ are bounded. Together with \cite[Theorem 9.13 \& Proposition 5.15]{RW98}, there exists a constant $\alpha_1\ge 0$ such that 
  \[
   \partial|\ell_{F}(\overline{x}^{k+1};x^k)|\subset \alpha_1\mathbb{B}\ \ {\rm and}\ \  \partial\|\ell_{T}(\overline{x}^{k+1};x^k)\|_1\subset \alpha_1\mathbb{B}, 
  \]
  where $\mathbb{B}$ is the unit ball of $\mathbb{R}^{m\times n}$ centered at the origin. Note that the linear operators $D^2F(x^k)(\cdot,\cdot)$ and $D^2T(x^k)(\cdot,\cdot)$ are bounded. Hence, if necessary by increasing $\alpha_1$, 
  \begin{equation}\label{Gam1-ineq}
  \|\Gamma_2^k\|_F\le \alpha_1\!\interleave\overline{x}^{k+1}-x^k\interleave\quad{\rm for\ all}\  k\in\mathbb{N}. 
  \end{equation}
  Recall that $\lim_{k\to\infty}\interleave\overline{x}^{k+1}-x^k\interleave=0$ and $\nabla\vartheta$ is strictly continuous on an open set containing $\mathbb{R}_{+}^{m\times n}$. Then, there exists a constant $\alpha_2>0$ such that for all $k\ge\overline{k}_1$ (if necessary by increasing $\overline{k}_1$), 
  \begin{align*}
   \|\nabla\vartheta(|F(x^k)|)-\nabla\vartheta(|\ell_F(\overline{x}^{k+1};x^k)|)\|_F
   &\le \alpha_2\|F(x^k)-\ell_F(\overline{x}^{k+1};x^k)\|_F\\
   &\le \alpha_2\|F'(x^k)\|\interleave\overline{x}^{k+1}\!-x^k\interleave. 
  \end{align*} 
  By the boundedness of $\{F'(x^k)\}_{k\in\mathbb{N}}$, if necessary by increasing  $\alpha_2$, for all $k\ge\overline{k}_1$,
  \begin{equation*}
  \big\|\nabla F(x^k)\big[(\nabla\vartheta(|F(x^k)|)-\nabla\vartheta(|\ell_F(\overline{x}^{k+1};x^k)|))\circ v^k\big]\big\|_F\le \alpha_2\interleave\overline{x}^{k+1}-x^k\interleave,
  \end{equation*}
  which implies that $\|\Gamma_1^k\|_F\le(2\widehat{\gamma}+\alpha_2)\interleave\overline{x}^{k+1}\!-x^k\interleave$. Together with \eqref{Gam1-ineq} and the definition of $\Gamma^k$,
  for all $k\ge\overline{k}_1$, it holds that 
  \begin{equation*}
  \|\Gamma^k\|_F\le \sqrt{\alpha_1^2+(2\widehat{\gamma}+\alpha_2)^2}\interleave\overline{x}^{k+1}\!-x^k\interleave+2t^k
  \le \alpha_3\sqrt{-f_{k,j_k}(x^{k+1})}+2t^k
  \end{equation*}
  with $\alpha_3=\frac{2+\overline{\mu}}{\gamma}\sqrt{\alpha_1^2+(2\widehat{\gamma}+\alpha_2)^2}$, where the second inequality is due to \eqref{temp-ineq42} and $f_{k,j_k}(x^{k+1})<0$. By recalling $t_k=\!\sqrt{-(\mu_k/2)f_{k,j_k}(x^{k+1})}$ and $\mu_k\le\overline{\mu}$, the desired result holds with $\alpha=\alpha_3+\sqrt{2\overline{\mu}}$. The proof is completed. 
 \end{proof}   
 
 Now we are ready to establish the convergence of $\{x^k\}_{k\in\mathbb{N}}$ under the KL property.  
 \begin{theorem}\label{convergence}
 Suppose that $\Xi$ is a KL function, and that Assumption \ref{ass1} holds. Then, $\sum_{k=1}^{\infty}\interleave x^{k+1}-x^k\interleave<\infty$ and hence $\{x^k\}_{k\in\mathbb{N}}$ is convergent with limit $x^*\in\mathcal{S}^*$.  	
\end{theorem} 
\begin{proof}
 If there exists some $\overline{k}\ge\overline{k}_1$ such that $\Theta(x^{\overline{k}})=\varpi^*$, where $\overline{k}_1$ is the same as in Proposition \ref{prop-Xi}, then by Proposition \ref{descent-Theta} (iii) we have $\Theta(x^{k})=\varpi^*$ for all $k\ge\overline{k}$. This along with Proposition \ref{descent-Theta} (ii) implies that $x^k=x^{k+1}$ for all $k\ge\overline{k}$, and the conclusion then follows. Hence, it suffices to consider the case that $\Theta(x^{k})>\varpi^*$ for all $k\ge\overline{k}$. By Proposition \ref{prop-subconverge} (ii), the accumulation point set $\mathcal{X}^*$ of $\{x^k\}_{k\in\mathbb{N}}$ is nonempty and compact, and for any $x\in\mathcal{X}^*$, $\Theta(x)=\varpi^*=\lim_{k\to\infty}\Theta(x^k)$. Write $ \Omega^*:=\big\{(x,x,0)\ |\ x\in\mathcal{X}^*\big\}$. Then, for any $w=(x,x,0)\in\Omega^*$,  $\Xi(w)=\Theta(x)=\varpi^*$. Since $\Xi$ is assumed to be a KL function, by invoking \cite[Lemma 6]{Bolte14}, there exist $\epsilon>0,\eta>0$ and $\varphi\in\Upsilon_{\!\eta}$ such that for all $w\in[\varpi^*<\Xi<\varpi^*+\eta]\cap\{w\in\mathbb{X}\times \mathbb{X}\times\mathbb{R}_{+}\ |\ {\rm dist}(w,\Omega^*)<\epsilon\}$, 
 \begin{equation}\label{ineq41-KL}
 \varphi'(\Xi(w)-\varpi^*){\rm dist}(0,\partial \Xi(w))\ge 1.
 \end{equation}
 By invoking Proposition \ref{prop-Xi} (i) and Proposition \ref{descent-Theta} (iii), for all $k\ge\overline{k}$, it holds that 
 \begin{equation}\label{ineq42-KL}
 \varpi^*<\Theta(x^k)\le \Xi(w^{k-1})\le \Theta(x^{k-1})+\underline{\gamma}^{-1}\widehat{\gamma}\min\{2,1+\overline{\mu}/2\}f_{k-1,j_{k-1}}(x^{k})<\varpi^*+\eta, 
 \end{equation}
 where the last inequality is due to $\lim_{k\to\infty}\Theta(x^k)=\varpi^*$ and $\lim_{k\to\infty}f_{k-1,j_{k-1}}(x^k)=0$. From Proposition \ref{descent-Theta} (v), $\lim_{ k\to\infty}\interleave\overline{x}^{k}-x^k\interleave=0$, so ${\rm dist}(\overline{x}^k,\mathcal{X}^*)<{\epsilon}/{3}$ for all $k\ge\overline{k}$ (if necessary by increasing $\overline{k}$). Recall that $t^k=\sqrt{-(\mu_k/2)f_{k,j_k}(x^{k+1})}$ for each $k$. We have $\lim_{k\to\infty}t^k=0$, so $|t^{k-1}|<{\epsilon}/{6}$ for all $k\ge\overline{k}$ (if necessary by increasing $\overline{k}$). Then, 
 \begin{equation*}
  {\rm dist}(w^{k-1},\Omega^*)\le {\rm dist}(\overline{x}^k,\mathcal{X}^*)+{\rm dist}(x^{k-1},\mathcal{X}^*)+2|t^{k-1}|<\epsilon\quad\forall k\ge\overline{k}. 
 \end{equation*}
 Together with the above \eqref{ineq41-KL}-\eqref{ineq42-KL}, $\varphi'(\Xi(w^{k-1})-\varpi^*){\rm dist}(0,\partial \Xi(w^{k-1}))\ge 1$. Since $\varphi$ is concave and continuously differentiable on $(0,\eta)$, we have that $\varphi'$ is nonincreasing on $(0,\eta)$. Then, $\varphi'(\Theta(x^k)-\varpi^*)\ge \varphi'(\Xi(w^{k-1})-\varpi^*)$. Thus, for all $k\ge\overline{k}$,
 \begin{equation}\label{KL-ineq}
  \varphi'(\Theta(x^k)-\varpi^*){\rm dist}(0,\partial \Xi(w^{k-1}))\ge 1,
 \end{equation}
 which by Proposition \ref{prop-Xi} (ii) implies that $\varphi'(\Theta(x^k)-\varpi^*)\ge \frac{1}{\alpha\sqrt{-f_{k-1,j_{k-1}}(x^{k})}}$ for all $k\ge\overline{k}$. Now using the concavity of the function $\varphi$ yields that for all $k\ge\overline{k}$,  
  \begin{align*}
  \Delta_k&:=\varphi(\Theta(x^k)-\varpi^*)-\varphi(\Theta(x^{k+1})-\varpi^*)
  \ge \varphi'(\Theta(x^k)-\varpi^*)(\Theta(x^k)-\Theta(x^{k+1})) \nonumber \\
  &\ge \frac{\Theta(x^k)-\Theta(x^{k+1})}{\alpha\sqrt{-f_{k-1,j_{k-1}}(x^{k})}}\ge\frac{-f_{k,j_{k}}(x^{k+1})}{\alpha\sqrt{-f_{k-1,j_{k-1}}(x^{k})}},
 \end{align*}
 where the last inequality is from \eqref{eq-descend}. This is equivalent to saying that for all $k\ge\overline{k}$,
 \[ 
 -f_{k,j_{k}}(x^{k+1})\le \alpha\Delta_k\sqrt{-f_{k-1,j_{k-1}}(x^{k})}.
 \] 
 By using the basic inequality $2\sqrt{ab}\le a+b$ with $a=\alpha\Delta_k$ and $b=\sqrt{-f_{k-1,j_{k-1}}(x^{k})}$,
 \begin{equation*}
 2\sqrt{-f_{k,j_{k}}(x^{k+1})}\le\alpha\Delta_k +\sqrt{-f_{k-1,j_{k-1}}(x^{k})}
 \quad{\rm for\ all}\ k\ge\overline{k}.
 \end{equation*}
 Summing the above inequality from $k\ge\overline{k}$ to any $l>k$ immediately results in 
 \begin{align*}
  \sum_{i=k}^{l}\sqrt{-f_{i,j_{i}}(x^{i+1})} 
  &\le\sum_{i=k}^{l}\alpha\Delta_k+\sqrt{-f_{k-1,j_{k-1}}(x^{k})}\\
  &\le \alpha \varphi(\Theta(x^k)-\varpi^*) +\sqrt{-f_{k-1,j_{k-1}}(x^{k})}, 
 \end{align*}
 where the second inequality also uses the nonnegativity of $\varphi$.
 Passing the limit $l\to\infty$ to the both sides of the above inequality  and using Proposition \ref{descent-Theta} (i) leads to $\sum_{i=k}^{\infty}\interleave x^{i+1}-x^i\interleave<\infty$.
 Therefore, $\{x^k\}_{k\in\mathbb{N}}$ is a convergent sequence and converges to a point $x^*\in\mathcal{X}^*\subset\mathcal{S}^*$. The proof is completed.	
 \end{proof} 

 The KL assumption on $\Xi$ in Theorem \ref{convergence} is mild. As will be shown in Section \ref{sec5}, for the image reconstruction models there, the associated $\Xi$ has the KL property. When $\Xi$ has the KL property of exponent $\theta\in(0,1)$, inequality \eqref{KL-ineq} holds for all $k\ge\overline{k}$ with $\varphi(t)=ct^{1-\theta}$ for some $c>0$, which by Proposition \ref{prop-Xi} (ii) and \eqref{eq-descend} implies that $\Delta^k\le[\alpha(1\!-\!\theta)]^{\frac{1}{\theta}} (\Delta^{k-1}-\Delta^k)^{\frac{1}{2\theta}}$ for all $k\ge\overline{k}$.  
 By this recursive formula, using the same arguments as those in \cite[Theorem 2]{Attouch09} leads to the following result. 
 \begin{corollary}\label{convergence-rate}
  Suppose that $\Xi$ is a KL function of exponent $\theta\in(0,1)$, and that Assumption \ref{ass1} holds. Then, $\{x^k\}_{k\in\mathbb{N}}$ converges to a point $x^*\in\mathcal{S}^*$, and moreover,
  \begin{itemize}
  \item [(i)] for $\theta\in(0,\frac{1}{2}]$, there exist $\gamma>0$ and $\varrho\in(0,1)$ such that $\Theta(x^k)-\Theta(x^*)\le \gamma\varrho^k$ for all $k$ large enough;
   
  \item [(ii)] for $\theta\in(\frac{1}{2},1)$, there exists $\gamma>0$ such that $\Theta(x^k)-\Theta(x^*)\le\gamma k^{-\frac{1}{2\theta-1}}$ for all $k$ large enough. 
 \end{itemize} 	
 \end{corollary}
\section{Applications to image reconstruction}\label{sec5}

 In this section, we shall apply Algorithm \ref{iPMM} to solve several classes of specific nonconvex composite optimization problems arising from image deblurring and inpainting. It is well known that isotropic and anisotropic TV functions play a crucial role in image deblurring and inpainting. Let $\mathcal{D}_h\!:\mathbb{R}^{m\times n}\to\mathbb{R}^{m\times n}$ and $\mathcal{D}_v\!:\mathbb{R}^{m\times n}\to\mathbb{R}^{m\times n}$ be the horizontal and vertical 2-D finite difference operators, respectively, and write $\mathcal{D}(X):=[\mathcal{D}_h(X);\mathcal{D}_v(X)]$ for $X\in\mathbb{R}^{m\times n}$. The isotropic and anisotropic TV functions have the form $\psi(\mathcal{D}(\cdot))$ with 
 \begin{subnumcases}{}\label{TV-fun}
  \psi(z)\!:=\sum_{i=1}^{m}\sum_{j=1}^{n}\sqrt{x_{i,j}^2+y_{i,j}^2}\quad{\rm for}\  z=(x,y)\in\mathbb{R}^{m\times n}\times\mathbb{R}^{m\times n},\\
  \psi(z)\!:=\sum_{i=1}^{m}\sum_{j=1}^{n}\sqrt{x_{i,j}^2+y_{i,j}^2}\quad{\rm for}\  z=(x,y)\in\mathbb{R}^{m\times n}\times\mathbb{R}^{m\times n}.
  \label{aTV-fun}
 \end{subnumcases}
 Unless otherwise stated, $\psi$ in the rest of this section means the one in \eqref{TV-fun} or \eqref{aTV-fun}. 
 \subsection{Gray image deblurring}\label{sec5.1}

 Consider the following TV regularized composite optimization model for image deblurring
 \begin{equation}\label{Xmodel}
 \min_{x\in\mathbb{R}^{m\times n}}\vartheta(|\mathcal{A}(x)-b|)+\nu\psi(\mathcal{D}(x)) +\mathbb{I}_{\Lambda}(x),
\end{equation}
 where $\mathcal{A}\!:\mathbb{R}^{m\times n}\to\mathbb{R}^{m\times n}$ is a blur operator, $b\in\mathbb{R}^{m\times n}$ denotes a noisy gray image, and $\Lambda\!:=\!\{x\in\mathbb{R}^{m\times n}\,|\, x_{ij}\in[l_{ij},u_{ij}]\ {\rm for\ all}\ (i,j)\in[m]\times[n]\}$ with $l,u\in\mathbb{R}_{+}^{m\times n}$ is a matrix box set. It is clear that model \eqref{Xmodel} is a special case of \eqref{model} with $\mathbb{X}=\mathbb{R}^{m\times n}$, $F(\cdot)=\mathcal{A}(\cdot)-b$, $T(\cdot)=\mathcal{D}(\cdot)$ and $h(\cdot)=\mathbb{I}_{\Lambda}(\cdot)$. When $\vartheta$ is the matrix function induced by $\theta_1$ in Table \ref{tab0}, model \eqref{Xmodel} reduces to the convex model considered in \cite{Chan13}; when $\vartheta$ is the matrix function induced by $\theta_{4}$ and $\Lambda$ is the whole space $\mathbb{R}^{m\times n}$, this model reduces to the one considered in \cite{Zhang17}. From \cite[Section 4]{Attouch10} and the discussion after Definition \ref{Def-KL}, the potential function $\Xi$ associated to model \eqref{Xmodel} with $\vartheta$ induced by one of  $\theta_1$-$\theta_5$ in Table \ref{tab0} is a KL function. Our algorithm is applicable to model \eqref{Xmodel} with $\vartheta$ induced by a nonconvex $\theta\in\mathscr{L}$. 

 Next we take a look at the subproblem \eqref{subprobk} when Algorithm \ref{iPMM} is applied to solve problem \eqref{Xmodel}, and provide its dual problem. To this end, write $\mathcal{C}\!:=\mathcal{A}\times(\nu\mathcal{D})$ and $\overline{b}=(b;0)\in\mathbb{R}^{3m\times n}$. For each $k\in\mathbb{N}$ and $j\in\{0\}\cup[j_k]$, we take $\mathcal{Q}_{k,j}\!:=\gamma_{k,j}\mathcal{I}\!+\!\alpha_k\mathcal{C}^*\mathcal{C}$, where $\alpha_k>0$ will be specified in numerical experiments, and let
$\omega^k\!:=[\vartheta'(|F(x^k)|);E]$ where $E$ is a $2m\times n$ matrix of all ones. Then problem \eqref{subprobk} is equivalently written as 
\begin{equation*}
 \min_{x\in\mathbb{R}^{m\times n}}f\big(\mathcal{C}(x)-\overline{b}\big)+ \mathbb{I}_{\Lambda}(x)+\frac{\gamma_{k,j}}{2}\|x-x^k\|_F^2+\frac{\alpha_k}{2}\big\|\mathcal{C}(x-x^k)\big\|_F^2+C^k,
\end{equation*}
 where $f(y,w)\!:=\|\omega^k\circ y\|_1+\psi(w)$ for $(y,w)\in\mathbb{R}^{m\times n}\times(\mathbb{R}^{m\times n}\times\mathbb{R}^{m\times n})$. By introducing a new variable $z\!\in\mathbb{R}^{3m\times n}$, this minimization problem can be reformulated as 
 \begin{align} \label{Esubprobk1}
 &\min_{x,z}\,f(z)+\mathbb{I}_{\Lambda}(x) +\frac{\gamma_{k,j}}{2}\|x-x^k\|_F^2+\frac{\alpha_k}{2}\|z-c^k\|_F^2 +C^k\nonumber \\
 &\ {\rm s.t.}\ \ \mathcal{C}(x)-\overline{b}=z\quad\ \ {\rm with}\ c^k\!:=[F(x^k);\nu T(x^k)].
 \end{align}
An elementary calculation yields that the dual of \eqref{Esubprobk1} has the following form
\begin{align}\label{dsubprobk1}
 \min_{\xi\in\mathbb{R}^{3m\times n}} \Phi_k(\xi)&=\frac{1}{2\alpha_k}\|\xi\|_F^2-e_{\alpha^{-1}_k}f(c^k\!+\!\alpha_k^{-1}\xi)+\frac{1}{2\gamma_{k,j}}\|\mathcal{C}^*(\xi)\|_F^2\nonumber\\
 &\qquad-\frac{\gamma_{k,j}}{2}\|\Pi_{\Lambda}(X^k\!-\!\gamma_{k,j}^{-1}\mathcal{C}^*(\xi))-x^k\|_F^2 -C^k
\end{align}
 where, for a proper lsc function $g\!:\mathbb{X}\to\overline{\mathbb{R}}$ and a parameter $\gamma>0$, $e_{\gamma}g$ and $\mathcal{P}_{\!\gamma}g$ respectively denote the Moreau envelope and proximal mapping of $g$ associated to $\gamma$:
 \[
  e_{\gamma}g(x)\!:=\min_{z\in\mathbb{X}}\Big\{\frac{1}{2\gamma}\|z-x\|_F^2+g(z)\Big\}\ \ {\rm and}\ \ 
  \mathcal{P}_{\!\gamma}g(x)\!:=\mathop{\arg\min}_{z\in\mathbb{X}}\Big\{\frac{1}{2\gamma}\|z-x\|_F^2+g(z)\Big\}.
 \]
 Due to the strong convexity of problem \eqref{Esubprobk1}, there is no dual gap between \eqref{Esubprobk1} and its dual \eqref{dsubprobk1}. Therefore, one can achieve the unique optimal solution $(x^*,z^*)$ of \eqref{Esubprobk1} by solving the dual problem \eqref{dsubprobk1}, which is an unconstrained convex minimization with the $C^{1,1}$ objective function $\Phi_k$ (i.e., $\Phi_k$ is continuously differentiable and $\nabla\Phi_k$ is globally Lipschitz continuous). After an elementary calculation, for any $\xi\in\mathbb{R}^{3m\times n}$,     
 \[
  \nabla\Phi_k(\xi)=\mathcal{P}_{\!\alpha_{k}^{-1}}f\big(c^k\!+\!\alpha_{k}^{-1}\xi\big)+\overline{b}-\mathcal{C}\Pi_{\Lambda}(x^k\!-\!\gamma_{k,j}^{-1}\mathcal{C}^*(\xi)).
 \]
 Clearly, if $\xi^*$ is an optimal solution of \eqref{dsubprobk1}, then $(x^*,z^*)$ with $x^*\!=\Pi_{\Lambda}(x^k-\gamma_{k,j}^{-1}\mathcal{C}^*(\xi^*))$ and $ z^*=\mathcal{P}_{\!\alpha_{k}^{-1}}f\big(c^k\!+\!\alpha_k^{-1}\xi^*\big)$ is the unique optimal solution of problem \eqref{Esubprobk1}. 
 \subsection{Color image inpainting}\label{sec5.2} 
 
 Image inpainting can be modeled as a matrix completion problem by assuming that image data is of low rank or approximate low rank. Many low rank plus TV regularization models have been proposed for image inpainting (see, e.g., \cite{SunLi22,Qiu21,Qiu23}), in which the nuclear norm or a nonconvex surrogate of the rank function are used as the low rank regularization term. As is well known, such regularization terms require at least one SVD of a full matrix in each step of algorithms for solving the corresponding models. To avoid the computation cost of SVDs, a popular approach is to enforce the low-rank property explicitly by using the factorization of matrix variables (see, e.g., \cite{Sun16,Tao22, Li20}). Motvited by these works, we consider a novel matrix completion model for image inpainting. 
 
 Fix an integer $r\in[1,\min\{m,n\}]$ and write $\mathbb{X}_{r}\!:=\mathbb{R}^{m\times r}\times \mathbb{R}^{n\times r}$. Let $\Omega\subset[m]\times[n]$ be an index set, and denote by $\mathcal{P}_{\Omega}\!:\mathbb{R}^{m\times n}\to \mathbb{R}^{m\times n}$ the projection operator on $\Omega$, i.e., $[\mathcal{P}_{\Omega}(x)]_{i,j}=x_{i,j}$ if $(i,j)\in\Omega$, otherwise $[\mathcal{P}_{\Omega}(x)]_{i,j}=0$. We are interested in model
 \begin{equation}\label{UVmodel}
 \min_{(U,V)\in\mathbb{X}_{r}}\vartheta(|\mathcal{P}_{\Omega}(UV^{\top})-b|) +\nu\psi(\mathcal{D}(UV^{\top}))+\lambda(\|U\|_{2,1}+\|V\|_{2,1}),
 \end{equation}
 where $b$ is a noisy observation matrix, and $\lambda(\|U\|_{2,1}+\|V\|_{2,1})$ is used as a low rank regularization term. Obviously, \eqref{UVmodel} is a special case of model \eqref{model} with $\mathbb{X}=\mathbb{X}_r$ and $ F(x):=\mathcal{P}_{\Omega}(UV^{\top})-b, T(x):=\mathcal{D}(UV^{\top})$ and $h(x):=\|U\|_{2,1}+\|V\|_{2,1}$ for $x=(U,V)\in\mathbb{X}_r$. From \cite[Section 4]{Attouch10} and the discussion after Definition \ref{Def-KL}, the potential function $\Xi$ associated to \eqref{UVmodel} with $\vartheta$ induced by one of  $\theta_1$-$\theta_5$ is also a KL function. 
 
 To present the subproblem \eqref{subprobk} when Algorithm \ref{iPMM} is applied to solve problem \eqref{UVmodel}, for each $k\in\mathbb{N}$, we write $x^k\!:=\!(U^k,V^k)\in\!\mathbb{X}_r$ and $\mathcal{C}_k:=\mathcal{A}_k\times\mathcal{B}_k$ with  $\mathcal{A}_k\!:=F'(x^k)$ and $\mathcal{B}_k\!:=\nu T'(x^k)$. By the expressions of $F$ and $T$, for any $(G,H)\in\mathbb{X}_r$, 
 \begin{equation}\label{ABk}
 \mathcal{A}_k(G,H):=\mathcal{A}\big(G(V^k)^{\top}\!+\!U^kH^{\top}\big)\ \ {\rm and}\ \ 
 \mathcal{B}_k(G,H):=\nu\mathcal{D}\big(G(V^k)^{\top}\!+\!U^kH^{\top}\big).
 \end{equation}
 For each $k\in\mathbb{N}$ and $j\in\{0\}\cup[j_k]$, take $\mathcal{Q}_{k,j}\!:=\gamma_{k,j}\mathcal{I}\!+\!\alpha_k\mathcal{C}_k^*\mathcal{C}_k$, where $\alpha_k>0$ is specified in numerical experiments. Then, the subproblem \eqref{subprobk} is equivalently written as  
 \begin{equation*}
  \min_{x\in\mathbb{X}_r}\big\|\omega^k\circ(\mathcal{C}_k(x)-b^k)\big\|_1+\lambda h(x)+\frac{\gamma_{k,j}}{2}\|x-x^k\|_F^2+\frac{\alpha_k}{2}\big\|\mathcal{C}_k(x-x^k)\big\|_F^2+C^k
 \end{equation*}
 with $\omega^k\!:=\![\nabla\vartheta(|F(x^k)|);E]$ and $b^k\!:=\![\mathcal{A}_k(x^k)\!-\!F(x^k);\mathcal{B}_k(x^k)-\!\nu T(x^k)]$. 
 By introducing a variable $z\in\mathbb{R}^{3m\times n}$, this minimization problem is reformulated as 
 \begin{align} \label{Esubprobk}
  &\min_{x,z}\,f(z)+\lambda h(x) +\frac{\gamma_{k,j}}{2}\|x-x^k\|_F^2+\frac{\alpha_k}{2}\|z-c^k\|_F^2 +C^k,\nonumber \\
  &\ \ {\rm s.t.}\ \ \mathcal{C}_k(x)-b^k =z\quad{\rm with}\  c^k\!:=[F(x^k);\nu T(x^k)] 
 \end{align}
 with $f(z)\!:=\|\omega^k\circ z\|_1$ for $z\!\in\mathbb{R}^{3m\times n}$. A simple calculation yields the dual of \eqref{Esubprobk} as 
 \begin{align}\label{dsubprobk}
  \min_{\xi\in\mathbb{R}^{3m\times n}} \Phi_k(\xi)&=\frac{1}{2\alpha_k}\|\xi\|_F^2-e_{\alpha^{-1}_k}f(c^k+\alpha_k^{-1}\xi)+\frac{1}{2\gamma_{k,j}}\|\mathcal{C}_k^*(\xi)\|_F^2\nonumber\\
  &\qquad -\lambda e_{\lambda\gamma_{k,j}^{-1}}h\big(x^k\!-\!\gamma_{k,j}^{-1}\mathcal{C}^*_k(\xi)\big)-C^k. 
 \end{align}
 Similar to Section \ref{sec5.1}, one can achieve the unique optimal solution of $(x^*,z^*)$ of \eqref{Esubprobk} by solving its dual \eqref{dsubprobk}, an unconstrained $C^{1,1}$ convex minimization with
 \[ 
 \nabla\Phi_k(\xi)=\mathcal{P}_{\!\alpha_{k}^{-1}}f\big(c^k\!+\!\alpha_{k}^{-1}\xi\big)+Y^k-\mathcal{C}_k\big[\mathcal{P}_{\lambda\gamma_{k,j}^{-1}}h\big(x^k\!-\!\gamma_{k,j}^{-1}\mathcal{C}^*_k(\xi)\big)\big]\quad\forall \xi\in\mathbb{R}^{3m\times n}.
 \]
 If $\xi^*$ is an optimal solution of \eqref{dsubprobk}, then $(x^*,z^*)$ with $x^*\!=\mathcal{P}_{\lambda\gamma_{k,j}^{-1}}h\big(x^k\!-\!\gamma_{k,j}^{-1}\mathcal{C}^*_k(\xi^*)\big)$ and $z^*=\mathcal{P}_{\!\alpha_{k}^{-1}}f\big(c^k\!+\!\alpha_{k}^{-1}\xi\big)$ is the unique optimal solution of \eqref{Esubprobk}.
 
 \section{Numerical experiments}\label{sec6}
 
 We validate the efficiency of Algorithm \ref{iPMM} by applying it to solve models \eqref{Xmodel} and \eqref{UVmodel}. From Sections \ref{sec5.1} and \ref{sec5.2}, an inexact minimizer of subproblem \eqref{subprobk} can be achieved by solving an unconstrained $C^{1,1}$ convex minimization problem \eqref{dsubprobk1} or \eqref{dsubprobk}. Note that the Lipschitz constant of their gradient function $\nabla\Phi_k$ depends on $\alpha_k^{-1}$ and $\gamma_{k,j}^{-1}$, which will become larger as the parameters $\alpha_k$ and $\gamma_{k,j}$ decrease. This means that the minimization problem \eqref{dsubprobk1} or \eqref{dsubprobk}  becomes much harder as the parameters $\alpha_k$ and $\gamma_{k,j}$ decrease. Inspired by this, we apply the limited-memory BFGS \cite{Liu89,Nocedal06} to solve a regularized version of problems \eqref{dsubprobk1} and \eqref{dsubprobk}, i.e., the following strongly convex minimization
 \begin{equation}\label{iterate-PPA}
 	\min_{\xi\in\mathbb{R}^{3m\times n}}\Psi_{k}(\xi):=\Phi_k(\xi)+(\tau_k/2)\|\xi-\xi^{k-1}\|^2,
 \end{equation}
 where $\xi^{k-1}$ is the final iterate for minimizing $\Psi_{k-1}$ with $\xi^0=0$, and $\tau_k>0$ is a regularization parameter. We compare the performance of Algorithm \ref{iPMM} armed with the limited-memory BFGS for solving \eqref{iterate-PPA}, abbreviated as iPMM-lbfgs, with that of two state-of-art solvers. The quality of solutions or recovered images is measured in terms of the peak signal to noise ratio (PSNR) and the structural similarity (SSIM) \cite{Wang04}. All tests are performed in Matlab 2020 on a desktop running on 64-bit Windows System with an Intel(R) Core(TM) i7-6800K CPU 3.40GHZ and 16 GB RAM. 
 \subsection{Tests for image deblurring with impluse noise}\label{sec6.1} 
 
 We test the performance of iPMM-lbfgs to solve model \eqref{Xmodel} with $\vartheta$ being the function associated to $\theta_4$ in Table \ref{tab0} and $\psi$ being the TV function in \eqref{TV-fun}. For numerical simulation, the linear operator $\mathcal{A}$ is generated by the average blur with $7\times7$ testing kernel and Gaussian blur with $9\times9$ testing kernel (standard deviation $2$) from Matlab 2020, and the observation matrix $b$ is generated by imposing salt-and-pepper impulse noise on some gray images, including Boat ($512\times512$), Cameraman ($256\times256$), House ($256\times256$), Man ($512\times512$), Building ($360\times360$), Parrot ($256\times256$), which are shown in Figure \ref{grayiamge} below. For the subsequent tests, we consider four kinds of noise levels $30\%,50\%,70\%,90\%$. 
 \begin{figure}[h] 
 	\centering  
 	\subfigure[Boat]{
 		\includegraphics[width=2.8cm]{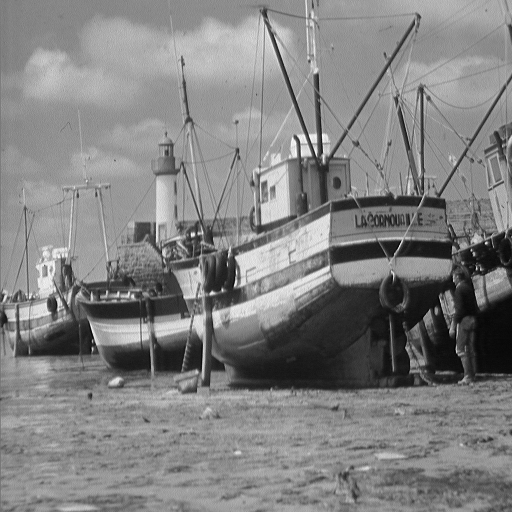}} \quad
 	\subfigure[Cameraman]{
 		\includegraphics[width=2.8cm]{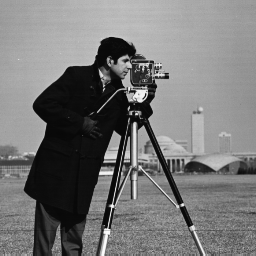}}\quad
 	\subfigure[House]{
 		\includegraphics[width=2.8cm]{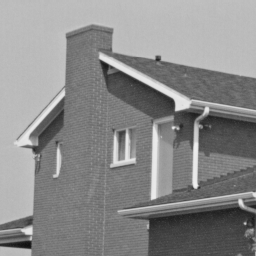}}\quad
 	\subfigure[Man]{
 		\includegraphics[width=2.8cm]{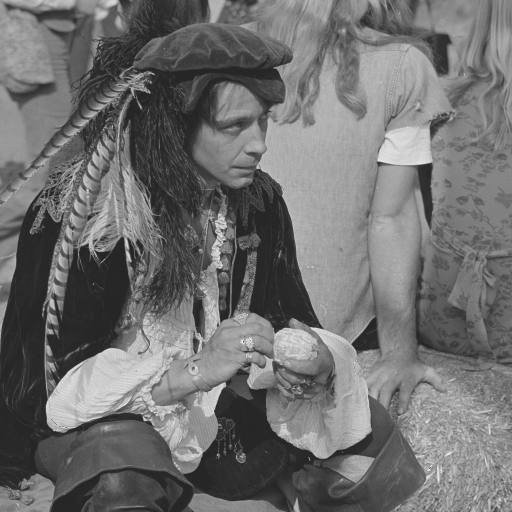}}\quad
 	\subfigure[Building]{
 		\includegraphics[width=2.8cm]{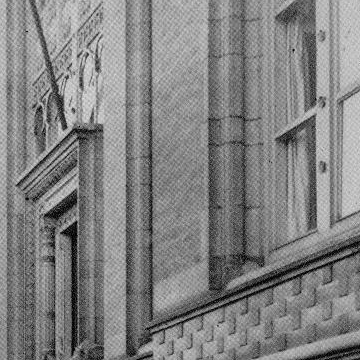}}\quad
 	\subfigure[Parrot]{
 		\includegraphics[width=2.8cm]{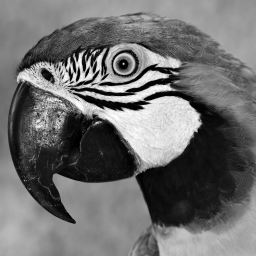}}\quad
 	\caption{The original gray images used for deblurring tests} \label{grayiamge}
 \end{figure} 
  We compare the performance of iPMM-lbfgs with that of the proximal linearized minimization algorithm (PLMA) proposed in \cite{Zhang17} for solving 
 \begin{equation}\label{Xmodel1}
 	\min_{x\in\mathbb{R}^{m\times n}}\gamma\vartheta(|\mathcal{A}(x)-b|)+\psi(\mathcal{D}(x)),
 \end{equation}
 which is actually model \eqref{Xmodel} with $\Lambda=\mathbb{R}^{m\times n}$ and $\gamma=\nu^{-1}$. Each step of PLMA is seeking $(x^{k+1},u^{k+1})$ such that $u^{k+1}\in\partial Q(x^{k+1}|x^k)$ with $\|u^{k+1}\|_F\le\rho\epsilon\|x^{k+1}\!-x^k\|_F$ for some $\rho>0$ and $\epsilon\in(0,\frac{1}{2})$, where $Q(\cdot|x^k)$ is the majorization of the objective function of \eqref{Xmodel1} at $x^k$:
 \[
 Q(x|x^k):=\gamma\langle\nabla\vartheta(|\mathcal{A}(x^k)-b|),|\mathcal{A}(x)-b|\rangle+\psi(\mathcal{D}(x))+\frac{\rho}{2}\|x-x^k\|_F^2\quad\forall x\in\mathbb{R}^{m\times n}. 
 \]
 In the implementation of PLMA, the alternating direction multiplier method (ADMM) with a fixed penalty parameter is used to seek such $(x^{k+1},u^{k+1})$, and the details are seen in the code \url{https://github.com/xjzhang008/Nonconvex-TV/archive/master.zip}. During the subsequent tests, the parameters of PLMA are set as the default ones. 
 
 According to the discussion in \cite[Section 2.1]{Zhang17} for the influence of $\varepsilon$ in $\theta_4$ on model \eqref{Xmodel1}, we choose $\theta_4$ with $\varepsilon=90$ for \eqref{Xmodel} and \eqref{Xmodel1} in the subsequent tests. 
 \subsubsection{Influence of regularization parameter $\nu$ on PSNR}\label{sec6.1.1}
  We take Cameraman for example to check the influence of $\nu$ on PSNR. Figure \ref{fig1} below plots the PSNR curves under two noise levels with $\mathcal{A}$ generated randomly as above. We see that the PSNR has more variation as $\nu$ varies under a high noise level. In view of this, our numerical results in Section \ref{sec6.1.3} will report the interval of $\nu$ corresponding to better PSNR for the tested images. Under the $90\%$ noise level, the PSNR for the Gaussian blur $\mathcal{A}$ is a little higher than the one for the average blur $\mathcal{A}$, but under the $30\%$ noise level, the PSNR for the former is much lower than the one for the latter.
  
 \begin{figure}[H]
 	\centering
 	\subfigure[\label{fig1a} noise level $90\%$]{\includegraphics[width=0.45\textwidth]{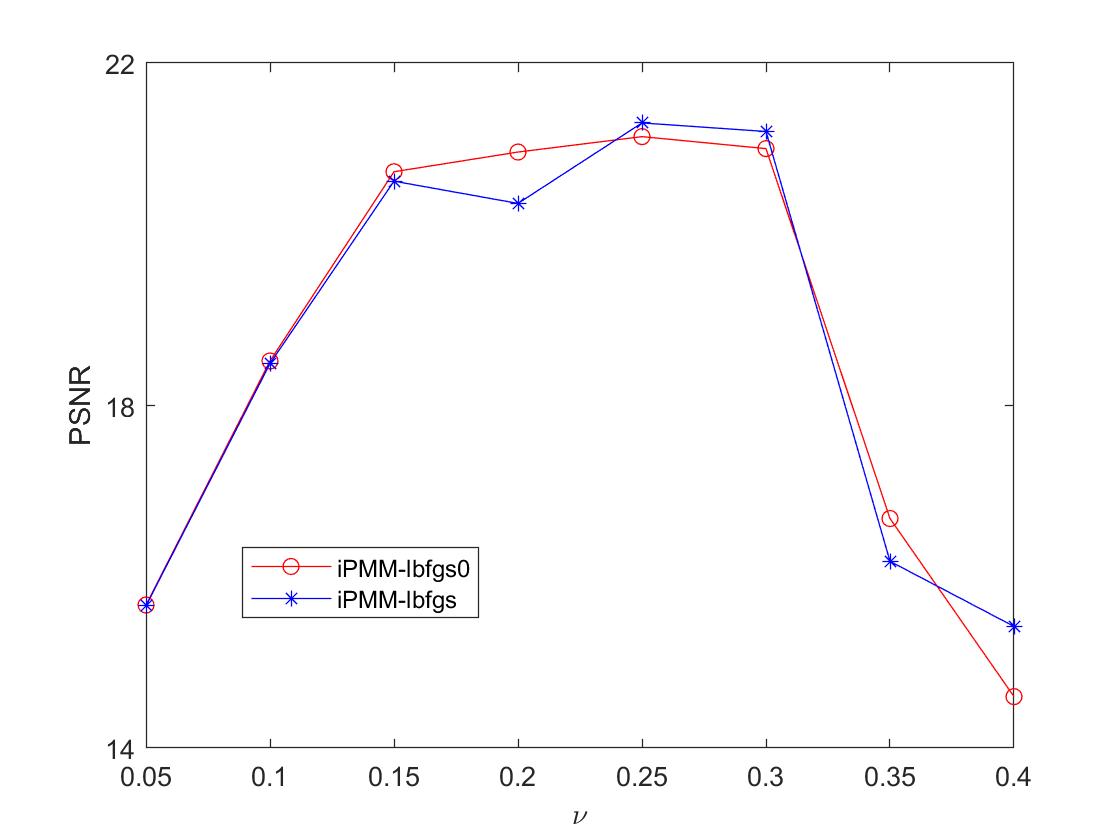}}
 	\subfigure[\label{fig1b} noise level $30\%$]{\includegraphics[width=0.45\textwidth]{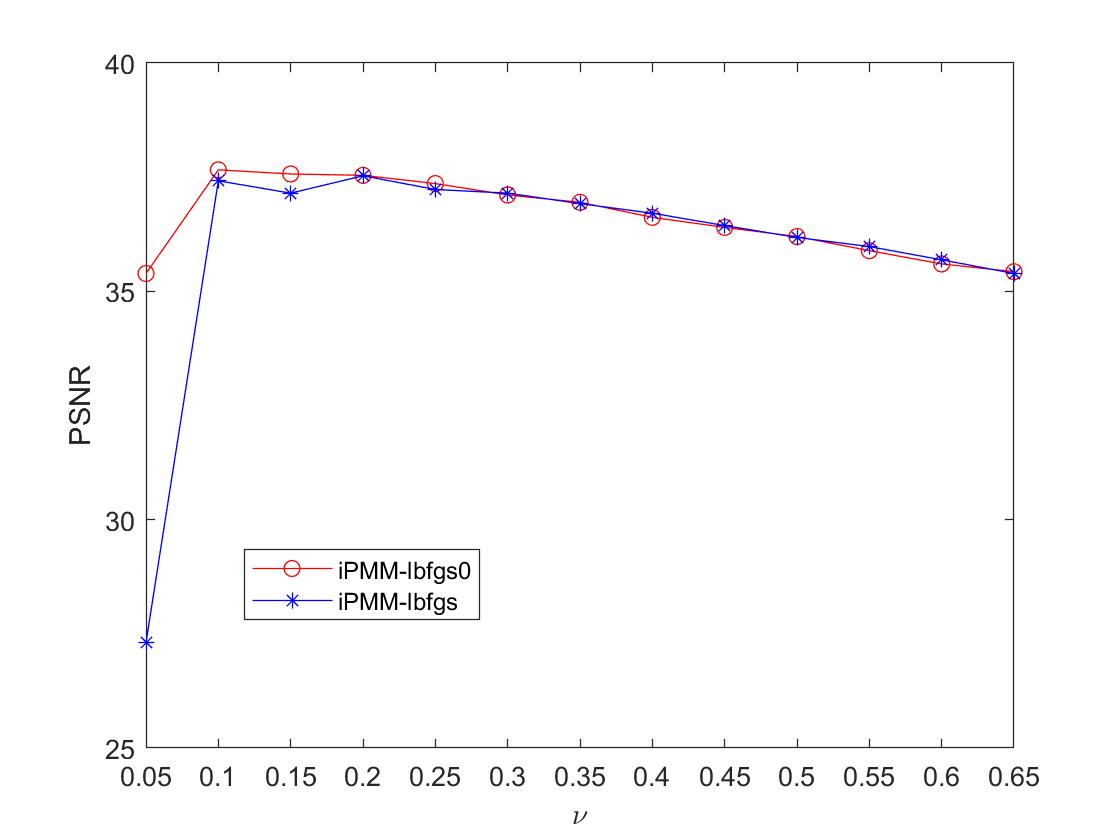}}
 	\setlength{\abovecaptionskip}{10pt}
 	\caption{Influence of parameter $\nu$ on PSNR under two noise levels}
 	\label{fig1}
 \end{figure}
 \subsubsection{Implementation details of iPMM-lbfgs}\label{sec6.1.2}
 
 We first take a look at the choice of parameters involved in Algorithm \ref{iPMM}. 
 As mentioned in Section \ref{sec5.1},  $\mathcal{Q}_{k,j}=\gamma_{k,j}\mathcal{I}+\alpha_k\mathcal{C}^*\mathcal{C}$. We update the parameter $\alpha_k$ by the following formula
 \begin{equation}\label{alphak-update}
 	\alpha_{k+1}=\left\{\begin{array}{cl}
 		\!\max\{\frac{\alpha_k}{1.05},10^{-3}\}&{\rm if}\ {\rm mod}(k,3)=0,\\
 		\alpha_k &{\rm otherwise}
 	\end{array}\right.\ {\rm with}\ \alpha_0=\min\big\{\nu^{-1}\rho_{\nu},50\big\}.
 \end{equation}
 Among others, $\rho_{\nu}$ is chosen to be $5,2.5,2/3$ for the noise levels $30\%,50\%,70\%$, respectively, and for the noise level $90\%$, $\rho_{\nu}=0.1$ and $1/2$ are respectively used for $\mathcal{A}$ generated by the average blur with $7\times 7$ testing kernel and the Gaussian blur with $9\times 9$ testing kernel. The other parameters of Algorithm \ref{iPMM} are chosen as follows
 \begin{equation}\label{other-para}
 \varrho=2,\,\underline{\gamma}=\alpha_0,\,\overline{\gamma}=10^6,\,\mu_k=\overline{\mu}/{k^{2.1}}\ {\rm with}\ \overline{\mu}=10^{10}.
 \end{equation}
 Such a sequence $\{\mu_k\}_{k\in\mathbb{N}}$ satisfies Assumption \ref{ass1} (iii). For the parameter $\tau_k$ in \eqref{iterate-PPA}, we update it by  $\tau_{k+1}=\max\{\tau_k/\rho_{\tau},\epsilon^*\}$ with $\tau_0=\min\{\underline{\gamma},10\}$, where $\rho_{\tau}=1.2$ if $\Theta(x^0)>10^5$, otherwise $\rho_{\tau}=1.15$, and $\epsilon^*>0$ is the tolerance of the stopping condition. The starting point $x^0$ is always chosen to be $b$.  
 
 We terminate Algorithm \ref{iPMM} whenever one of the following conditions is satisfied:
 \begin{equation}\label{stopping-cond}
 \frac{\|x^{k+1}-x^k\|_F}{1+\|b\|_F}\le\epsilon^*,\ \tau_k\le\epsilon^*\ \ {\rm and}\ \
 	\frac{|\Theta(x^k)-\max_{1\le j\le 9}\Theta(x^{k-j})|}{\max\{1,\Theta(x^k)\}}\le 10^{-5},
 \end{equation} 
 where $\epsilon^*=\min\{10^{-6},10^{-6}/\underline{\gamma}\}$ if $\Theta(x^0)>10^5$, otherwise $\epsilon^*=\min\{10^{-8},10^{-6}/\underline{\gamma}\}$. 
 For the stopping condition of LBFGS, let $\xi^{l}$ be the current iterate of LBFGS and write $x^{k,j}=\Pi_{\Lambda}(x^k-\!\gamma_{k,j}^{-1}\mathcal{C}^*(\xi^l))$. By the weak duality theorem, $\Phi_k(\xi^{l})\le \Theta_{k,j}(\overline{x}^k)$, so we terminate LBFGS at the iterate $\xi^{l}$ when
 $\Theta_{k,j}(x^{k,j})<\Theta(x^k)$ and $ \Theta_{k,j}(x^{k,j})-\Phi_k(\xi^{l})\le\frac{\mu_k}{2}[\Theta(x^k)-\Theta_{k,j}(x^{k,j})]$,
 and now $x^{k,j}$ is an inexact minimizer of \eqref{subprobk} satisfying the inexactness condition \eqref{inexct-cond}. For the subsequent tests, we choose the number of memory of LBFGS to be ${\bf 10}$ and the maximum number of iterates to be ${\bf 50}$.   
   
 \subsubsection{Numerical results for image deblurring}\label{sec6.1.3}
 We report the numerical results of iPMM-lbfgs for solving \eqref{Xmodel} where the box set $\Lambda$ is defined with $l_{ij}\equiv 0$ and $u_{ij}\equiv 1$. Since the PLMA in \cite{Zhang17} is only applicable to model \eqref{Xmodel} with $\Lambda=\mathbb{R}^{m\times n}$, for comparison, we also report the numerical results of iPMM-lbfgs for solving model \eqref{Xmodel} without $\mathbb{I}_{\Lambda}$. From the results reported in \cite[Section 5]{Zhang17}, we observe that PLMA for solving \eqref{Xmodel1} with the suggested $\gamma$ is superior to L1TV \cite{Yang09} and L1Nonconvex \cite{Nikolova13}, so here we do not compare the performace of iPMM-lbfgs with that of the latter two methods. For model \eqref{Xmodel} with and without the indicator function $\mathbb{I}_{\Lambda}$, we use the same $\nu$ whose value is listed in the $\nu$ column of Table \ref{tab1}-\ref{tab2}, and the interval in $\nu_{\rm box}$ column is the range of $\nu$ for better PSNR. The results of PLMA are obtained by applying PLMA to model \eqref{Xmodel1} with a carefully selected $\gamma$ from the range suggested in \cite{Zhang17}, i.e., the selected $\gamma$ makes the PSNRs yielded by PLAM best as much as possible.  All numerical results in Tables \ref{tab1}-\ref{tab2} are the average of those obtained for ${\bf 5}$ test instances generated randomly, and the best results are shown in boldface. 
 \begin{table}[H]
 \caption{\small PSNR and SSIM for images corrupted by average blur and impulse noise}\label{tab1}
 \centering
 \renewcommand\arraystretch{1.2}
 \scalebox{1}[1]{
 \begin{tabular}{@{\!}cc|ccc|cccccc@{}}
 \toprule
 Image &  Nlevel  &  & PLMA  &   &\multicolumn{2}{c}{iPMM} & &\multicolumn{2}{c}{${\rm iPMM}_{\rm box}$}& \\
 \midrule\midrule
 \multirow{5}{*}{Boat} 
  &   &   PSNR  & SSIM  & $\gamma$  & PSNR & SSIM & PSNR & SSIM & $\nu$ & $\nu_{\rm box}$  \\
 \cline{3-11}
  &30\%  &37.81 &0.9975 & 50 &37.77 & 0.9975  &\bf{38.08} & \bf{0.9977} &0.15&[0.1,0.4] \\
  \cline{2-11}
  &50\% &33.95 &0.9939  & 20  &34.63 &0.9949   &\bf{34.79} &\bf{0.9951} &0.4&[0.2,0.55]   \\
  \cline{2-11}
  &70\%  &{\bf31.56} &{\bf 0.9896}& 4 & 31.31 & 0.9890 & 31.11&0.9885 &0.4&[0.25,0.6]  \\
  \cline{2-11}
  &90\% & 25.58 & 0.9577 & 10  &{\bf 25.84} & {\bf0.9603} &25.74 &0.9595 &0.15&[0.15,0.35] \\
  \midrule\hline
  \multirow{4}{*}{Cameraman}
  &30\% &36.62 & 0.9979 &50 & 36.76 &0.9981  & {\bf37.56} &{\bf0.9985} & 0.15&[0.1,0.35] \\
  \cline{2-11}
  &50\% &\bf{32.91} & {\bf0.9957} &10 &32.55 &0.9953  &32.78 &  0.9956 & 0.4&[0.3,0.55]   \\
  \cline{2-11}
  &70\% &{\bf28.50}&{\bf 0.9881}  &{4} & 27.23 & 0.9842 & 27.75&0.9859&0.4&[0.25,0.55]   \\
  \cline{2-11}
  &{90\%} & 17.30&0.7926 &{8} & 18.24&0.8683 &{\bf 19.09}&{\bf 0.8874}&0.15&[0.15,0.3]   \\
  \midrule\hline
  \multirow{4}{*}{House}
  &{30\%}  &\bf{43.90}&\bf{0.9994} & {50} &42.55  &0.9992  &42.70&0.9992&0.15&[0.08,0.4] \\
  \cline{2-11}
  &{50\%} &\bf{39.44} &\bf{0.9983}& {5} &38.45 &0.9978  &38.57&0.9979&0.4&[0.15,0.5]    \\
  \cline{2-11}
  &{70\%}  &34.66 &0.9948 & {2}  &35.10 & 0.9953  &{\bf 35.15}&{\bf 0.9954} &0.4&[0.2,0.5]    \\
  \cline{2-11}
  &{90\%}  &27.87  &0.9741 &{20}  &30.20 &0.9853   &{\bf 30.22}&\bf{0.9853}&0.15&[0.06,0.25]   \\
  \midrule\hline
  \multirow{4}{*}{Man}
  &{30\%}  &\bf{39.14}&\bf{0.9983} &{50}  &38.49  &0.9980  & 38.74 & 0.9981&0.15&[0.12,0.35]   \\
  \cline{2-11}
  &{50\%}  &\bf{35.24}&\bf{0.9957}  &{10}  &34.71&0.9952   & 34.75&0.9952 &0.4&[0.15,0.5]    \\
  \cline{2-11}
  &{70\%} &31.50  &0.9898 &{2}  &31.71 & 0.9904 & \bf{31.74}&{\bf 0.9904} &0.4&[0.2,0.5]  \\
  \cline{2-11}
  &{90\%}  &{\bf28.11} &{\bf0.9776} &{25} &28.04 &0.9772& 28.07 &0.9774&0.15&[0.15,0.3]   \\
  \midrule\hline
  \multirow{4}{*}{Building}
  &{30\%} &30.72&0.9815 &{50}  &30.65&0.9813  &\bf{30.77} &\bf{0.9818} &0.15&[0.12,0.4]     \\
  \cline{2-11}
  &{50\%}  &{\bf27.64} &{\bf0.9617}&{20}  &27.53 &0.9607 &27.57  &0.9606&0.4&[0.2,0.5]    \\
  \cline{2-11}
  &{70\%}  &25.38 &0.9340 &{2}  & 25.42& 0.9347 &{\bf 25.43}&{\bf 0.9348}&0.4&[0.2,0.5]    \\
  \cline{2-11}
  &{90\%}   &23.12 &0.8829 &{25}  &\bf{23.38}&\bf{0.8914}   &23.37&0.8911&0.15&[0.05,0.3] \\
  \midrule\hline
  \multirow{4}{*}{Parrot}
  &{30\%} &34.30 &0.9972 &{20}  &35.34 &0.9977 & \bf{36.37}&\bf{0.9983}&0.15&[0.12,0.4]   \\
  \cline{2-11}
  &{50\%}  & 31.52&0.9947 &{5} & 31.53  &0.9948 & \bf{31.63} &\bf{0.9949}&0.4&[0.25,0.5]     \\
  \cline{2-11}
  &{70\%}  &23.64 &0.9633 &{6}  &26.61&0.9833   &\bf{27.39}&{\bf 0.9861}&0.4&[0.2,0.45]   \\
  \cline{2-11}
  &{90\%}  & 16.18 &0.7686 &{15} &16.39 &0.8062 &\bf{16.63}&{\bf 0.8133}&0.15&[0.15,0.3]  \\
  \midrule\hline
   \cline{1-11}
  average& & 30.69 &0.9636&  &30.85  &0.9698  &{\bf31.08} &{\bf0.9712} && \\
  \bottomrule
  \end{tabular}
	}
\end{table}
 \begin{table}[H]
 \caption{PSNR and SSIM for images corrupted by Gaussian blur and impulse noise}\label{tab2}%
 \centering
 \renewcommand\arraystretch{1.2}
 \scalebox{1}[1]{
 \begin{tabular}{@{}cc|ccc|cccccc@{}}
  \toprule
  Image &  Nlevel  &  & PLMA  &   &\multicolumn{2}{c}{iPMM} & &\multicolumn{2}{c}{${\rm iPMM}_{\rm box}$}& \\
  \midrule\midrule
  \multirow{5}{*}{Boat} 
  &   &   PSNR  & SSIM  & $\gamma$  & PSNR & SSIM & PSNR & SSIM & $\nu$ & $\nu_{\rm box}$  \\
  \cline{3-11}
  &{30\%}  &{\bf 36.57}&{\bf 0.9968} &{50}& 34.88&0.9952 &33.84 &0.9955&0.12&[0.12,0.3] \\
  \cline{2-11}
  &{50\%}  &{\bf34.52}  &{\bf0.9948}&{10}  &33.37 &0.9932 & 33.03& 0.9936 &0.12&[0.1,0.2] \\
  \cline{2-11}
  &{70\%} & 29.45& 0.9831&{4}  &{\bf 30.72}& {\bf0.9874} & 29.95& 0.9847&0.15&[0.12,0.3]  \\
  \cline{2-11}
  &{90\%}  &25.77&0.9592 &{10}  &{\bf26.16}&{\bf0.9637} & 26.02&0.9627&0.15&[0.15,0.25] \\
  \hline\hline
  \multirow{4}{*}{Cameraman}
  &{30\%} &{\bf33.38}&{\bf0.9962} &{50} &31.99 &0.9947 &32.26 &0.9950 &0.12&[0.08,0.3] \\
  \cline{2-11}
  &{50\%} &{\bf31.26} &{\bf0.9937} &{10} &30.56 &0.9927 &30.88 &0.9932&0.12&[0.1,0.2] \\
  \cline{2-11}
  &{70\%}  & 26.49 &0.9810 &{4}  & 26.93&0.9830 & {\bf27.41} &{\bf0.9848}&0.15&[0.12,0.25]   \\
  \cline{2-11}
  &{90\%}  & 16.72 &0.7821&{16}  &20.56& 0.9198  & {\bf20.69}&{\bf0.9246}&0.15&[0.12,0.25]  \\
  \midrule\hline
  \multirow{4}{*}{House}
  &{30\%}  &{\bf39.75}&{\bf0.9984}  &{50} &37.98 &0.9975   &38.00 &0.9971 &0.12& [0.05,0.2]   \\
  \cline{2-11}
  &{50\%} &{\bf36.27} &{\bf0.9964} &{10} &35.64&0.9958 & 35.31 &0.9955&0.12&[0.09,0.25]   \\
  \cline{2-11}
  &{70\%} &32.49&0.9914 &{4}  & {\bf 33.34} &{\bf 0.9930}&{\bf 33.34}&{\bf0.9930}&0.15&[0.05,0.2]   \\
  \cline{2-11}
  &{90\%}  &27.61  &0.9728&{20}  &29.78 &0.9837  &{\bf29.83} &{\bf0.9839}&0.15&[0.05,0.18]  \\
  \midrule\hline
  \multirow{4}{*}{Man}
  &{30\%} &{\bf36.36}  &{\bf0.9967}&{50} & 34.87&0.9954  &33.66 & 0.9939&0.12&[0.12,0.3]    \\
  \cline{2-11}
  &{50\%} &33.41 &0.9934 &{10} & 33.39&0.9935  &{\bf 33.64}&{\bf 0.9939}&0.12&[0.1,0.2]   \\
  \cline{2-11}
  &{70\%} &30.17 &0.9862 &{4} &{\bf31.33}&{\bf0.9895} & 31.32 &{\bf0.9895}&0.15&[0.1,0.2]   \\
  \cline{2-11}
  &{90\%} &{\bf 28.44}&{\bf0.9793} &{16} &28.39 &0.9791 & 28.43&{\bf0.9793}&0.15&[0.12,0.25] \\
  \midrule\hline
  \multirow{4}{*}{Building}
  &{30\%} &{\bf28.68} &{\bf0.9699} &{50}  &27.84&0.9633  &27.21&0.9576&0.12&[0.1,0.2]   \\
  \cline{2-11}
  &{50\%} &26.40 & 0.9482 &{10} &{\bf26.65} &{\bf0.9514} & 26.54 &0.9501 &0.12&[0.08,0.2]   \\
  \cline{2-11}
  &{70\%} &24.51 & 0.9186 &{4} &25.06&0.9290  &{\bf25.09}&{\bf0.9294}&0.15&[0.05,0.2] \\
  \cline{2-11}
  &{90\%} &23.42 &0.8919  &{16}  &23.52 &0.8958  &{\bf23.58 } &{\bf0.9873}&0.15&[0.05,0.25]\\
 			\midrule\hline
 			\multirow{4}{*}{Parrot}
 &{30\%} &{\bf33.44} &{\bf0.9966}&{50}  &32.04  &0.9954  & 32.30 &0.9956 &0.12&[0.08,0.3] \\
 			\cline{2-11}
 &{50\%} &{\bf31.29} &{\bf0.9945} &{10}  &29.73  &0.9924 & 30.76 &0.9940 &0.12&[0.1,0.25] \\
 			\cline{2-11}
 &{70\%} &27.33&0.9862 &{8}  &27.70 &0.9873&{\bf27.97}&{\bf 0.9881}&0.15&[0.1,0.25]  \\
 			\cline{2-11}
 &{90\%} &16.07& 0.7543 &{16}  &17.24 & 0.8494  &{\bf17.32}&{\bf 0.8517}&0.15&[0.1,0.3]  \\
 \midrule\hline
 \cline{1-11}
 average & & {\bf29.57} &0.9601&  &{\bf29.57}  &0.9717  &29.52 &{\bf0.9756} && \\
 \bottomrule
 \end{tabular}
 	}
 \end{table}
    
 The last row of Tables \ref{tab1}-\ref{tab2} reports the average PSNR and SSIM yielded by two solvers for six images under all noise levels, which indicate that the average PSNR and SSIM yielded by iPMM-lbfgs for six images are better than those yielded by PLMA. We also see that under the $70\%$ and $90\%$ noise levels, ${\rm iPMM}_{\rm box}$ has some advantage over the PLMA; for example, the PSNRs returned by iPMM-lbfgs for Parrot image corrupted by the average blur and $70\%$ impluse noise and Cameraman image polluted by Gaussian blur and $90\%$ impluse noise are higher than that of PLMA with more ${\bf 3.5}$ db. In addition, the results of iPMM and ${\rm iPMM}_{\rm box}$ columns in Tables \ref{tab1}-\ref{tab2} show that model \eqref{Xmodel} with the box constraint is superior to the one without it, which also validates that the model with two regularizers are better than the model with a single one. It is worth pointing out that the results of ${\rm iPMM}_{\rm box}$ are all obtained by solving model \eqref{Xmodel} with the same $\nu$ for the same noise level, while the results of PLMA are obtained with different $\gamma$ even for the same noise level. From the above discussion, model \eqref{Xmodel} with the box constraint has better robustness, and iPMM-lbfgs are effective for solving it. 
 
 Figures \ref{noise1-grayiamge} and \ref{noise4-grayiamge} below illustrate the recovery results of two solvers for four images under the $30\%$ and $90\%$ noise levels. From Figure \ref{noise1-grayiamge}, for the $30\%$ noise level, there are almost no difference between the recovery results of two solvers. However, under the $90\%$ noise level, the second column of Figure \ref{noise4-grayiamge} indicates that iPMM and ${\rm iPMM}_{\rm box}$ can still recover the edge details, but the PLMA fails to do this. This shows that iPMM-lbfgs has superiority to recovering images with high impluse noise.        
  \begin{figure}[H] 
 	\centering  
 	\subfigure[Noise with $30\%$ ]{
 		\includegraphics[width=2.8cm]{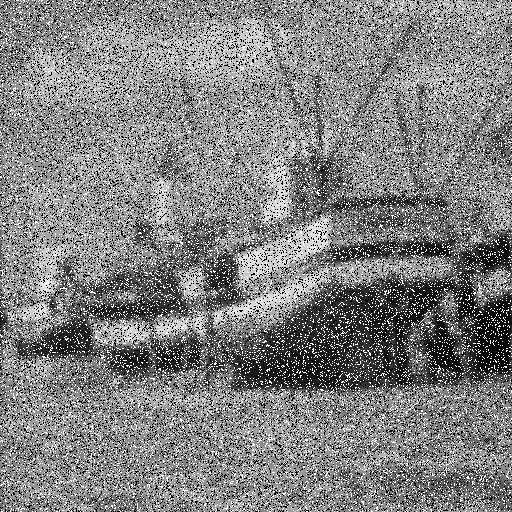}}\quad
 	\subfigure[Noise with $30\%$ ]{
 		\includegraphics[width=2.8cm]{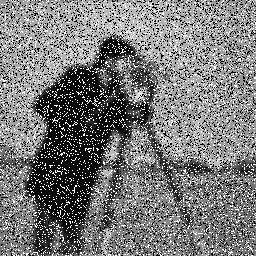}}\quad
 	\subfigure[Noise with $30\%$  ]{
 		\includegraphics[width=2.8cm]{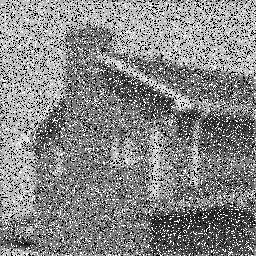}}\quad
 	\subfigure[Noise with $30\%$  ]{
 		\includegraphics[width=2.8cm]{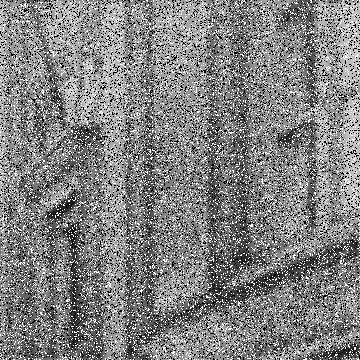}}\quad
 	\subfigure[PLMA ]{
 		\includegraphics[width=2.8cm]{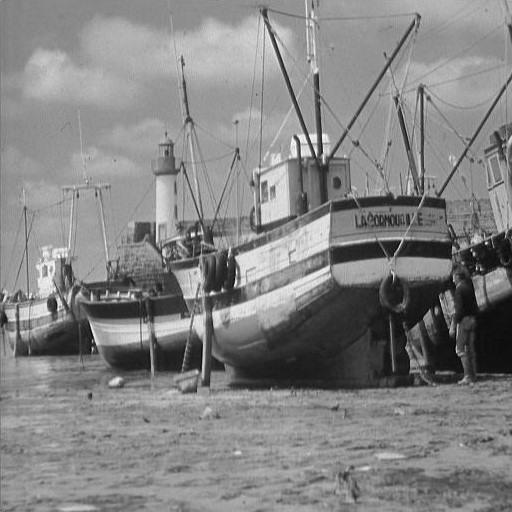}}\quad
 	\subfigure[PLMA ]{
 		\includegraphics[width=2.8cm]{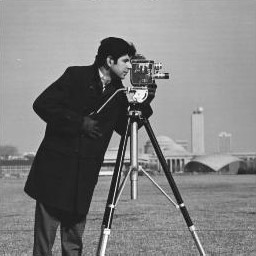}}\quad
 	\subfigure[PLMA ]{
 		\includegraphics[width=2.8cm]{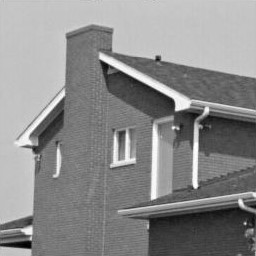}}\quad
 	\subfigure[PLMA ]{
 		\includegraphics[width=2.8cm]{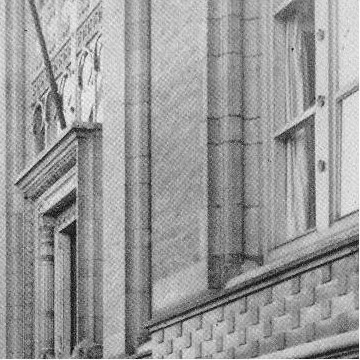}}
 	\quad
 	\subfigure[iPMM ]{
 		\includegraphics[width=2.8cm]{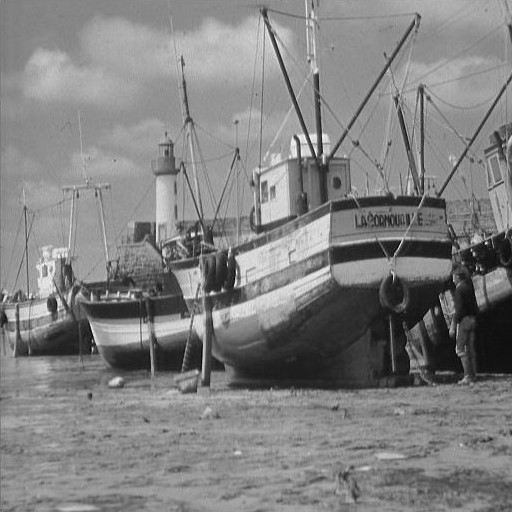}}\quad
 	\subfigure[iPMM ]{
 		\includegraphics[width=2.8cm]{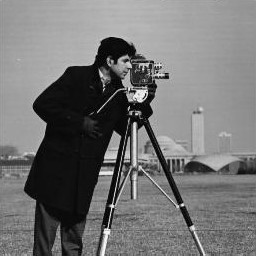}}\quad
 	\subfigure[iPMM]{
 		\includegraphics[width=2.8cm]{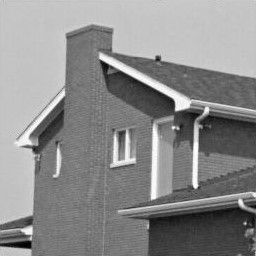}}\quad
 	\subfigure[iPMM]{
 		\includegraphics[width=2.8cm]{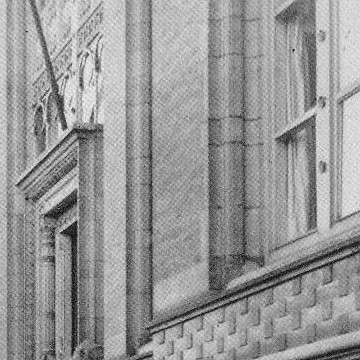}}
 	\quad
 	\subfigure[${\rm iPMM}_{\rm box}$]{
 		\includegraphics[width=2.8cm]{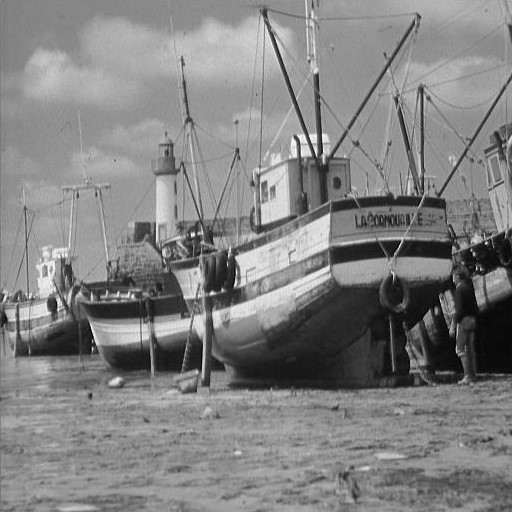}}\quad
 	\subfigure[${\rm iPMM}_{\rm box}$ ]{
 		\includegraphics[width=2.8cm]{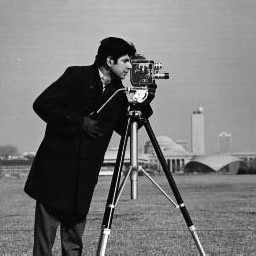}}\quad
 	\subfigure[${\rm iPMM}_{\rm box}$]{
 		\includegraphics[width=2.8cm]{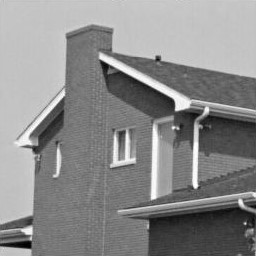}}\quad
 	\subfigure[${\rm iPMM}_{\rm box}$]{
 		\includegraphics[width=2.8cm]{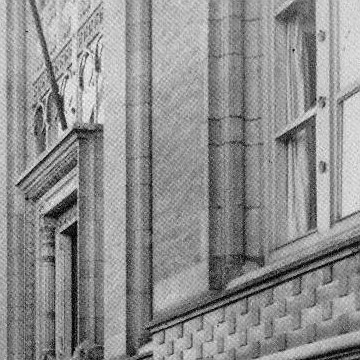}}
 	\caption{Recovery results on Boat, Cameraman, House and Building images corrupted by average blur and salt-and-pepper impluse noise with $30\%$ noise level} \label{noise1-grayiamge}	
 \end{figure}
 \begin{figure}[H] 
 	\centering  
 	\subfigure[Noise with $90\%$ ]{
 		\includegraphics[width=2.8cm]{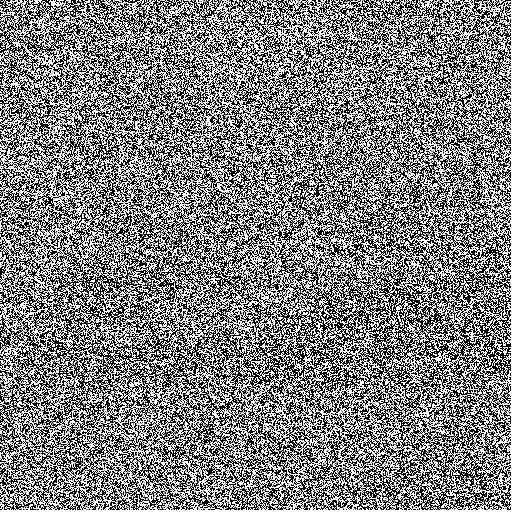}}\quad
 	\subfigure[Noise with $90\%$ ]{
 		\includegraphics[width=2.8cm]{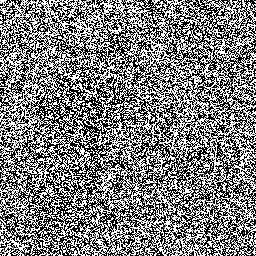}}\quad
 	\subfigure[Noise with $90\%$  ]{
 		\includegraphics[width=2.8cm]{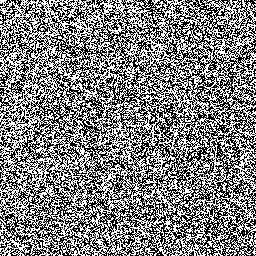}}\quad
 	\subfigure[Noise with $90\%$  ]{
 		\includegraphics[width=2.8cm]{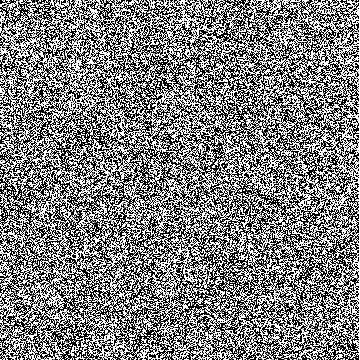}}\quad
 	\subfigure[PLMA ]{
 		\includegraphics[width=2.8cm]{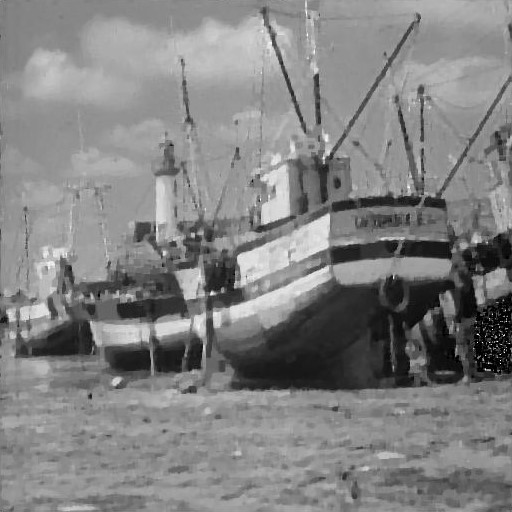}}\quad
 	\subfigure[PLMA ]{
 		\includegraphics[width=2.8cm]{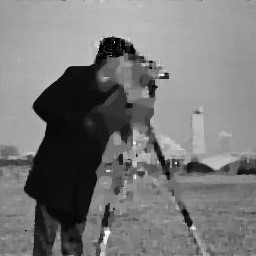}}\quad
 	\subfigure[PLMA ]{
 		\includegraphics[width=2.8cm]{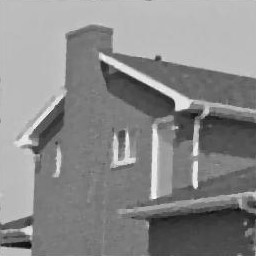}}\quad
 	\subfigure[PLMA ]{
 		\includegraphics[width=2.8cm]{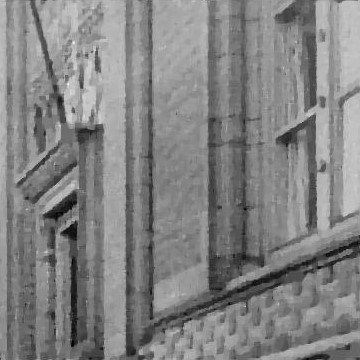}}
 	\quad
 	\subfigure[iPMM ]{
 		\includegraphics[width=2.8cm]{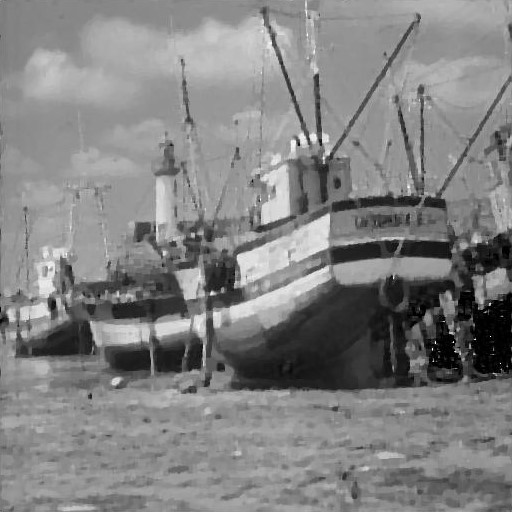}}\quad
 	\subfigure[iPMM ]{
 		\includegraphics[width=2.8cm]{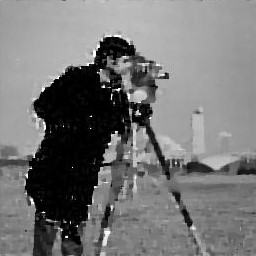}}\quad
 	\subfigure[iPMM]{
 		\includegraphics[width=2.8cm]{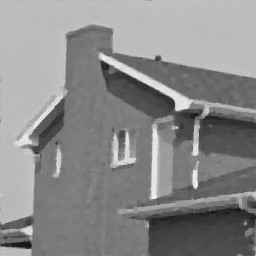}}\quad
 	\subfigure[iPMM]{
 		\includegraphics[width=2.8cm]{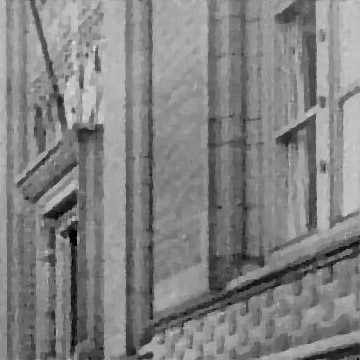}}
 	\quad
 	\subfigure[${\rm iPMM}_{\rm box}$]{
 		\includegraphics[width=2.8cm]{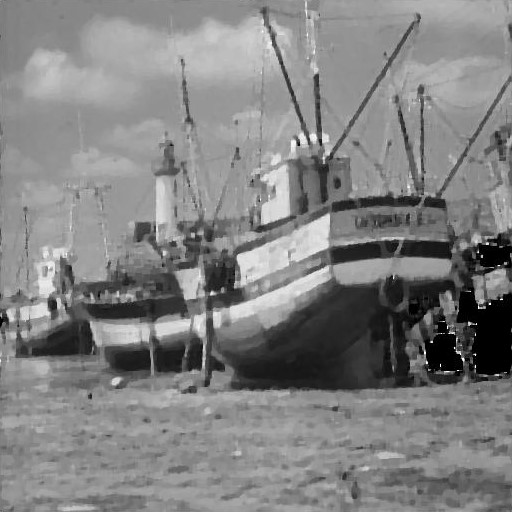}}\quad
 	\subfigure[${\rm iPMM}_{\rm box}$]{
 		\includegraphics[width=2.8cm]{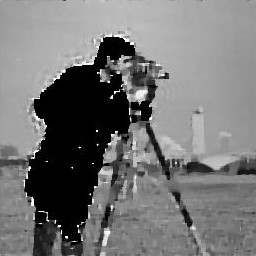}}\quad
 	\subfigure[${\rm iPMM}_{\rm box}$]{
 		\includegraphics[width=2.8cm]{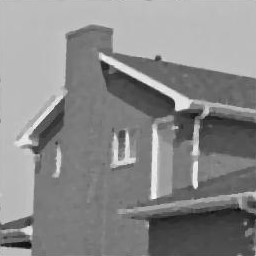}}\quad
 	\subfigure[${\rm iPMM}_{\rm box}$]{
 		\includegraphics[width=2.8cm]{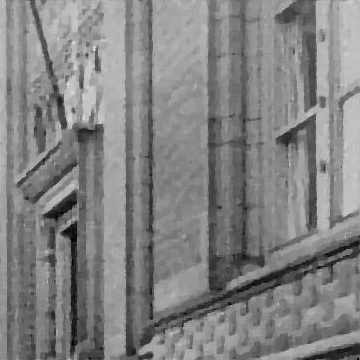}}
 	\caption{Recovery results on Boat, Cameraman, House and Building images corrupted by average blur and salt-and-pepper impluse noise with $90\%$ noise level} \label{noise4-grayiamge}	
 \end{figure}
 
 \subsection{Tests for image inpainting with impluse noise}\label{sec6.2} 
 
 We test the performance of iPMM-lbfgs to solve model \eqref{UVmodel} with $\vartheta$ being the function associated to $\theta_5$ for $\varepsilon=10^{-5}$ and $\psi$ being the TV function in \eqref{aTV-fun}. Since the color images have three channels, we apply this model to every channel, for which $\Omega=\Omega_i$ represents the index set covered by the mask in the $i$th channel of color images and $b=Y_i=\mathcal{P}_{\Omega_i}(\mathcal{Y}(:,:,i))$ for $i=1,2,3$. Here, the tensor $\mathcal{Y}$ denotes a color image corrputed by salt-and-pepper impulse noise. In our experiments, we consider two types of image masks: one is block mask that is composed of three black blocks of size $50\times 50$, and the other is text mask composed of a paragraph of black English text. The original color images include Airplane ($512\times512\times 3$), Barbara ($256\times256\times 3$), Butterfly ($256\times256\times 3$), Lena ($512\times512\times 3$), Peppers ($360\times360\times 3$), Sea ($256\times256\times 3$), which are shown in Figure \ref{ciamge} below. The subsequent tests involve three noise levels $0.0\%,10\%$ and $30\%$. 
 
 \begin{figure}[h] 
 	\centering  
 	\subfigure[Airplane]{
 		\includegraphics[width=2.8cm]{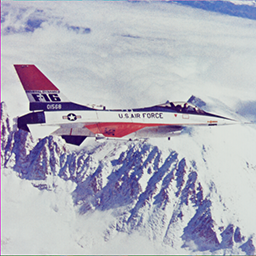}} \quad
 	\subfigure[Barbara]{
 		\includegraphics[width=2.8cm]{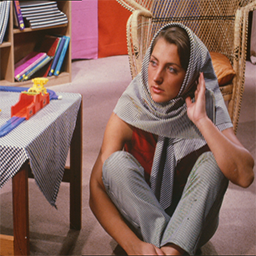}}\quad
 	\subfigure[House]{
 		\includegraphics[width=2.8cm]{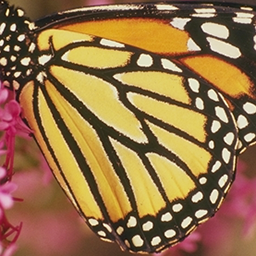}}\quad
 	\subfigure[Lena]{
 		\includegraphics[width=2.8cm]{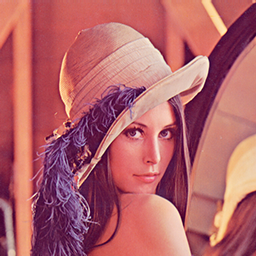}}\quad
 	\subfigure[Peppers]{
 		\includegraphics[width=2.8cm]{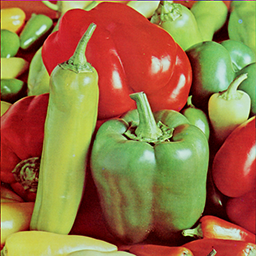}}\quad
 	\subfigure[Sea]{
 		\includegraphics[width=2.8cm]{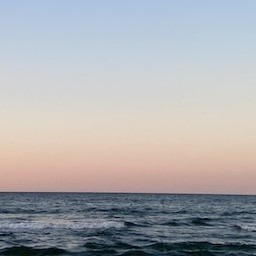}}\quad
 	\caption{The original color images used for testing image inpainting } \label{ciamge}
 \end{figure}
 
 \subsubsection{Influence of parameters $\nu$ and $\lambda$ on PSNR}\label{sec6.2.1}
 
 We take Airplane for example to check the influence of $\nu$ and $\lambda$ together on PSNR. Figure \ref{fig2} (a) and (c) plot the PSNR surfaces under the $30\%$ noise level as $(\nu,\lambda)$ varies in $[0.5,1]\times[10^{-3},10]$, and Figure \ref{fig2} (b) and (d) plot the PSNR surfaces under the noiseless scenario as $(\nu,\lambda)$ varies in $[0.1,0.5]\times[10^{-3},5]$. The PSNR surfaces in the first row of Figure \ref{fig2} are plotted for block mask, while those in the second row of Figure \ref{fig2} are plotted for text mask. We see that under the $30\%$ noise level, the PSNR attains the higher values for $(\nu,\lambda)\in[0.7,1.0]\times[10^{-3},1]$, while under the noiseless scenario, it has the higher values for $(\nu,\lambda)\in[0.2,0.4]\times[10^{-2},1]$. 
 Notably, the change range of $\lambda$ for better PSNR is far larger than that of $\nu$. This implies that the value of $\nu$ renders more influence on the PSNR than that of $\lambda$. In the subsequent experiments, we always take ${\bf\lambda=0.5}$, and report the interval of $\nu$ corresponding to better PSNR for each color image.   
  \begin{figure}[h]
 	\centering
 	\subfigure[\label{fig2a} noise level $30\%$]{\includegraphics[width=0.47\textwidth]{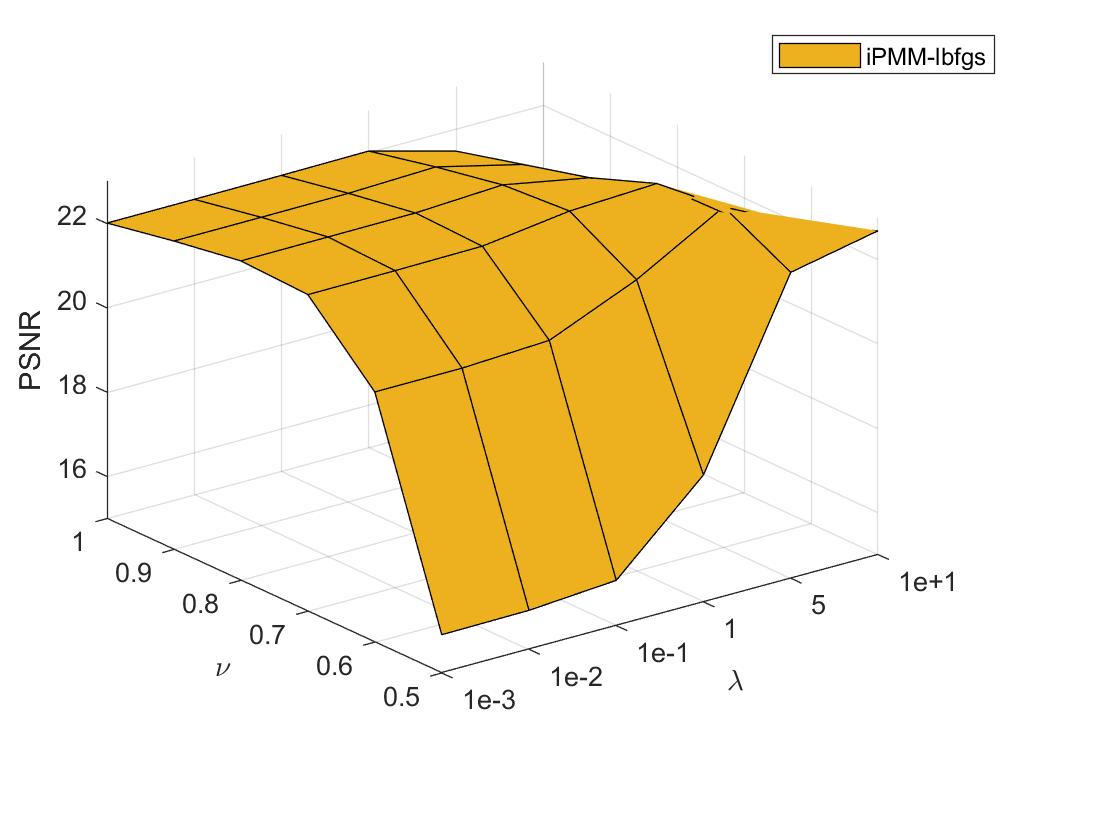}}
 	\subfigure[\label{fig2b} noise level $0.0\%$]{\includegraphics[width=0.47\textwidth]{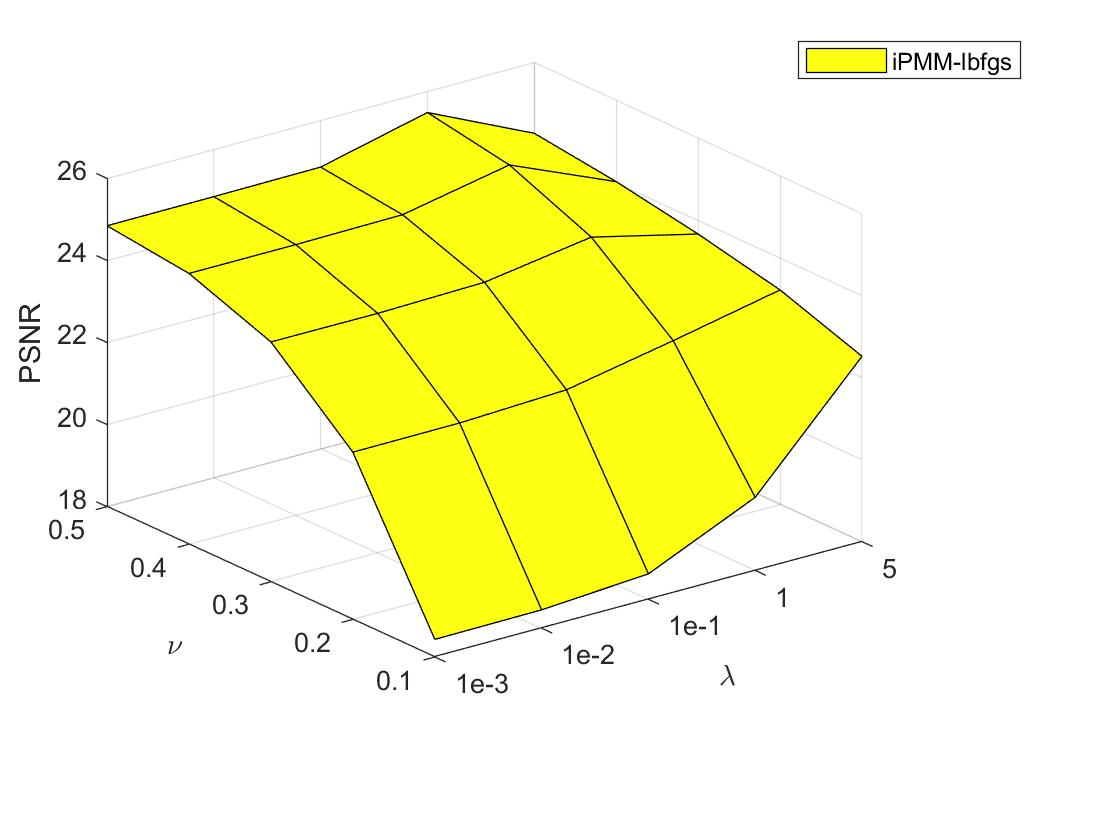}}
 	\subfigure[\label{fig3a} noise level $30\%$]{\includegraphics[width=0.47\textwidth]{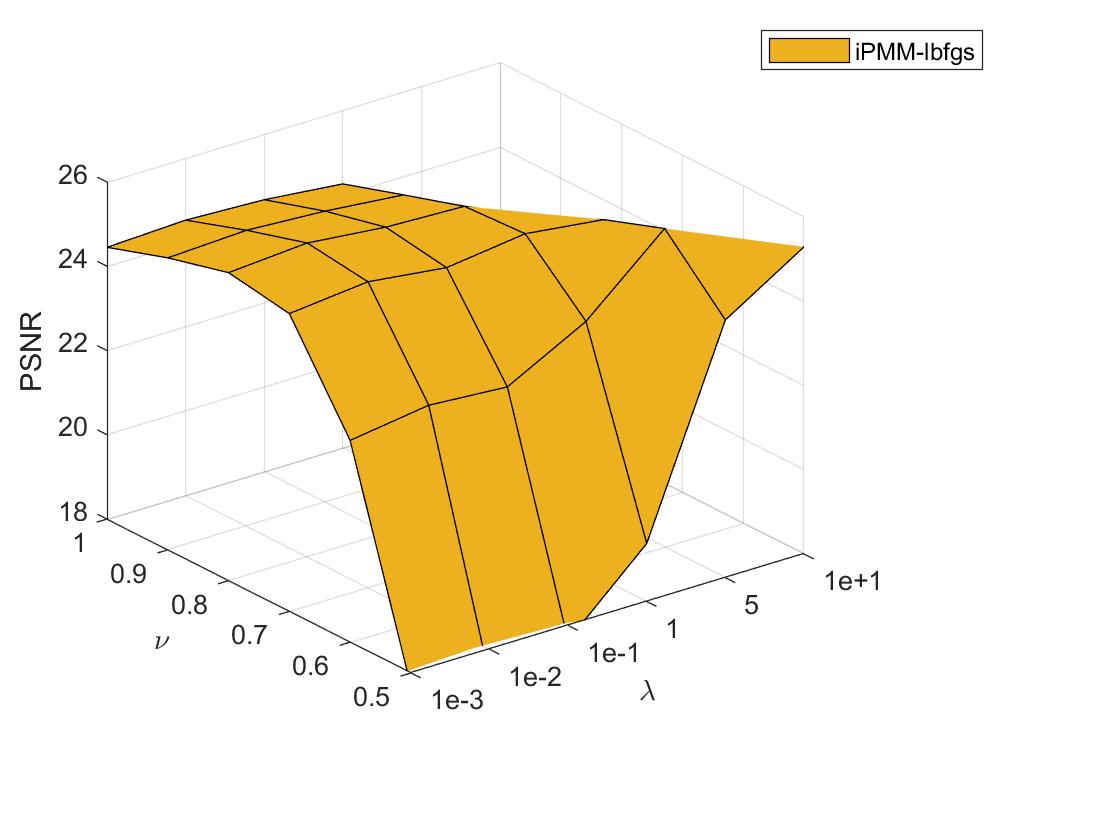}}
 	\subfigure[\label{fig3b} noise level $0.0\%$]{\includegraphics[width=0.47\textwidth]{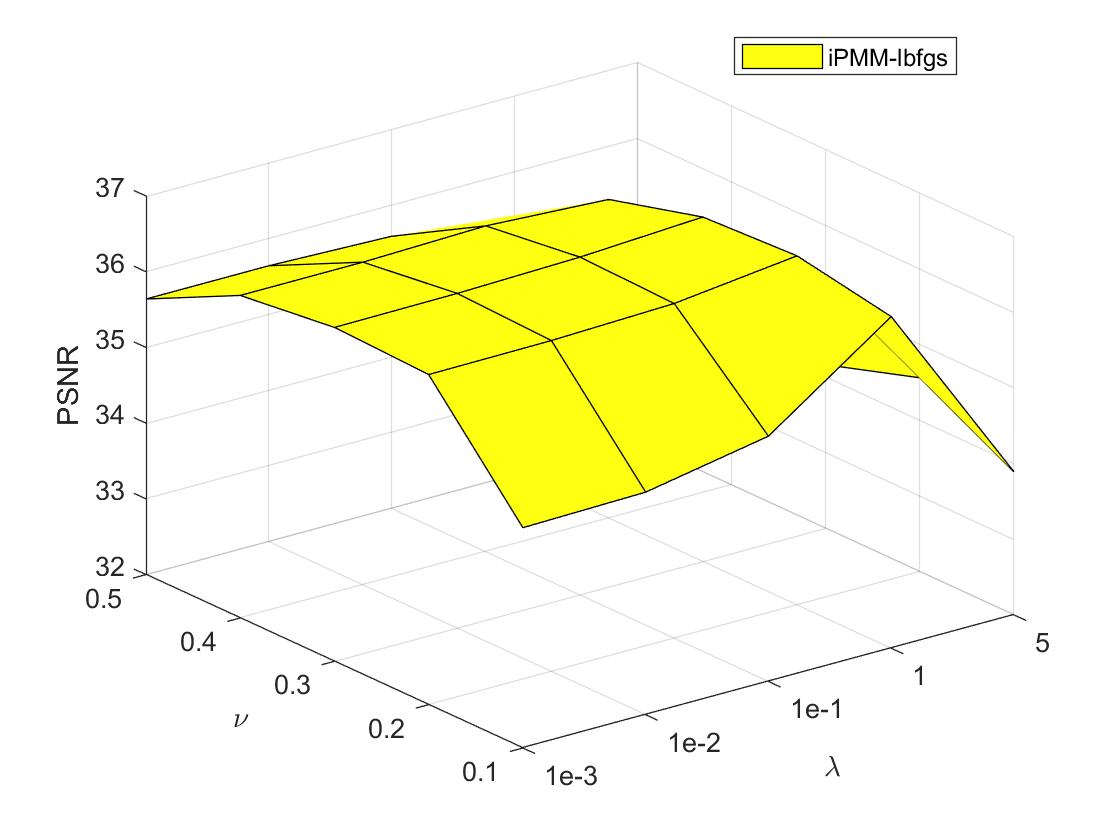}}
 	\setlength{\abovecaptionskip}{10pt}
 	\caption{Influence of $\lambda$ and $\nu$ on PSNR for block and text masks under two noise levels}
 	\label{fig2}
 \end{figure}
 \subsubsection{Implementation details of iPMM-lbfgs}\label{sec6.2.2}
 
 We first provide the choice of parameters for Algorithm \ref{iPMM} when it is applied to solve \eqref{UVmodel}. As mentioned in Section \ref{sec5.2},  $\mathcal{Q}_{k,j}=\gamma_{k,j}\mathcal{I}+\alpha_k\mathcal{C}_k^*\mathcal{C}_k$. We update $\alpha_k$ by \eqref{alphak-update} with $\alpha_0=0.1\underline{\gamma}$. The other parameters are chosen as in \eqref{other-para} except that $\underline{\gamma}=100$. For the parameter $\tau_k$ in \eqref{iterate-PPA}, we update it by  $\tau_{k+1}=\max\{\tau_k/1.2,\epsilon^*\}$ with $\tau_0=1.0$, where $\epsilon^*$ is the tolerance of the stopping condition. When applying model \eqref{UVmodel} to the $i$th channel, we choose the starting point $x^0=(U^0,V^0)$ with $U^0=P_1\Sigma^{1/2}$ and $V^0=Q_1\Sigma^{1/2}$, where $\Sigma={\rm diag}(\sigma_1(Y_i),\ldots,\sigma_r(Y_i))$ with $\sigma_1(Y_i)\ge\cdots\ge\sigma_r(Y_i)$, and $P_1$ and $Q_1$ are the matrices consisting of the first $r$ columns of $P$ and $Q$ with $Y_i=P{\rm diag}(\sigma(Y_i))Q^{\top}$. Here, $\sigma(Y_i)$ is the singular value vector of $Y_i$ arranged in a nonincreasing order, and ${\rm diag}(\sigma(Y_i))$ is the rectangular diagonal matrix with $\sigma(Y_i)$ being the diagonal vector.  
  We terminate Algorithm \ref{iPMM} under condition \eqref{stopping-cond} with $\epsilon^*=10^{-6}/\underline{\gamma}$. For the stopping criterion of LBFGS, let $\xi^{l}$ be the current iterate of LBFGS and let $x^{k,j}=\mathcal{P}_{\lambda\gamma_{k,j}^{-1}}h(x^k-\!\gamma_{k,j}^{-1}\mathcal{C}_k^*(\xi^l))$. By the weak duality theorem, $\Phi_k(\xi^{l})\le \Theta_{k,j}(\overline{x}^k)$. Hence, when either of the conditions
 \[
 \Theta_{k,j}(x^{k,j})<\Theta(x^k)\ \ {\rm and}\ \ \Theta_{k,j}(x^{k,j})-\Phi_k(\xi^{l})\le\frac{\mu_k}{2}\big[\Theta(x^k)-\Theta_{k,j}(x^{k,j})\big],
 \]
 we terminate LBFGS at the iterate $\xi^{l}$, and now $x^{k,j}$ is an inexact minimizer of \eqref{subprobk} satisfying the condition \eqref{inexct-cond}. For the subsequent tests, we choose the number of memory of LBFGS to be ${\bf10}$ and the maximum number of iterates to be ${\bf50}$.   
 
 We compare the performance of iPMM-lbfgs for solving model \eqref{UVmodel} with that of the ADMM proposed in \cite{Shang18} for solving the following model 
 \begin{align}\label{UVmodel1}
 &\min_{(U,V)\in\mathbb{X}_{r}\atop L,S\in\mathbb{R}^{m\times n}}\, \|\mathcal{P}_{\Omega}(S)\|_{\ell_{1/2}}^{1/2}+(\gamma/2)\big(\|U\|_*+\|V\|_*\big)\nonumber\\
 &\quad\ {\rm s.t.}\ \ UV^{\top}=L,\,L+S=D,
 \end{align}
 where $\|X\|_{\ell_{1/2}}$ is the $\ell_{1/2}$-norm of  $X\in\mathbb{R}^{m\times n}$, i.e., $\|X\|_{\ell_{1/2}}^{1/2}\!=\!\sum_{i=1}^m\sum_{j=1}^n|X_{ij}|^{1/2}$. More accurately, the ADMM is applied to the equivalent reformulation of \eqref{UVmodel1}:
 \begin{align}\label{EUVmodel1}
 &\min_{(U',V'),(U,V)\in\mathbb{X}_{r}\atop L,S\in\mathbb{R}^{m\times n}} \|\mathcal{P}_{\Omega}(S)\|_{\ell_{1/2}}^{1/2}+(\gamma/2)\big(\|U\|_*+\|V\|_*\big)\nonumber\\
 &\qquad\ {\rm s.t.}\ \ U'-U=0,\, V'-V=0,\,UV^{\top}=L,\,L+S=D.
 \end{align}
 Since the code of ADMM in \cite{Shang18} is unavailable, we implement it according to the iterate steps described in \cite[Algorithm 1]{Shang18} by ourselves. In the implementation of
 ADMM, we update the penalty parameter $\mu$ by the rule  $\mu_{k+1}=\min\{1.1\mu_k,10^{10}\}$ if ${\rm mod}(k,50)=0$, and the initial penalty parameter $\mu_0$ is chosen to be $10^{-3}$. Such an adjusting rule is basically the same as the one suggested in \cite{Shang18} for the noiseless scenario, but is better than the latter for the noisy case. 
  
 Note that models \eqref{UVmodel} and \eqref{UVmodel1} both depend on the positive integer $r$. Our  preliminary tests indicate that for fixed $\lambda$ and $\nu$, the PSNR is better for the larger $r$. In view of this, we take $r=\min\{m,n\}$ for the subsequent experiments.
 
 \subsubsection{Numerical results for image inpainting}\label{sec6.2.3}
 
 We report the numerical results of iPMM-lbfgs for solving model \eqref{UVmodel} and those of ADMM for solving model \eqref{UVmodel1} in Tables \ref{tab3}-\ref{tab4}.  All results in Tables \ref{tab3}-\ref{tab4} are the average of those obtained for ${\bf 5}$ test problems generated randomly  and the best results are shown in boldface. The average results of six images under all noise levels are listed in the last row. The value of parameter $\nu$ for model \eqref{UVmodel} is listed in $\nu$ column, and the interval in $\nu_{\rm box}$ column is the range of $\nu$ with $\lambda=0.5$ for better PSNR. During the testing, we observe that the quality of solutions returned by ADMM depends heavily on the parameter $\gamma$ in model \eqref{UVmodel1}, so we report its value when using ADMM to solve \eqref{UVmodel1}. 
 \begin{table}[h]
 	\caption{PSNR and SSIM for color images corrupted by block missing and impulse noise}\label{tab3}%
 	\centering
 	\renewcommand\arraystretch{1}
 	\scalebox{1}[1]{
 		\begin{tabular}{@{}cc|ccc|cccc@{}}
 			\toprule
 			Image &  Noise level  &  & ADMMLp  &   & & \multicolumn{2}{c}{iPMM-lbfgs}&  \\
 			\midrule\midrule
 			\multirow{4}{*}{Airplane} 
 			&   &   PSNR  & SSIM  & $\gamma$  & PSNR & SSIM & $\nu$ & $\nu_{\rm box}$  \\
 			\cline{3-9}
 			&{0\%}  &23.20 &0.8953 & {15}  &{\bf25.28}&{\bf 0.9282}&0.5&[0.4,0.6]  \\
 			\cline{2-9}
 			&{10\%}  &20.90&0.7937 &{80}  &{\bf24.42}&{\bf0.9068}&0.6&[0.5,0.8]\\
 			\cline{2-9}
 			&{30\%}  &17.55  &0.6385&{190}  &{\bf22.67} & {\bf0.8494}&0.8& [0.65,1.0] \\
 			\hline\hline
 			\multirow{3}{*}{Barbara} 
 			&{0\%}  &22.27  &0.9023&{15} &{\bf25.16}&{\bf0.9391}&0.5&[0.3,0.6]  \\
 			\cline{2-9}
 			&{10\%}  & 20.90  &0.7937&{90} &{\bf23.47}& {\bf0.9177}&0.6&[0.5,0.8]\\
 			\cline{2-9}
 			&{30\%} &18.52 &0.6705&{190} &{\bf22.28}&{\bf0.8697}&0.8&[0.7,1.0]\\
 			\midrule\hline
 			\multirow{3}{*}{Butterfly} 
 			&{0\%}  &19.37  &0.9106&{15} &{\bf 21.00}&{\bf0.9453}&0.5&[0.5,0.7]\\
 			\cline{2-9}
 			&{10\%}  & 18.57 &0.8634 &{70} &{\bf20.58} &{\bf0.9345}&0.6&[0.55,0.7]\\
 			\cline{2-9}
 			&{30\%}  &16.16 &0.7465&{120} &{\bf19.58}&{\bf0.9063}&0.8&[0.7,1.0]  \\
 			\midrule\hline
 			\multirow{3}{*}{Lena} 
 			&{0\%}  &23.32  &0.9590 &{15} &{\bf 25.37} &{\bf0.9790}&0.5&[0.4,0.7] \\
 			\cline{2-9}
 			&{10\%}  &21.83&0.9213&{60}  &{\bf24.69}&{\bf0.9717}&0.6&[0.5,0.9] \\
 			\cline{2-9}
 			&{30\%} &19.89&0.8831  &{120} &{\bf23.56}& {\bf0.9553}&0.8&[0.65,0.95]\\
 			\hline\hline
 			\multirow{3}{*}{Peppers} 
 			&{0\%}  &21.34 &0.9372 &{15}  &{\bf22.86}&{\bf0.9537} &0.5&[0.3,1.0]  \\
 			\cline{2-9}
 			&{10\%}  &20.24 &0.9079&{70} &{\bf22.52}&{\bf0.9493}&0.7&[0.5,0.9] \\
 			\cline{2-9}
 			&{30\%}  &18.62&0.8632 &{120} &{\bf21.82}&{\bf0.9371}&0.8&[0.65,1.0]\\
 			\midrule\hline
 			\multirow{3}{*}{Sea} 
 			&{0\%}   &{\bf33.15}&{\bf0.9750} &{30} &32.22& 0.9756&0.5&[0.5,0.9] \\
 			\cline{2-9}
 			&{10\%}  &29.28 &0.9340&{70}  &{\bf31.73}&{\bf0.9679 }&0.6&[0.55,0.9]\\
 			\cline{2-9}
 			&{30\%} &26.27&0.8952 &{130} &{\bf30.05}& {\bf0.9448}&0.8&[0.75,1.0]  \\
 			\midrule\hline
 		average	&   &21.74&0.8586 & &{\bf24.40}& {\bf0.9351}&& \\
 			\midrule		
 		\end{tabular} 
 	}
 \end{table}
\begin{table}[h]
	\caption{PSNR and SSIM for color images corrupted by text missing and impulse noise}\label{tab4}%
	\centering
	\renewcommand\arraystretch{1}
	\scalebox{1}[1]{
		\begin{tabular}{@{}cc|ccc|cccc@{}}
			\toprule
			Image &  Noise level  &  & ADMMLp  &   & & \multicolumn{2}{c}{iPMM-lbfgs}&  \\
			\midrule\midrule
			\multirow{4}{*}{Airplane} 
			&   &   PSNR  & SSIM  & $\gamma$  & PSNR & SSIM & $\nu$ & $\nu_{\rm box}$  \\
	
			&{0\%}   &31.39 &0.9232 &{15} &{\bf36.55}&{\bf0.9915}&0.3&[0.1,0.45]\\
			\cline{2-9}
			&{10\%}  &23.99  &0.8539&{70} &{\bf30.50}&{\bf0.9561}&0.6&[0.55,0.7] \\
			\cline{2-9}
			&{30\%}  &19.91 &0.7166 &{120} &{\bf25.34}&{\bf0.8789}&0.8&[0.7,0.9]  \\
			\hline\hline
			\multirow{3}{*}{Barbara} 
			&{0\%}   &30.72 &0.9507 &{20} & {\bf36.15}&{\bf0.9903}&0.3&[0.15,0.4]\\
			\cline{2-9}
			&{10\%}  &25.27 &0.8796 &{70} &{\bf30.08}&{\bf0.9632}&0.6&[0.55,0.7]\\
			\cline{2-9}
			&{30\%}  &21.13  &0.7527 &{120} &{\bf25.86}&{\bf0.9089}&0.8&[0.7,0.95]\\
			\midrule\hline
			\multirow{3}{*}{Butterfly} 
			&{0\%}   &29.42  &0.9789&{20} &{\bf 33.67}&{\bf0.9949}&0.3&[0.1,0.55]  \\
			\cline{2-9}
			&{10\%}  &22.52 &0.9194 &{70}  &{\bf28.66}&{\bf0.9815}&0.6&[0.55,0.7]\\
			\cline{2-9}
			&{30\%}  &17.16 &0.7762 &{120} &{\bf23.77}&{\bf0.9459} &0.8&[0.65,1.0] \\
			\midrule\hline
			\multirow{3}{*}{Lena} 
			&{0\%}  &31.31 &0.9860  &{15} &{\bf36.70}&{\bf0.9964}&0.3&[0.35,0.8] \\
			\cline{2-9}
			&{10\%}  &24.49 &0.9405&{60}  &{\bf31.37}&{\bf0.9876}&0.6&[0.5,0.9] \\
			\cline{2-9}
			&{30\%}  &20.71 &0.8941&{120} &{\bf27.14}&{\bf0.9681}&0.8&[0.65,1.0]\\
			\hline\hline
			\multirow{3}{*}{Peppers} 
			&{0\%}  &31.81 & 0.9889 &{15} &{\bf38.11}&{\bf0.9977}&0.3&[0.2,0.45]  \\
			\cline{2-9}
			&{10\%} &24.10&0.9454  &60 &{\bf32.25}&{\bf0.9915 }&0.6&[0.55,0.7] \\
			\cline{2-9}
			&{30\%}  &20.48 &0.8960  &{120} &{\bf27.31}&{\bf0.9748}&0.8&[0.7,0.9]  \\
			\midrule\hline
			\multirow{3}{*}{Sea} 
			&{0\%}  &38.50 &0.9908&{20} &{\bf42.63}& {\bf0.9969}&0.3&[0.15,0.55] \\
			\cline{2-9}
			&{10\%}  &30.09 &0.9444 &{80} &{\bf38.80}&{\bf0.9824}&0.6&[0.6,0.8]\\
			\cline{2-9}
			&{30\%} &25.91&0.8893 &{140} &{\bf31.90}&{\bf0.9467}&0.8&[0.8,1.0]\\
			\midrule\hline
		\cline{1-9}
		average&   &26.05&0.9015 & &{\bf32.04}& {\bf0.9696}&& \\
			\midrule
		\end{tabular}
	}
\end{table}

 From Tables \ref{tab3}-\ref{tab4}, we see that the PSNRs yielded by iPMM-lbfgs for solving model \eqref{UVmodel} are much better than those yielded by ADMM for solving model \eqref{UVmodel1} with carefully selected $\gamma$. Under noiseless scenario, the PSNRs of the former are higher than those of the latter with more ${\bf 2}$ db; while under the $30\%$ noise level, the PSNRs of the former are generally higher than those of the latter with at least ${\bf 4}$ db. This not only confirms the TV plus low rank regularization term is better than the single low rank regularization term in image inpainting, but also validates the efficiency of iPMM-lbfgs for solving the nonconvex and nonsmooth model \eqref{UVmodel}. Figures \ref{block-coloriamge}-\ref{text-coloriamge} illustrate the recovery results of two solvers for four images. The last two rows of Figure \ref{block-coloriamge} show that the images recovered by the two solvers still involve block marks because this typle of inpainting tasks is very difficult, while the last two rows of Figure \ref{text-coloriamge} demonstrate that the four images recovered by iPMM-lbfgs have clear edge details, but the edge details of those images recovered by the ADMM-L1/2 are blurring.       
\begin{figure}[h] 
	\centering  
	\subfigure[Noise with $0\%$ ]{
		\includegraphics[width=2.8cm]{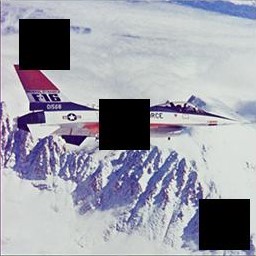}}\quad
	\subfigure[Noise with $0\%$ ]{
		\includegraphics[width=2.8cm]{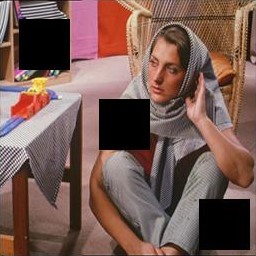}}\quad
	\subfigure[Noise with $0\%$  ]{
		\includegraphics[width=2.8cm]{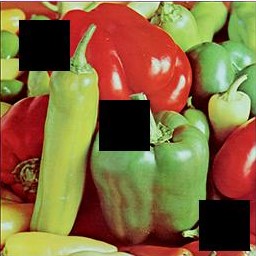}}\quad
	\subfigure[Noise with $0\%$  ]{
		\includegraphics[width=2.8cm]{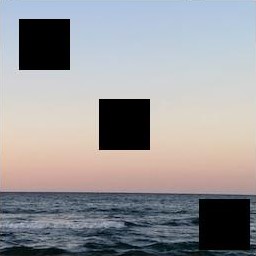}}\quad
	\subfigure[ADMM-L1/2]{
		\includegraphics[width=2.8cm]{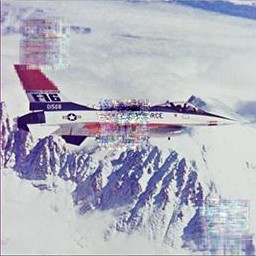}}\quad
	\subfigure[ADMM-L1/2 ]{
		\includegraphics[width=2.8cm]{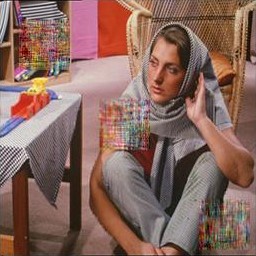}}\quad
	\subfigure[ADMM-L1/2 ]{
		\includegraphics[width=2.8cm]{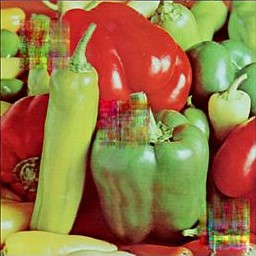}}\quad
	\subfigure[ADMM-L1/2 ]{
		\includegraphics[width=2.8cm]{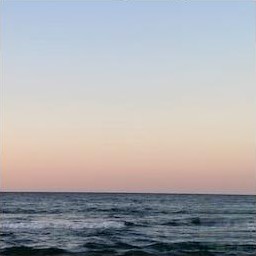}}
	\quad
	\subfigure[iPMM-lbfgs ]{
		\includegraphics[width=2.8cm]{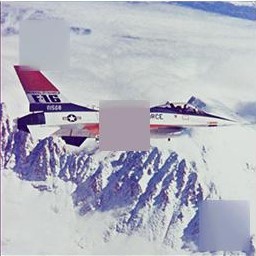}}\quad
	\subfigure[iPMM-lbfgs ]{
		\includegraphics[width=2.8cm]{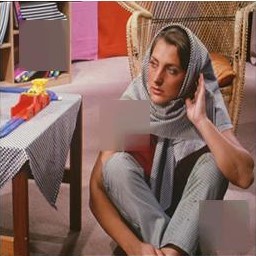}}\quad
	\subfigure[iPMM-lbfgs]{
		\includegraphics[width=2.8cm]{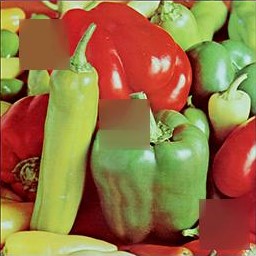}}\quad
	\subfigure[iPMM-lbfgs]{
		\includegraphics[width=2.8cm]{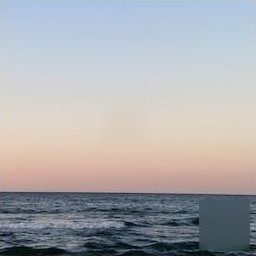}}
	\caption{Recovery results on Airplane, Barbara, Peppers and Sea images with block mask} \label{block-coloriamge}	
\end{figure}
\begin{figure}[h] 
	\centering  
	\subfigure[Noise with $30\%$ ]{
		\includegraphics[width=2.8cm]{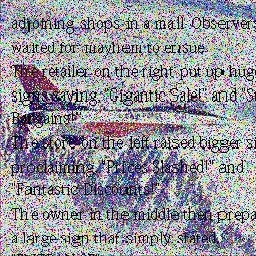}}\quad
	\subfigure[Noise with $30\%$ ]{
		\includegraphics[width=2.8cm]{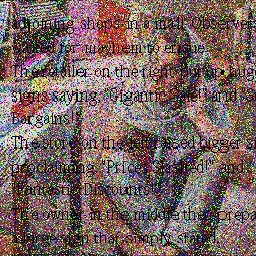}}\quad
	\subfigure[Noise with $30\%$  ]{
		\includegraphics[width=2.8cm]{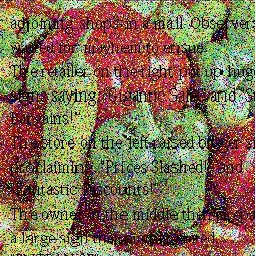}}\quad
	\subfigure[Noise with $30\%$  ]{
		\includegraphics[width=2.8cm]{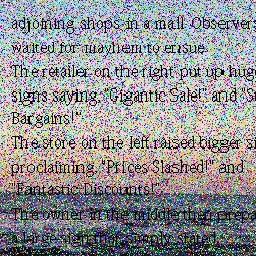}}\quad
	\subfigure[ADMM-L1/2]{
		\includegraphics[width=2.8cm]{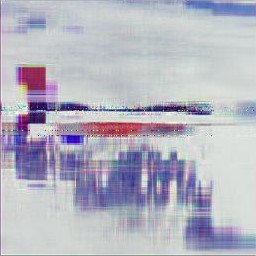}}\quad
	\subfigure[ADMM-L1/2 ]{
		\includegraphics[width=2.8cm]{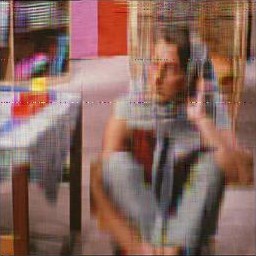}}\quad
	\subfigure[ADMM-L1/2 ]{
		\includegraphics[width=2.8cm]{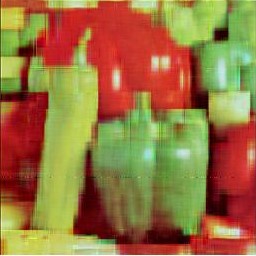}}\quad
	\subfigure[ADMM-L1/2 ]{
		\includegraphics[width=2.8cm]{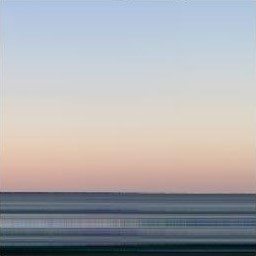}}
	\quad
	\subfigure[iPMM-lbfgs ]{
		\includegraphics[width=2.8cm]{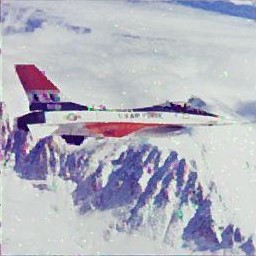}}\quad
	\subfigure[iPMM-lbfgs ]{
		\includegraphics[width=2.8cm]{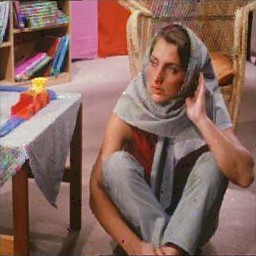}}\quad
	\subfigure[iPMM-lbfgs]{
		\includegraphics[width=2.8cm]{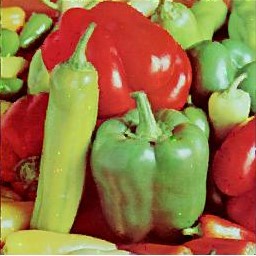}}\quad
	\subfigure[iPMM-lbfgs]{
		\includegraphics[width=2.8cm]{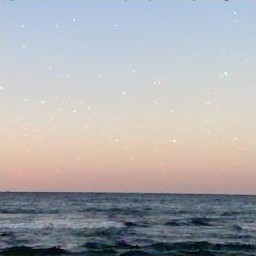}}
	
	\caption{Recovery results on Airplane, Barbara, Peppers and Sea images corrupted by text mask and salt-and-pepper impluse noise with noise level $30\%$} \label{text-coloriamge}	
\end{figure} 
\section{Conclusion}\label{sec7}

 We proposed a general composite optimization model for image reconstruction with impluse noise, which involves a family of nonconvex functions as a data-fidelity term to promote sparsity and two different regularization terms to enhance the quality of recovered images. For this class of nonconvex and nonsmooth optimization problems, we develop an inexact proximal MM method by  constructing the majorization of its objective function at the current iterate and then seeking an inexact minimizer of the majorization function in each step, and establish the convergence of the generated iterate sequence under the KL property of the potential function $\Xi$. We apply the developed method to image deblurring and inpainting tasks, and conduct extensive numerical experiments on real-world images. Numerical comparions show that our method has better robustness, and it is better than the PLMA \cite{Zhang17} in terms of PSNR and SSIM under the high-noise scenario, and is significantly superior to the ADMMLp \cite{Shang18} in terms of PSNR and SSIM. 


%
%


\bibliographystyle{siamplain}
\bibliography{references}
\end{document}